\documentclass[reqno,11pt,letterpaper]{amsart}

\usepackage{amsmath}
\usepackage{amssymb}
\usepackage{amsthm}
\usepackage{mathrsfs}
\usepackage{accents}
\usepackage{calc}
\usepackage{arydshln}
\usepackage{upgreek}
\usepackage{slashed}
\usepackage{xifthen}
\usepackage{graphicx}
\usepackage{longtable}
\usepackage[inline]{enumitem}

\usepackage{xcolor}
\definecolor{winered}{rgb}{0.6,0,0}
\definecolor{lessblue}{rgb}{0,0,0.7}

\usepackage[pdftex,colorlinks=true,linkcolor=winered,citecolor=lessblue,urlcolor=lessblue,breaklinks=true,bookmarksopen=true]{hyperref}

\makeatletter
\newcommand{\myitem}[3]{\item[#2]\def\@currentlabel{#3}\label{#1}}
\makeatother

\usepackage{titletoc}

\makeatletter

\def\@tocline#1#2#3#4#5#6#7{
\begingroup
  \par
    \parindent\z@ \leftskip#3 \relax \advance\leftskip\@tempdima\relax
                  \rightskip\@pnumwidth plus 4em \parfillskip-\@pnumwidth
    \ifcase #1 
       \vskip 0.6em \hskip 0em 
       \or
       \or \hskip 0em 
       \or \hskip 1em 
    \fi%
    #6
    \nobreak\relax{\leavevmode\leaders\hbox{\,.}\hfill}
    \hbox to\@pnumwidth {\@tocpagenum{#7}}
  \par
\endgroup
}

 \def\l@section{\@tocline{0}{0pt}{0pc}{}{}}

\renewcommand{\tocsection}[3]{%
  \indentlabel{\@ifnotempty{#2}{ 
    \ignorespaces\bfseries{#2. #3}}}
  \indentlabel{\@ifempty{#2}{\ignorespaces\bfseries{#3}}{}} 
    \vspace{1.5pt}}

\renewcommand{\tocsubsection}[3]{%
  \indentlabel{\@ifnotempty{#2}{
    \ignorespaces#2. #3}}
  \indentlabel{\@ifempty{#2}{\ignorespaces #3}{}}
    \vspace{1.5pt}}

\renewcommand{\tocsubsubsection}[3]{%
  \indentlabel{\@ifnotempty{#2}{
    \ignorespaces#2. #3}}
  \indentlabel{\@ifempty{#2}{\ignorespaces #3}{}}
    \vspace{1.5pt}}

\makeatother

\makeatletter
\def\@nomenstarted{0}
\newlength{\@nomenoldtabcolsep}

\newcommand{\nomenstart}
  {%
    \def\@nomenstarted{1}%
    \setlength{\@nomenoldtabcolsep}{\tabcolsep}%
    \setlength{\tabcolsep}{3.5pt}%
    \begin{longtable}{p{0.11\textwidth} p{0.86\textwidth}}
  }

\newcommand{\nomenitem}[2]{%
    \ifcase\@nomenstarted%
      \or 
      \or \\ 
    \fi%
    #1\,{\leavevmode\leaders\hbox{\,.}\hfill} & #2%
    \def\@nomenstarted{2}%
  }%
\newcommand{\nomenend}
  {\\%
      \end{longtable}%
      \setlength{\tabcolsep}{\@nomenoldtabcolsep}%
      \def\@nomenstarted{0}%
  }
\makeatother

\makeatletter
\newcommand{\vast}{\bBigg@{4}}
\newcommand{\Vast}{\bBigg@{5}}
\newcommand{\VAST}[1]{\bBigg@{#1}}
\makeatother

\allowdisplaybreaks

\numberwithin{equation}{section}
\numberwithin{figure}{section}
\newtheorem{thm}{Theorem}[section]

\newtheorem{prop}[thm]{Proposition}
\newtheorem{lemma}[thm]{Lemma}
\newtheorem{cor}[thm]{Corollary}
\newtheorem*{thm*}{Theorem}
\newtheorem*{prop*}{Proposition}
\newtheorem*{cor*}{Corollary}
\newtheorem*{conj*}{Conjecture}

\theoremstyle{definition}
\newtheorem{definition}[thm]{Definition}
\newtheorem{notation}[thm]{Notation}

\theoremstyle{remark}
\newtheorem{rmk}[thm]{Remark}

\makeatletter
\newcommand{\fakephantomsection}{%
  \Hy@MakeCurrentHref{\@currenvir.\the\Hy@linkcounter}
  \Hy@raisedlink{\hyper@anchorstart{\@currentHref}\hyper@anchorend}%
  \Hy@GlobalStepCount\Hy@linkcounter%
}
\makeatother

\newcommand{\mc}{\mathcal}
\newcommand{\cA}{\mc A}

\newcommand{\cC}{\mc C}
\newcommand{\cD}{\mc D}

\newcommand{\cF}{\mc F}

\newcommand{\cL}{\mc L}

\newcommand{\cO}{\mc O}

\newcommand{\cV}{\mc V}

\newcommand{\cX}{\mc X}

\newcommand{\ms}{\mathscr}

\newcommand{\sC}{\ms C}
\newcommand{\sD}{\ms D}

\newcommand{\sF}{\ms F}
\newcommand{\sG}{\ms G}

\newcommand{\scri}{\ms I}

\newcommand{\sL}{\ms L}

\newcommand{\sR}{\ms R}

\newcommand{\sY}{\ms Y}

\newcommand{\N}{\mathbb{N}}
\newcommand{\R}{\mathbb{R}}

\newcommand{\Sph}{\mathbb{S}}

\newcommand{\sfG}{\mathsf{G}}

\newcommand{\fc}{\mathfrak{c}}

\newcommand{\fm}{\mathfrak{m}}

\newcommand{\sld}{\slashed{\dd}{}}
\newcommand{\slg}{\slashed{g}{}}
\newcommand{\slh}{\slashed{h}{}}

\newcommand{\slu}{\slashed{u}{}}
\newcommand{\slA}{\slashed{A}{}}
\newcommand{\slB}{\slashed{B}{}}

\newcommand{\slGamma}{\slashed{\Gamma}{}}

\newcommand{\sldelta}{\slashed{\delta}{}}
\newcommand{\slDelta}{\slashed{\Delta}{}}
\newcommand{\slnabla}{\slashed{\nabla}{}}
\newcommand{\slpi}{\slashed{\pi}{}}

\newcommand{\slomega}{\slashed{\omega}{}}
\newcommand{\sltr}{\operatorname{\slashed\tr}}

\newcommand{\ran}{\operatorname{ran}}

\newcommand{\Hom}{\operatorname{Hom}}
\renewcommand{\Re}{\operatorname{Re}}

\newcommand{\tr}{\operatorname{tr}}

\newcommand{\dv}{\operatorname{div}}

\newcommand{\cd}{\fc}

\newcommand{\Ups}{\Upsilon}

\newcommand{\eps}{\epsilon}

\newcommand{\hra}{\hookrightarrow}
\newcommand{\la}{\langle}

\newcommand{\ol}{\overline}
\newcommand{\pa}{\partial}
\newcommand{\dd}{{\mathrm d}}
\newcommand{\ad}{{\mathrm{ad}}}
\newcommand{\ra}{\rangle}

\newcommand{\wh}{\widehat}
\newcommand{\wt}{\widetilde}

\newcommand{\ubar}[1]{\underaccent{\bar}#1}
\newcommand{\pfstep}[1]{$\bullet$\ \underline{\textit{#1}}}

\newcommand{\bop}{{\mathrm{b}}}

\newcommand{\scl}{{\mathrm{sc}}}
\newcommand{\ebop}{{\mathrm{e,b}}}

\newcommand{\cp}{{\mathrm{c}}}

\newcommand{\Diff}{\mathrm{Diff}}

\newcommand{\Vb}{\cV_\bop}
\newcommand{\Diffb}{\Diff_\bop}

\newcommand{\Veb}{\cV_\ebop}

\newcommand{\Diffeb}{\Diff_\ebop}

\newcommand{\Tb}{{}^{\bop}T}
\newcommand{\Tsc}{{}^{\scl}T}

\newcommand{\Teb}{{}^{\ebop}T}

\newcommand{\half}{{\tfrac{1}{2}}}

\newcommand{\loc}{{\mathrm{loc}}}
\newcommand{\CI}{\cC^\infty}

\newcommand{\CIc}{\cC^\infty_\cp}

\newcommand{\Hb}{H_{\bop}}

\newcommand{\Heb}{H_{\ebop}}

\newcommand{\Ric}{\mathrm{Ric}}

\newcommand{\bhm}{\fm}

\newcommand{\openbigpmatrix}[1]
  {%
    \def\@bigpmatrixsize{#1}%
    \addtolength{\arraycolsep}{-#1}%
    \begin{pmatrix}%
  }
\newcommand{\closebigpmatrix}
  {%
    \end{pmatrix}%
    \addtolength{\arraycolsep}{\@bigpmatrixsize}%
  }

\newlength{\enummargin}

\newcommand{\usref}[1]{{\upshape\ref{#1}}}

\DeclareGraphicsExtensions{.mps}

\makeatletter
\newcommand*{\fwbw}[1]{\expandafter\@fwbw\csname c@#1\endcsname}
\newcommand*{\@fwbw}[1]{\ifcase #1 \or {\rm fw}\or {\rm bw}\fi}
\AddEnumerateCounter{\fwbw}{\@fwbw}
\makeatother

\begin{document}

\title[Exterior stability]{Exterior stability of Minkowski space in generalized harmonic gauge}

\date{\today}

\subjclass[2010]{Primary 35B40, Secondary 83C05, 35L05}
\keywords{Nonlinear stability, gauge conditions, edge metrics}

\author{Peter Hintz}
\address{Department of Mathematics, ETH Z\"urich, R\"amistrasse 101, 8092 Z\"urich, Switzerland}
\email{peter.hintz@math.ethz.ch}

\begin{abstract}
  We give a short proof of the existence of a small piece of null infinity for $(3+1)$-dimensional spacetimes evolving from asymptotically flat initial data as solutions of the Einstein vacuum equations. We introduce a modification of the standard wave coordinate gauge in which all non-physical metric degrees of freedom have strong decay at null infinity. Using a formulation of the gauge-fixed Einstein vacuum equations which implements constraint damping, we establish this strong decay regardless of the validity of the constraint equations. On a technical level, we use notions from geometric singular analysis to give a streamlined proof of semiglobal existence for the relevant quasilinear hyperbolic equation.
\end{abstract}

\maketitle

\section{Introduction}
\label{SI}

The goal of this paper is to introduce a novel generalized wave coordinate gauge on asymptotically flat spacetimes, and to demonstrate its utility by proving the existence of a piece of null infinity for the spacetime evolving from asymptotically flat initial data sets.

\begin{thm}[Main theorem, rough version]
\label{ThmIMain}
  Let $\Sigma=\{x\in\R^3\colon|x|>R\}$ for some $R>0$, and suppose $\gamma,k\in\CI(\Sigma;S^2 T^*\Sigma)$ are a Riemannian metric, resp.\ smooth symmetric 2-tensor on $\Sigma$ satisfying the constraint equations,\footnote{These are $R_\gamma-|k|_\gamma^2+(\tr_\gamma k)^2=0$, $\delta_\gamma k+\dd\tr_\gamma k=0$, with $R_\gamma$ the scalar curvature and $\delta_\gamma$ the (negative) divergence. As shown by Choquet-Bruhat \cite{ChoquetBruhatLocalEinstein}, the constraint equations are necessary and sufficient for the existence of a short time solution of the initial value problem for the Einstein vacuum equations.} with
  \[
    \gamma = \gamma_\bhm + \tilde\gamma,\qquad \gamma_\bhm=\Bigl(1-\frac{2\bhm}{r}\Bigr)^{-1}\dd r^2+r^2\slg\quad (\bhm\in\R,\ R>2\bhm),
  \]
  where $\slg$ is the standard metric on $\Sph^2$. Let $\ell_0\in(0,1)$. Suppose that $\tilde\gamma$ and $k$ are small in the sense that\footnote{This implies pointwise $r^{-1-\ell_0}$ decay of $\gamma$ to $\gamma_\bhm$, and pointwise $r^{-2-\ell_0}$ decay of $k$ as $r\to\infty$.}
  \begin{equation}
  \label{EqIMainSmall}
    \sum_{|\alpha|\leq N} \| r^{-1/2+\ell_0}(r\nabla)^\alpha(\tilde\gamma,r k) \|_{L^2} < \eps
  \end{equation}
  with $N$ large, $\eps>0$ small. Then there exists a Lorentzian metric $g$ on
  \[
    \Omega = \{ (t,r,\omega) \in [0,\infty)\times(R,\infty)\times\Sph^2 \colon r_*-t>R+1 \} \subset \R^4 = \{ (z^0=t,z^1,z^2,z^3) \},
  \]
  where $r_*=r+2\bhm\log(r-2\bhm)$, with the following properties:
  \begin{enumerate}
  \item $g$ solves the Einstein vacuum equations $\Ric(g)=0$;
  \item identifying $\Sigma\cap\{r_*>R+1\}$ with $t^{-1}(0)\subset\Omega$, the induced metric and second fundamental form of $g$ at $\Sigma$ are given by $\gamma$ and $k$;
  \item\label{ItIMainGauge} $g$ is in a modified wave coordinate gauge relative to the Schwarzschild metric $g_\bhm=-\bigl(1-\tfrac{2\bhm}{r}\bigr)\dd t^2+\gamma_\bhm$, see \eqref{EqIYGauge} and \eqref{EqIY0};
  \item\label{ItIMainDecay} $g$ approaches the Schwarzschild metric in a quantitative manner,
    \[
      g = g_\bhm + r^{-1}h,
    \]
    where all coefficients $h(\pa_{z^i},\pa_{z^j})$ are uniformly bounded. More precisely, if $L=\pa_t+\pa_{r_*}$ and $\ubar L=\pa_t-\pa_{r_*}$ denote outgoing and incoming null vector fields for the Schwarzschild metric, and $\Omega$ denotes an arbitrary vector field on $\Sph^2$, then\footnote{In the main body of the text, we will use a more convenient notation for metric coefficients, with $h(L,L)$, $h(\ubar L,\ubar L)$, $h(L,r^{-1}\Omega)$ denoted $h_{0 0}$, $h_{1 1}$, $h_{0\bar a}$ etc; see Lemma~\ref{LemmaNPCoeff}.}
    \begin{equation}
    \label{EqIMainDecay}
      |h(L,L)|,\ |h(L,r^{-1}\Omega)|,\ |\tr_{\slg}h|,\ |h(L,\ubar L)|,\ |h(\ubar L,\Omega)| \lesssim r^{-\ell_0+},
    \end{equation}
    while the trace-free part of the restriction of $h$ to $T\Sph^2$ and the component $h(\ubar L,\ubar L)$ have have smooth limits as $r\to\infty$, $|r_*-t|\lesssim 1$, with $(r_*-t)^{-\ell_0}$ decay.
  \end{enumerate}
\end{thm}

More generally, we prove a semiglobal existence theorem and the same asymptotics for the solution $g$ of a quasilinear hyperbolic (gauge-fixed) version of the Einstein equations for general (i.e.\ not necessarily arising from an initial data set) suitably decaying and regular Cauchy data for $h$; see Corollary~\ref{CorNYStab}. Through a combination of the new gauge with constraint damping and a simple nonlinear iteration scheme, we are able to obtain these asymptotics in one fell swoop.

\begin{figure}[!ht]
\centering
\includegraphics{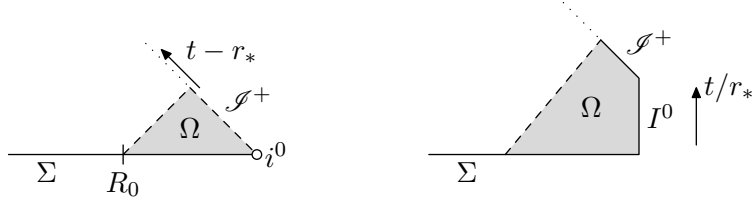}
\caption{Illustration of Theorem~\ref{ThmIMain}, on the left in a Penrose-diagrammatic fashion, and on the right in the blow-up of the Penrose diagram at spacelike infinity $i^0$.}
\label{FigIMain}
\end{figure}

We recall that Christodoulou--Klainerman \cite{ChristodoulouKlainermanStability} gave the first proof of the nonlinear stability of Minkowski space, with initial data given on all of $\R^3$ (requiring stronger decay $\ell_0>\half$ but less regularity); the evolving spacetime metric is geodesically complete. Klainerman--Nicol\`o \cite{KlainermanNicoloEvolution} gave a new proof of the stability of the exterior region (as in Theorem~\ref{ThmIMain}) using a double null foliation; see \cite{ShenMinkExtStab} for improvements. Earlier work by Friedrich \cite{FriedrichStability} established the global nonlinear stability for special initial data $(\gamma,k)$ which are equal to $(\gamma_\bhm,0)$ outside a compact set; the existence of such data was proved by Corvino \cite{CorvinoScalar,ChruscielDelayMapping}. Bieri \cite{BieriZipserStability} lowered the decay assumptions to $\ell_0>-\half$ and required only $N=3$ derivatives on the initial data. (There is a vast literature on extensions and variants of the nonlinear stability problem on asymptotically flat spacetimes, including \cite{WangThesis,LeFlochMaEinsteinMassive,BieriChruscielADMBondi,ChruscielScriPiece,TaylorEinsteinVlasov,LindbladTaylorVlasov,FajmanJoudiouxSmuleviciEinsteinVlasov,KlainermanSzeftelPolarized,DafermosHolzegelRodnianskiSchwarzschildStability,AnderssonBackdahlBlueMaKerr,HaefnerHintzVasyKerr,IonescuPausaderKG,DafermosHolzegelRodnianskiTaylorSchwarzschild,KlainermanSzeftelKerr,WangMinkEinsteinVlasov}.)

Closely related to the present work is the global stability proof by Lindblad--Rodnianski \cite{LindbladRodnianskiGlobalExistence,LindbladRodnianskiGlobalStability} in the standard wave coordinate gauge $\Box_g z^\mu=0$; due to logarithmic divergences arising in simplistic formal nonlinear iteration arguments (see \cite[\S1]{LindbladRodnianskiGlobalStability} and also~\eqref{EqIModel1} below), this gauge condition was considered unsuitable for a proof of global stability until Lindblad--Rodnianski \cite{LindbladRodnianskiWeakNull} discovered that the Einstein equations in wave coordinates satisfy a \emph{weak null condition} at null infinity $\scri^+$. Lindblad \cite{LindbladAsymptotics} subsequently proved sharp decay at $\scri^+$ by using vector field multipliers and commutators adapted to the large scale Schwarzschild geometry (rather than the Minkowski geometry as in \cite{LindbladRodnianskiGlobalStability}). In the wave coordinate gauge, only the first three components in~\eqref{EqIMainDecay} have the stated decay, while all other components of $h$ have leading order terms at $\scri^+$, with the exception of $h(\ubar L,\ubar L)$ which blows up logarithmically as $r\to\infty$. This result was extended by the author and Vasy in \cite{HintzVasyMink4} where it was shown that the geodesically complete spacetime metric $g$, evolving from initial data close to the trivial data, is polyhomogeneous on a compactification of $\R^4$ to a manifold with corners; this result utilizes a wave map gauge relative to the Schwarzschild metric (discussed further below). Furthermore, in this gauge, \cite{HintzVasyMink4} clarified the nature of the logarithmically divergent leading order term of $h(\ubar L,\ubar L)$ by relating its average over spherical sections of null infinity to the Bondi mass \cite{BondivdBMetznerGravity,ChristodoulouNonlinear}. In a different direction, Keir \cite{KeirWeak}, focusing on the analysis of weak null conditions, proved the global well-posedness of the Einstein equations in harmonic coordinates (in the standard formulation, i.e.\ without constraint damping) for general small Cauchy data.

By contrast, in the novel gauge introduced here, $h(\ubar L,\ubar L)$ remains bounded, the spherical averages of its limit at $\scri^+$ being related to the Bondi mass; and the trace-free spherical part of $h$ directly encodes the Bondi news function and outgoing energy flux, as indicated in Remark~\ref{RmkNYBondi} (following \cite[\S8]{HintzVasyMink4}). All other metric components have faster decay; in this sense, our gauge suppresses all `non-physical' degrees of freedom to leading order at null infinity. We discuss this in~\S\ref{SsIY}.

In addition to this improved decay, our analysis takes full advantage of notions from geometric singular analysis, concretely the notions of b- and edge-metrics and -operators going back to Melrose \cite{MelroseAPS} and Mazzeo \cite{MazzeoEdge}; see~\S\ref{SsIeb}.

\subsection{Constraint damping and novel gauge}
\label{SsIY}

A natural \emph{generalized wave coordinate gauge} or \emph{generalized harmonic gauge} for a spacetime metric $g$ which is a perturbation of $g_\bhm$ is the \emph{wave map gauge} relative to $g_\bhm$: this requires the identity map $(\Omega,g)\to(\Omega,g_\bhm)$ to be a wave map. One can solve the Einstein vacuum equations in this gauge by solving the quasilinear wave equation
\begin{equation}
\label{EqIY0}
  \Ric(g) - \delta_g^*\Ups(g;g_\bhm) = 0,\qquad \Ups(g;g_\bhm)^\kappa:=g^{\mu\nu}(\Gamma(g)_{\mu\nu}^\kappa-\Gamma(g_\bhm)_{\mu\nu}^\kappa)\ \ \text{(gauge 1-form)},
\end{equation}
for $g$; here, $(\delta_g^*\Ups)_{\mu\nu}=\half(\Ups_{\mu;\nu}+\Ups_{\nu;\mu})$ is the symmetric gradient. Constraint damping amounts to modifying $\delta_g^*$ by zeroth order terms; it was introduced in \cite{GundlachCalabreseHinderMartinConstraintDamping,BrodbeckFrittelliHubnerReulaSCP} and used to by Pretorius \cite{PretoriusBinaryBlackHole} as a device in numerical evolution schemes to ensure that violations of the gauge condition are damped. It also played a key role in the recent proofs of black hole stability in cosmological spacetimes \cite{HintzVasyKdSStability,HintzKNdSStability}: when solving linearizations of equation~\eqref{EqIY0} (with nontrivial right hand side) in a Nash--Moser iteration scheme, constraint damping ensures improved decay of $\Ups(g;g_\bhm)$ throughout the iteration. Concretely, we modify $(\delta_g^*\Ups)_{\mu\nu}$ in~\eqref{EqIY0} by a zeroth order term not involving derivatives of $g$ to
\[
  (\delta_{g,E^\cC}^*\Ups)_{\mu\nu} := \half(\Ups_{\mu;\nu}+\Ups_{\nu;\mu}) + 2\gamma^\cC(\cd_\mu\Ups_\nu+\cd_\nu\Ups_\mu) - \gamma^\cC\Ups_\kappa\cd^\kappa g_{\mu\nu},\qquad E^\cC=(\cd,\gamma^\cC),
\]
where we take $\cd=r^{-1}\,\dd t$ and $\gamma^\cC>0$. (This is similar to the definition of $\tilde\delta^*$ in \cite[\S3.3]{HintzVasyMink4}.) Usage of $\delta_{g,E^\cC}^*$ instead of $\delta_g^*$ in the modified gauge-fixed Einstein vacuum equations
\begin{equation}
\label{EqIY1}
  \Ric(g) - \delta_{g,E^\cC}^*\Ups(g;g_\bhm) = 0
\end{equation}
does \emph{not} change the gauge in which one solves $\Ric(g)=0$. However, it ensures, for general Cauchy data which may violate the constraint equations, that $\Ups$ satisfies a modified (here: \emph{damped}) wave equation $\delta_g\sfG_g\delta_{g,E^\cC}^*\Ups(g;g_\bhm)=0$ by virtue of the second Bianchi identity. (Here, $\sfG_g=1-\half g\tr_g$ is the trace reversal operator, and $\delta_g$ is the negative divergence.) On Minkowski space and with general Cauchy data, this ensures that $\Ups(g;g_\bhm)$ decays faster than $r^{-1}$ at null infinity, which concretely means that certain metric components---$4$ in number, matching the number of components of the 1-form $\Ups(g;g_\bhm)$---of $r^{-1}h=g-g_\bhm$ (in fact the first three in~\eqref{EqIMainDecay}, with $\Omega$ accounting for $2$ components) similarly have stronger decay; notably, the component $h(L,L)$ controls the deviation of outgoing light cones for the metric $g$ from the Schwarzschildean ones. This improved decay and the resulting fixing of the geometry near null infinity allowed for an application of a global nonlinear iteration scheme for solving~\eqref{EqIY1} in \cite{HintzVasyMink4}.

The new gauge we introduce here is a modification of $\Ups(g;g_\bhm)$ by a zeroth order term,
\begin{equation}
\label{EqIYGauge}
  \Ups_{E^\Ups}(g;g_\bhm)_\mu := \Ups(g;g_\bhm)_\mu - 2\gamma^\Ups\cd^\nu(g-g_\bhm)_{\mu\nu};
\end{equation}
we again use $\cd=r^{-1}\dd t$, and $\gamma^\Ups<0$. (See Definitions~\ref{DefLMod} and \ref{DefNEOp}, and Remark~\ref{RmkLDuality} for the duality of gauge modifications and constraint damping.) The gauge-fixed Einstein vacuum equations we shall solve in the proof of Theorem~\ref{ThmIMain} are then
\begin{equation}
\label{EqIY2}
  \Ric(g) - \delta_{g,E^\cC}^*\Ups_{E^\Ups}(g;g_\bhm) = 0.
\end{equation}
The coefficients of $r^{-1}h=g-g_\bhm$ which have improved decay by virtue of $\Ups_{E^\Ups}(g;g_\bhm)=0$ (or strong decay of $\Ups_{E^\Ups}(g;g_\bhm)$ at $\scri^+$ due to constraint damping) are the same as in the formulation~\eqref{EqIY1}. On the other hand, upon \emph{combining} the new gauge with the ungauged Einstein operator, $2$ further components of $h$ (the final two in~\eqref{EqIMainDecay}) have improved decay, and furthermore $h(\ubar L,\ubar L)$ does not diverge logarithmically anymore; see~\S\ref{SsNY}. We alert the reader to Appendix~\ref{SMax} where gauge changes and constraint damping of this sort are discussed in the context of the Maxwell equations; there, we also give a more conceptual explanation for why the gauge modification has the advertised effect.

We substantiate this discussion schematically in terms of the often used model for couplings and semilinear interactions for the Einstein vacuum equations in harmonic gauge,
\[
  \Box_g\phi_1=(\pa_t^2-\pa_x^2)\phi_1=0,\qquad \Box_g\phi_2=(\pa_t\phi_1)^2,
\]
with $g$ the Minkowski metric (see \cite[\S1]{LindbladRodnianskiGlobalStability}). Here,
\begin{enumerate}
\item $\phi_1$ encodes gravitational radiation escaping to null infinity and corresponds to the trace-free spherical part of metric perturbations $r^{-1}h$ above;
\item $\phi_2\sim h(\ubar L,\ubar L)$ encodes the Bondi mass.
\end{enumerate}
The $\cO(r^{-1})$ decay of $\phi_1$ creates $\cO(r^{-2})$ forcing for $\phi_2$, leading to the logarithmic divergence of $\phi_2=\cO(r^{-1}\log r)$ at $\scri^+$. We supplement this by two more equations,
\begin{equation}
\label{EqIModel1}
  \Box_g\phi_1 = 0,\qquad
  \Box_g\phi_2 = (\pa_t\phi_1)^2, \qquad
  \Box_g\phi_3=0,\qquad
  \Box_g\phi_4=0,
\end{equation}
where we ignore couplings at sub-leading order at $\scri^+$. Here,
\begin{enumerate}
\setcounter{enumi}{2}
\item $\phi_3$ models the $4$ metric coefficients whose leading order behavior at $\scri^+$ is constrained by the wave coordinate condition, as discussed after~\eqref{EqIY1};
\item $\phi_4$ models the remaining $3$ metric coefficients which are affected only once one combines the new gauge with the ungauged Einstein equations, and which do not encode any leading order physical degrees of freedom at $\scri^+$.
\end{enumerate}

Constraint damping turns the equation for $\phi_3$ into a damped wave equation of the sort
\[
  (\Box_g + 2\gamma^\cC r^{-1}\pa_t)\phi_3 = 0;
\]
this leads to $\phi_3=\cO(r^{-1-\gamma^\cC})$ decay at null infinity. Since a main effect of constraint damping is of quasilinear nature (namely, it fixes the geometry near null infinity), a more precise model than~\eqref{EqIModel1} replaces all occurrences of $g$ by ``$g+\phi_3$''; this makes apparent the advantage of ensuring better decay for $\phi_3$.

The improvement afforded by the gauge change leads to the schematic model
\begin{equation}
\label{EqIModel3}
\begin{aligned}
  \Box_{g+\phi_3}\phi_1&=0, &\quad
  (\Box_{g+\phi_3}-2\gamma^\Ups r^{-1}\pa_t)\phi_2 &= (\pa_t\phi_1)^2, \\
  (\Box_{g+\phi_3}+2\gamma^\cC r^{-1}\pa_t)\phi_3&=0, &\quad
  (\Box_{g+\phi_3}-2\gamma^\Ups r^{-1}\pa_t)\phi_4&=0.
\end{aligned}
\end{equation}
Thus, $\phi_3$ and $\phi_4$ have better-than-$r^{-1}$ decay at $\scri^+$, and the $\cO(r^{-2})$ forcing term for $\phi_2$ is no longer borderline, and hence $\phi_2=\cO(r^{-1})$. This leaves $\phi_1,\phi_2$ as the only components with nontrivial radiation fields; the other components ($\phi_3$ and $\phi_4$) decay faster. The relationship between the model~\eqref{EqIModel3} and the gauge-fixed Einstein equations is further discussed at the end of \S\ref{SL}, after Definition~\ref{DefNPMetrics}, and after the statement of Corollary~\ref{CorNENonlin}.

\subsection{Energy estimates and edge-b-metrics}
\label{SsIeb}

Our analysis here is based on energy estimates. The rough `background' estimate uses the vector field multiplier
\[
  W=\Bigl(\frac{r}{r_*-t}\Bigr)^{2\alpha_{\!\scri}}(r_*-t)^{2\alpha_0}\bigl(r L + (r_*-t)\pa_t\bigr),\qquad L=\pa_t+\pa_{r_*},
\]
for suitable weights $\alpha_0,\alpha_{\!\scri}\in\R$; this is stronger than $\pa_t$ and weaker than the conformal Morawetz vector field, while still being compatible with the types of metric perturbations one encounters in the stability problem. Concretely, usage of $W$ allows one to control the derivatives of the metric perturbation along
\begin{equation}
\label{EqIVf}
\begin{gathered}
  r(\pa_t+\pa_{r_*}),\quad (r_*-t)(\pa_t-\pa_{r_*}) \ \text{(weighted approx.\ outgoing/incoming derivative)}, \\
  \Bigl(\frac{r_*-t}{r}\Bigr)^{1/2}\Omega\ \ \text{(spherical vector fields with $r^{-1/2}$ decay at null infinity)},
\end{gathered}
\end{equation}
in a weighted spacetime $L^2$-space. (One can replace the incoming null vector field by the scaling vector field $t\pa_t+r_*\pa_{r_*}$.) Higher regularity is proved by commuting stronger vector fields (see Remark~\ref{RmkNBeb}) through the equation; a minor simplification is that due to our strong background estimate we can relax the requirements on these commutator vector fields, cf.\ Lemma~\ref{LemmaNMOp}. By contrast, in \cite{LindbladRodnianskiGlobalStability}, the background estimate is weaker than the edge-b-estimate, and thus the commutator vector fields need to be chosen more carefully, much as in \cite{KlainermanNullCondition}.

Besides proving the improved asymptotics in the new gauge, a secondary goal of this paper is to contribute to the development of the global analytic point of view for nonelliptic PDE using techniques from geometric singular analysis. Concretely, as discussed in detail in~\S\ref{SsNM}, the Schwarzschild metric is a weighted \emph{edge-metric} (with Lorentzian signature) at null infinity in the sense of Mazzeo \cite{MazzeoEdge}; see Corollary~\ref{CorNMLot}. Indeed, pulling back $g_\bhm$ to the interior of
\begin{equation}
\label{EqICpt}
  \R_{t_*} \times [0,\infty)_{x_{\!\scri}} \times \Sph^2,\qquad t_*=t-r_*,\quad x_{\!\scri}=r^{-1/2},
\end{equation}
with $x_{\!\scri}^{-1}(0)$ being (the interior of) null infinity, one computes
\[
  x_{\!\scri}^2 g_\bhm = 4\,\dd t_*\frac{\dd x_{\!\scri}}{x_{\!\scri}} + \slg_{a b}\frac{\dd x^a}{x_{\!\scri}}\frac{\dd x^b}{x_{\!\scri}} + \text{[terms with more decay]}
\]
where $x^2,x^3$ are local coordinates on $\Sph^2$. Dual to the 1-forms $\dd t_*$, $\tfrac{\dd x_{\!\scri}}{x_{\!\scri}}$, $\tfrac{\dd x^a}{x_{\!\scri}}$ appearing here are the vector fields $\pa_{t_*}$, $x_{\!\scri}\pa_{x_{\!\scri}}$, $x_{\!\scri}\pa_{x^a}$, which are precisely those smooth vector fields on the manifold~\eqref{EqICpt} which are tangent to the fibers of the fibration $(t_*,\omega)\mapsto\omega$ of the boundary $x_{\!\scri}^{-1}(0)$; and linear combinations of these vector fields are precisely those listed in~\eqref{EqIVf}. The compactification of the domain $\Omega$ in Theorem~\ref{ThmIMain}, as shown on the right in Figure~\ref{FigIMain}, has a second boundary hypersurface $I^0$ where $g_\bhm$ is a weighted b-metric \cite{MelroseAPS}; globally, $g_\bhm$ is a weighted edge-b-metric, or \emph{eb-metric} for short. This observation is key for the streamlining of the functional analytic setup in the present paper as compared to \cite{HintzVasyMink4}.\footnote{The companion paper \cite{HintzVasyScrieb}, joint with Vasy, goes significantly further by giving a (microlocal) linear analysis of a large class of tensorial wave type equations; we refer to its introduction for a more detailed discussion of edge-b-analysis near null infinity. The present paper only uses physical space techniques.}

The metric perturbations arising in Theorem~\ref{ThmIMain} are lower order perturbations of $g_\bhm$ as symmetric edge-b-2-tensors; see Lemma~\ref{LemmaNPCoeff}. Thus, regularity with respect to the vector fields~\eqref{EqIVf} is a very natural notion. Unlike in Riemannian geometry, there are typically many different types of rescaled vector bundles and boundary fibration structures with respect to which a given Lorentzian metric is nondegenerate down to boundaries at infinity, such as $x_{\!\scri}^{-1}(0)$ in~\eqref{EqICpt}. And indeed, while the edge-b point of view is convenient for the purpose of proving \emph{estimates}, controlling the \emph{geometry} of metric perturbations is more conveniently done in terms of the standard vector fields $\pa_{z^\mu}$ on $\R^4$ or linear combinations thereof as used in~\eqref{EqIMainDecay} and discussed in detail around Definition~\ref{DefNMVb}; cf.\ the significance of $h(L,L)$ for controlling outgoing light cones. Since the Einstein equations are \emph{quasilinear}, it is important to understand the relationship between the two points of view (Lemma~\ref{LemmaNMBundleRel}).

\subsection{Structure of the paper}

In~\S\ref{SL}, we present calculations for the linearized gauge-fixed Einstein vacuum equations on Minkowski space, with constraint damping and the (linearization of the) novel gauge, which lend support to the claims made in~\S\ref{SsIY}. In~\S\ref{SN}, we prove Theorem~\ref{ThmIMain}. We first introduce edge-b-structures in~\S\ref{SsND}; the partial compactification of the spacetime on which we shall work and a basic edge-b-energy estimate are presented in~\S\ref{SsNM}. The stability proof starts in~\S\ref{SsNP} where we define the class of metric perturbations arising in the stability problem in our new gauge. In~\S\ref{SsNE}, we define the modified gauge-fixed Einstein operator and describe its structure as an edge-b-differential operator. This is used in~\S\ref{SsNQ} to prove (tame) energy estimates and in~\S\ref{SsNY} to obtain sharp decay for metric perturbations using a Nash--Moser iteration. Appendix~\ref{SMax} illustrates the choice of gauge and constraint damping in the simpler setting of the Maxwell equations.

\subsection*{Acknowledgments}

The author is grateful for support from NSF grant DMS-1955614. This material is based upon work supported by the NSF under grant No.\ DMS-1440140 while the author was in residence at MSRI during the Fall 2019 semester. Part of this work was conducted during the period the author held Clay and Sloan Research Fellowships.

\section{Motivation; calculations on Minkowski space}
\label{SL}

Fix a background metric $g^0$. Then the operator
\begin{equation}
\label{EqLPUps}
  P(g) = \Ric(g) - \delta_g^*\Ups(g;g^0),\qquad \Ups(g;g^0)=g(g^0)^{-1}\delta_g\sfG_g g^0,
\end{equation}
is a quasilinear wave operator when $g$ is a Lorentzian metric; that is, its linearization is principally scalar, and its principal part is equal to $\half\Box_g$ (see below).

\begin{lemma}[Linearizations]
\label{LemmaLPLin}
  The linearization of the Ricci tensor is
  \begin{equation}
  \label{EqLPLinRic}
    D_g\Ric = \half\Box_g - \delta_g^*\delta_g\sfG_g + \sR_g,\quad
    (\sR_g u)_{\mu\nu}=R^\kappa{}_{\mu\nu}{}^\lambda u_{\kappa\lambda}+\half(\Ric(g)_\mu{}^\kappa u_{\kappa\nu}+\Ric(g)^\kappa{}_\nu u_{\mu\kappa}),
  \end{equation}
  where $R$ is the Riemann curvature tensor of $g$, and $\Box_g=-\tr_g\nabla^2$. Moreover,
  \begin{equation}
  \label{EqLPLinUps}
  \begin{split}
    &\Ups(g;g^0)_\mu=g_{\mu\nu}g^{\kappa\lambda}(\Gamma(g)_{\kappa\lambda}^\nu-\Gamma(g^0)_{\kappa\lambda}^\nu), \\
    &D_g\Ups(-;g^0) = -\delta_g\sfG_g - \sC_g + \sY_g, \\
    &\qquad \sC_g(u)_\kappa = g_{\kappa\lambda}C_{\mu\nu}^\lambda u^{\mu\nu},\quad
    C_{\mu\nu}^\lambda = \Gamma(g)_{\mu\nu}^\lambda - \Gamma(g^0)_{\mu\nu}^\lambda,\qquad
    \sY_g(u)_\kappa = \Ups(g;g^0)^\lambda u_{\kappa\lambda}.
  \end{split}
  \end{equation}
\end{lemma}
\begin{proof}
  See \cite{GrahamLeeConformalEinstein}, \cite[\S3.3]{HintzVasyMink4}.
\end{proof}

In the linearization of $P$ around $g=g^0$, given by $D_g P = D_g\Ric + \delta_g^*\delta_g\sfG_g$, we now generalize $\delta_g^*$, resp.\ $\delta_g$, for the purpose of (linearized) constraint damping, resp.\ gauge change, as follows:

\begin{definition}[Modifications]
\label{DefLMod}
  Let $E^\cC=(\cd^\cC,\gamma^\cC)$, where $\cd^\cC$ is a 1-form on spacetime, and $\gamma^\cC\in\R$. The \emph{modified symmetric gradient} is then defined as
  \begin{equation}
  \label{EqLModCD}
    \delta_{g,E^\cC}^* := \delta_g^* + \gamma^\cC\bigl(2\cd^\cC\otimes_s(-) - g \iota_{g^{-1}(\cd^\cC)}\bigr).
  \end{equation}
  For a pair $E^\Ups=(\cd^\Ups,\gamma^\Ups)$, we define the \emph{modified divergence} by
  \begin{equation}
  \label{EqLModGC}
    \delta_{g,E^\Ups} = (\delta_{g,E^\Ups}^*)^* = \delta_g + \gamma^\Ups\bigl(2\iota_{g^{-1}(\cd^\Ups)} - \cd^\Ups\tr_g\bigr).
  \end{equation}
  Finally, the \emph{linearized modified gauge-fixed Einstein operator} is
  \begin{equation}
  \label{EqLEin}
    P'_{g,E^\cC,E^\Ups} := D_g\Ric + \delta_{g,E^\cC}^*\delta_{g,E^\Ups}\sfG_g.
  \end{equation}
\end{definition}

Consider now the Minkowski metric
\begin{equation}
\label{EqLMinkMet}
  \ubar g:=-\dd x^0\,\dd x^1+r^2\slg,\qquad
  x^0=t+r,\quad x^1=t-r.
\end{equation}
In this section, we study the asymptotic behavior of solutions of $P'_{\ubar g,E^\cC,E^\Ups}(r^{-1}u)=0$ at null infinity $\scri^+$, i.e.\ for bounded $x^1$ when $x^0\to\infty$. In this region, we fix
\begin{equation}
\label{EqLMod}
  \cd^\cC=\cd^\Ups=r^{-1}\,\dd t,\qquad
  \gamma^\cC\in(0,1),\quad
  \gamma^\Ups\in(-1,0).
\end{equation}
We work with the bundle splittings
\begin{equation}
\label{EqLSplit}
\begin{split}
  T^*\R^4 &= \la \dd x^0\ra \oplus \la \dd x^1\ra \oplus r T^*\Sph^2, \\
  S^2 T^*\R^4 &= \la (\dd x^0)^2\ra \oplus \la 2 \dd x^0\dd x^1 \ra \oplus (2 \dd x^0\otimes_s r T^*\Sph^2) \\
    &\quad\qquad \oplus \la(\dd x^1)^2\ra \oplus (2 \dd x^1\otimes_s r T^*\Sph^2) \oplus \la r^2\slg\ra \oplus r^2\ker\sltr.
\end{split}
\end{equation}
Thus, we rescale the spherical part of the cotangent bundle, recording e.g.\ the covector $\omega_0\,\dd x^0+\omega_1\,\dd x^1+r\slomega$ with $\slomega\in T^*\Sph^2$ as $(\omega_0,\omega_1,\slomega)$. We shall only record the `main' terms of
\[
  2 P'_{\ubar g,E^\cC,E^\Ups} = \Box_{\ubar g} + 2(\delta_{\ubar g,E^\cC}^*\delta_{\ubar g,E^\Ups}-\delta_{\ubar g}^*\delta_{\ubar g})\sfG_{\ubar g} + 2\sR_{\ubar g}
\]
and drop all `error' terms (writing `$\equiv$' for an equality up to error terms). Concretely, we assign the weights $1$, $-1$, $0$, $0$ to $r$, $\pa_0$, $\pa_1$, $\cV(\Sph^2)$ (thus regarding $r\pa_0\sim r(\pa_t+\pa_r)$ $\pa_1\sim\pa_t-\pa_r$, $\cV(\Sph^2)$ as unweighted vector fields), and only record terms of total weight $\leq 0$. In the proof of Proposition~\ref{PropNELin}, we shall find $r^2\Box_{\ubar g} r^{-1} \equiv 4 \pa_1 r\pa_0$ and expressions for $\delta_{\ubar g}$, $\delta_{\ubar g}^*$, and $\sfG_{\ubar g}$ (the first terms in~\eqref{EqNELinDiv}, \eqref{EqNELinDelstar}, and \eqref{EqNELinTrRev}, respectively), and for $\delta_{\ubar g,E^\cC}^*-\delta_{\ubar g}^*$, resp.\ $\delta_{\ubar g,E^\Ups}-\delta_{\ubar g}$ (the first terms in~\eqref{EqELinTdelDelDiff2}, resp.\ \eqref{EqNELinDelstar2}). They give
\begin{equation}
\label{EqLOp}
  L_{\ubar g,E^\cC,E^\Ups} := 2 r^2 P'_{\ubar g,E^\cC,E^\Ups} r^{-1} \equiv 2\pa_1(2 r\pa_0+A_{E^\cC,E^\Ups}),
\end{equation}
where the endomorphism $A_{E^\cC,E^\Ups}$ of $S^2 T^*\R^4$ is given by
\begin{equation}
\label{EqLAEnd}
  A_{E^\cC,E^\Ups}=
  \openbigpmatrix{3pt}
    2\gamma^\cC & 0 & 0 & 0 & 0 & 0 & 0 \\
    -\gamma^\Ups & -\gamma^\Ups & 0 & 0 & 0 & 0 & 0 \\
    0 & 0 & \gamma^\cC & 0 & 0 & 0 & 0 \\
    0 & -2\gamma^\Ups & 0 & -2\gamma^\Ups & 0 & \gamma^\cC & 0 \\
    0 & 0 & \gamma^\cC-\gamma^\Ups & 0 & -\gamma^\Ups & 0 & 0 \\
    2\gamma^\cC & 0 & 0 & 0 & 0 & \gamma^\cC & 0 \\
    0 & 0 & 0 & 0 & 0 & 0 & 0
  \closebigpmatrix
\end{equation}

Passing to $\rho_{\!\scri}=(x^0)^{-1}$ (which, in the region of bounded $x^1=t-r$, vanishes at future null infinity $\scri^+$), we note that $2 r=x^0-x^1=\rho_{\!\scri}^{-1}(1-\rho_{\!\scri} x^1)$ and $\pa_0=-\rho_{\!\scri}^2\pa_{\rho_{\!\scri}}$; thus,
\[
  r^2 P'_{\ubar g,E^\cC,E^\Ups} r^{-1}\equiv -2\pa_1(\rho_{\!\scri}\pa_{\rho_{\!\scri}}-A_{E^\cC,E^\Ups}).
\]
By standard regular-singular ODE analysis (and as previously shown rigorously in~\cite[\S3.3]{HintzVasyMink4}), we can read off the decay at $\scri^+$ of a metric perturbation $u$ solving
\begin{equation}
\label{EqLEq}
  P'_{\ubar g,E^\cC,E^\Ups}u=0
\end{equation}
from the spectral decomposition of $A_{E^\cC,E^\Ups}$.

\begin{rmk}[Duality of constraint damping and gauge change]
\label{RmkLDuality}
  By~\eqref{EqLPLinRic}, the adjoint of $D_g\Ric$ is $(D_g\Ric)^*=\sfG_g\circ D_g\Ric\circ\sfG_g$ (i.e.\ $\sfG_g\circ D_g\Ric$ is formally self-adjoint); thus,
  \begin{equation}
  \label{EqLAdj}
    \sfG_g (P'_{g,E^\cC,E^\Ups})^*\sfG_g = P'_{g,E^\Ups,E^\cC},
  \end{equation}
  demonstrating a duality between constraint damping and gauge changes. Equation~\eqref{EqLAdj} also implies $\sfG_{\ubar g} A_{E^\cC,E^\Ups}^*\sfG_{\ubar g}=-A_{E^\Ups,E^\cC}$. Since we would like as many eigenvalues as possible of $A_{E^\cC,E^\Ups}$ to be positive, this suggests taking $\gamma^\cC$ and $\gamma^\Ups$ to have opposite signs. In view of~\eqref{EqLAdj}, this forces the endomorphism $A_{E^\Ups,E^\cC}$ corresponding to $(P'_{g,E^\cC,E^\Ups})^*$ to have many \emph{negative} eigenvalues. See Appendix~\ref{SMax} for a discussion of this point in a simpler context.
\end{rmk}

To study $A_{E^\cC,E^\Ups}$, we introduce the bundle projections (respecting the splitting~\eqref{EqLSplit})
\begin{equation}
\label{EqLProj}
\begin{split}
  \pi^\cC &\colon S^2 T^*\R^4 \to \la (\dd x^0)^2 \ra \oplus (2\dd x^0\otimes_s r T^*\Sph^2) \oplus \la r^2\slg \ra, \\
  \pi^\Ups &\colon S^2 T^*\R^4 \to \la 2\dd x^0\dd x^1\ra \oplus (2\dd x^1\otimes_s r T^*\Sph^2), \\
  \slpi_0 &\colon S^2 T^*\R^4 \to r^2\ker\sltr, \\
  \pi_{1 1} &\colon S^2 T^*\R^4 \to \la(\dd x^1)^2\ra.
\end{split}
\end{equation}
Then $\pi^\cC A_{E^\cC,E^\Ups}(1-\pi^\cC)=0$; in the splitting $\ran\pi^\cC\oplus\ran(1-\pi^\cC)$, the top left block of $A_{E^\cC,E^\Ups}$ (capturing rows and columns $1,3,6$) is then
\begin{subequations}
\begin{equation}
\label{EqLApiC}
  \pi^\cC A_{E^\cC,E^\Ups}\pi^\cC = \begin{pmatrix} 2\gamma^\cC & 0 & 0 \\ 0 & \gamma^\cC & 0 \\ 2\gamma^\cC & 0 & \gamma^\cC \end{pmatrix}
\end{equation}
Thus, $\pi^\cC u$ is expected to have components of size $\cO(\rho_{\!\scri}^{1+\lambda})$ at $\scri^+$, where $\lambda=\gamma^\cC,2\gamma^\cC$. Next, the bottom right block (capturing rows $2,4,5,7$) is
\begin{equation}
\label{EqLApiCcompl}
  (1-\pi^\cC)A_{E^\cC,E^\Ups}(1-\pi^\cC) =
  \begin{pmatrix}
    -\gamma^\Ups & 0 & 0 & 0 \\
    -2\gamma^\Ups & -2\gamma^\Ups & 0 & 0 \\
    0 & 0 & -\gamma^\Ups & 0 \\
    0 & 0 & 0 & 0
  \end{pmatrix},
\end{equation}
\end{subequations}
with eigenvalues $-\gamma^\Ups$, $-2\gamma^\Ups$, and $0$. It is thus natural to further split off the trace-free spherical part (the final row) using $\slpi_0$, with $\slpi_0 A_{E^\cC,E^\Ups}\slpi_0 = 0$. In~\S\ref{SsNE}, we will see that when solving the nonlinear Einstein equations via an iteration scheme, the $\slpi_0$ part will be a source term for the $(1,1)$ component in the subsequent iteration step; this is why we further split the bundle $\ran(\pi^\Ups+\pi_{1 1})$ into the ranges of $\pi^\Ups$ (rows 2 and 5 of $A_{E^\cC,E^\Ups}$), $\pi_{1 1}$ (row $4$).

Altogether then, the solution $u$ of~\eqref{EqLEq} can be analyzed step by step for bounded $t-r$ as follows (we omit error terms throughout):
\begin{enumerate}
\item $\pi^\cC u$ satisfies a decoupled equation (to leading order), thus $\pi^\cC u=\cO(\rho_{\!\scri}^{1+\gamma^\cC})$.
\item $\pi^\Ups u$ satisfies an equation with source terms given by $\pi^\cC u$. Choosing our parameters so that $-\gamma^\Ups<\gamma^\cC$, we then have $\pi_\Ups u=\cO(\rho_{\!\scri}^{1-\gamma^\Ups})$.
\item $\slpi_0 u$ has a radiation field, i.e.\ a leading order term of size $\cO(\rho_{\!\scri})=\cO(r^{-1})$, and has lower order terms of size $\rho_{\!\scri}^{1-\gamma^\Ups}$ from coupling to the remaining metric components.
\item $\pi_{1 1} u$ has the same decay as $\pi^\Ups u$.
\end{enumerate}

These improved decay rates (compared to the $\cO(\rho_{\!\scri})$ decay of typical scalar waves on Minkowski space) will persist for the nonlinear gauge-fixed Einstein vacuum equations, except for the decay of $\pi_{1 1}u$ (which is replaced by a $\cO(\rho_{\!\scri})$-leading order term), as already indicated before. In terms of our model~\eqref{EqIModel3}, $\slpi_0 u,\pi_{1 1}u,\pi^\cC u,\pi^\Ups u$ thus correspond to $\phi_1,\phi_2,\phi_3,\phi_4$, respectively. Leaving the model~\eqref{EqIModel3} behind, one can simplify the above scheme by solving at once for $(\pi^\cC+\pi^\Ups)u$, which in itself satisfies a decoupled equation leading to improved decay. This is the path we will take in~\S\ref{SsNY}.

\section{Nonlinear stability}
\label{SN}

\subsection{Differential operators and function spaces}
\label{SsND}

Consider an $n$-dimensional manifold $M$ with corners which has exactly two embedded boundary hypersurfaces $H_0$, $H_1$. Assume furthermore that $H_1$ is equipped with a fibration $\phi\colon H_1\to Y$ with typical fiber $F$.

\begin{definition}[b- and edge-b-vector fields]
\fakephantomsection\label{DefNDVf}
  \begin{enumerate}
  \item The Lie algebra $\Vb(M)\subset\cV(M)$ of \emph{b-vector fields} \cite{MelroseAPS} consists of all smooth vector fields $V\in\cV(M)$ which are tangent to $H_0$ and $H_1$.
  \item The Lie algebra $\Veb(M)\subset\Vb(M)$ of \emph{edge-b-vector fields} consists of all b-vector fields $V\in\Vb(M)$ for which $V|_{H_1}\in\Vb(H_1)$ is tangent to the fibers of $H_1\to Y$.
  \end{enumerate}
\end{definition}

On manifolds with a single embedded boundary hypersurface, edge vector fields were introduced by Mazzeo \cite{MazzeoEdge}. See \cite{AlbinGellRedmanDirac} for iterated structures giving rise to generalizations of $\Veb(M)$.

We discuss here only the case of interest for us: $H_0$ and $H_1$ have nonempty intersection, and a neighborhood of $H_0\cap H_1\subset M$ is diffeomorphic to
\begin{equation}
\label{EqNDCoord}
  [0,\infty)_{\rho_0} \times [0,\infty)_{\rho_1} \times \R^{n-2}_y,\qquad y=(y^2,\ldots,y^{n-1}),
\end{equation}
where $\rho_0,\rho_1$ are defining functions of $H_0,H_1$, respectively, and with the fibration of $H_1$ given by $(\rho_0,y)\mapsto y$; thus, the fibers $F$ are 1-dimensional. In this case, elements of $\Vb(M)$, resp.\ $\Veb(M)$ are linear combinations, with $\CI(M)$ coefficients, of
\begin{equation}
\label{EqNDbeb}
  \rho_0\pa_{\rho_0},\ 
  \rho_1\pa_{\rho_1},\ 
  \pa_{y^2},\dots,\pa_{y^{n-1}},\qquad\text{resp.}\qquad
  \rho_0\pa_{\rho_0},\ 
  \rho_1\pa_{\rho_1},\ 
  \rho_1\pa_{y^2},\dots,\rho_1\pa_{y^{n-1}}.
\end{equation}
The \emph{b-tangent bundle} and \emph{eb-tangent bundle}
\[
  \Tb M\to M,\qquad \Teb M\to M,
\]
are then the rank $n$ vector bundles with local frames given by the respective sets of vector fields~\eqref{EqNDbeb}; over the interior $M^\circ$, these are naturally isomorphic to the standard tangent bundle. By continuous extension from $M^\circ$, one can thus regard smooth sections of $\Tb M$ as vector fields on $M$, and in this sense, we have $\Vb(M)=\CI(M;\Tb M)$, likewise $\Veb(M)=\CI(M;\Teb M)$.\footnote{The benefit of using $\Tb M$ and $\Teb M$ is that one can capture the precise behavior (regularity, boundedness, decay) of vector fields at $\pa M$ without the need for any irrelevant choices (e.g.\ metrics).} The dual bundles $\Tb^*M\to M$ and $\Teb^*M\to M$ are called \emph{b-cotangent bundle} and \emph{eb-cotangent bundle}, respectively. Their smooth sections are linear combinations with $\CI(M)$ coefficients of
\[
  \frac{\dd\rho_0}{\rho_0},\ \frac{\dd\rho_1}{\rho_1},\ \dd y^2,\ldots,\dd y^{n-1}, \qquad\text{resp.}\qquad
  \frac{\dd\rho_0}{\rho_0},\ \frac{\dd\rho_1}{\rho_1},\ \frac{\dd y^2}{\rho_1},\ldots,\frac{\dd y^{n-1}}{\rho_1}.
\]

\begin{definition}[b- and eb-differential operators]
\label{DefNDOp}
  Let $k\in\N_0$, $\alpha_0,\alpha_1\in\R$.
  \begin{enumerate}
  \item The space $\Diffb^{k,(\alpha_0,\alpha_1)}(M)=\rho_0^{-\alpha_0}\rho_1^{-\alpha_1}\Diffb^k(M)$ consists of all differential operators $P$ on $M^\circ$ of the form $P=\rho_0^{-\alpha_0}\rho_1^{-\alpha_1}P_0$, where $P_0\in\Diffb^k(M)$ is a locally finite sum of compositions of up to $k$ b-vector fields.
  \item The space $\Diffeb^{k,(\alpha_0,\alpha_1)}(M)=\rho_0^{-\alpha_0}\rho_1^{-\alpha_1}\Diffeb^k(M)$ is defined analogously, with eb-vector fields replacing b-vector fields.
  \item The space $\Diff_{\ebop;\bop}^{m;k}(M)$ consists of all locally finite sums of operators of the form $P^\flat P^\sharp$ where $P^\flat\in\Diffeb^m(M)$, $P^\sharp\in\Diffb^k(M)$.
  \end{enumerate}
\end{definition}

Assume for the moment that $M$ is compact. Fix a smooth positive b-density on $M$; in local coordinates as above, this takes the form $a(\rho_0,\rho_1,y)|\frac{\dd\rho_0}{\rho_0}\frac{\dd\rho_1}{\rho_1}\dd y|$ with $a>0$ smooth. We then denote the $L^2$ space on $M$ by $L^2_\bop(M)$.
\begin{definition}[b- and eb-Sobolev spaces]
\label{DefNDSob}
  Let $k\in\N_0$, $\alpha_0,\alpha_1\in\R$.
  \begin{enumerate}
  \item The weighted b-Sobolev space
    \[
      \Hb^{k,(\alpha_0,\alpha_1)}(M) = \rho_0^{\alpha_0}\rho_1^{\alpha_1}\Hb^k(M)
    \]
    consists of all functions $u$ of the form $u=\rho_0^{\alpha_0}\rho_1^{\alpha_1}u_0$ with $u_0\in L^2_\bop(M)$ and $P u_0\in L^2_\bop(M)$ for all $P\in\Diffb^k(M)$. Equivalently, $P u\in L^2_\bop(M)$ for all $P\in\Diffb^{k,(\alpha_0,\alpha_1)}(M)$.
  \item The weighted eb-Sobolev space $\Heb^{k,(\alpha_0,\alpha_1)}(M)=\rho_0^{\alpha_0}\rho_1^{\alpha_1}\Heb^k(M)$ is defined analogously, with $\Diffeb$ replacing $\Diffb$.\footnote{We still use the b-density though, thus $\Heb^{0,(0,0)}(M)=L^2_\bop(M)$.}
  \item Let $m\in\N_0$. The mixed eb-b-Sobolev space $H_{\ebop;\bop}^{(m;k),(\alpha_0,\alpha_1)}(M) = \rho_0^{\alpha_0}\rho_1^{\alpha_1}H_{\ebop;\bop}^{(m;k)}(M)$ consists of all functions $u$ for which $P u\in\Heb^m(M)$ for all $P\in\Diffb^{k,(\alpha_0,\alpha_1)}(M)$. (Equivalently, $P u\in L^2_\bop(M)$ for all $P\in\rho_0^{-\alpha_0}\rho_1^{-\alpha_1}\Diff_{\ebop;\bop}^{m;k}(M)$.)
  \end{enumerate}
\end{definition}

All these spaces can be given the structure of Hilbert spaces; for instance, we can equip $\Heb^1(M)$ with the squared norm $\|u\|_{\Heb^1}^2=\|u\|_{L^2_\bop}^2+\sum\|V_i u\|_{L^2_\bop}^2$, where $\{V_i\}$ is a finite set of edge-b-vector fields spanning $\Veb(M)$ over $\CI(M)$. We also note the $L^\infty$ estimate
\begin{equation}
\label{EqNDSobEmb}
  \Hb^{s,(\alpha_0,\alpha_1)}(M) \hra \rho_0^{\alpha_0}\rho_1^{\alpha_1}L^\infty(M),\qquad s>\frac{n}{2}.
\end{equation}
This follows from the standard Sobolev embedding after the change of variables $z_0=\log\rho_0$, $z_1=\log\rho_1$, which transforms $\rho_j\pa_{\rho_j}$ into $\pa_{z_j}$ and the b-density $|\tfrac{\dd\rho_0}{\rho_0}\tfrac{\dd\rho_1}{\rho_1}\dd y|$ into $|\dd z_0\,\dd z_1\,\dd y|$.

If $M$ is a manifold with boundary and $\rho\in\CI(M)$ denotes a boundary defining function, then $\Hb^{k,\alpha}(M)=\rho^\alpha\Hb^k(M)$ is defined completely analogously (with respect to a smooth b-density). In the setting of Definition~\ref{DefNDSob}, this allows us to define spaces such as $\Hb^{k,\alpha}(H_1)$.

Lastly, if $M$ is noncompact and equipped with a smooth positive b-density, the spaces $H_{\bop,\loc}^k(M)$ consist of distributions which upon multiplication with elements of $\CIc(M)$ lie in $L^2_\bop(M)$ together with all their derivatives along all $P\in\Diffb^k(M)$; weighted spaces, eb-Sobolev spaces, and mixed edge-b;b-Sobolev spaces are defined analogously. If $\Omega\Subset M$ is an open set with compact closure, then we define
\[
  H_\ebop^k(\Omega) = \{ u|_\Omega \colon u\in H_{\ebop,\loc}^k(M) \};
\]
it can be given the structure of a Hilbert space as before. We analogously define
\[
  H_{\ebop;\bop}^{(m;k),(\alpha_0,\alpha_1)}(\Omega).
\]

Finally, we introduce the following general notation:
\begin{definition}[Operators with generalized coefficients]
\label{DefNDOpCoeff}
  If $\cF\subset\sD'(M^\circ)$ is a linear subspace of the space of distributions on $M^\circ$ and $\cD$ denotes a space of differential operators on $M$ with smooth coefficients, then $\cF\cD$ is the space of all operators $\CI(M^\circ)\to\sD'(M^\circ)$ of locally finite linear combinations $\sum a_i P_i$, where $a_i\in\cF$ and $P_i\in\cD$.
\end{definition}
Examples of interest in the present paper are spaces such as $\Hb^{k,(\alpha_0,\alpha_1)}\Diffeb^2(M)$.

\subsection{Spacetime manifold; basic energy estimate}
\label{SsNM}

\begin{definition}[Schwarzschild spacetime]
\label{DefNMSchw}
  Let $\bhm\in\R$. The Schwarzschild spacetime is
  \[
    \R_t\times\bigl(\max(0,2\bhm),\infty\bigr)_r\times\Sph^2,\qquad
    g_\bhm = -\Bigl(1-\frac{2\bhm}{r}\Bigr)\dd t^2 + \Bigl(1-\frac{2\bhm}{r}\Bigr)^{-1}\dd r^2 + r^2\slg.
  \]
  The Regge--Wheeler tortoise coordinate $r_*$ and the null coordinates $x^0,x^1$ are defined by
  \begin{equation}
  \label{EqNMx0x1}
    r_* := r + 2\bhm\log(r-2\bhm),\qquad
    x^0=t+r_*,\quad x^1=t-r_*.
  \end{equation}
\end{definition}

\begin{lemma}[Compactification of the far field of the Schwarzschild spacetime]
\label{LemmaNMSchwCausal}
  Define
  \begin{equation}
  \label{EqNMCoord}
    \rho_0 := \frac{1}{r_*-t},\quad
    \rho_{\!\scri} := \frac{r_*-t}{r},\quad
    x_{\!\scri}:=\rho_{\!\scri}^{1/2},\qquad
    \rho:=\rho_0\rho_{\!\scri}=r^{-1}.
  \end{equation}
  \begin{enumerate}
  \item Put $\bar x_{\!\scri}:=\min\bigl(\sqrt{\frac{3}{2}},\tfrac{1}{\sqrt{8\bhm}}\bigr)$ when $\bhm>0$ and $\bar x_{\!\scri}:=\sqrt{3/2}$ when $\bhm\leq 0$. Then the map $(\rho_0,x_{\!\scri},\omega)\to(t,r_*,\omega)$ with domain $M^\circ=(0,2)_{\rho_0}\times(0,\bar x_{\!\scri})_{x_{\!\scri}}\times\Sph^2$, where
  \begin{equation}
  \label{EqNMMfd}
    M := [0,2)_{\rho_0}\times[0,\bar x_{\!\scri})_{x_{\!\scri}}\times\Sph^2,
  \end{equation}
  is a diffeomorphism onto its image.
  \item Denoting the pullback of $g_\bhm$ to $M$ by $g_\bhm$ still, the hypersurface $x_{\!\scri}^{-1}(c)\subset M$ is spacelike for all $c\in(0,\bar x_{\!\scri})$, and the hypersurface $\rho_0^{-1}(c)$ is lightlike for all $c\in(0,1)$.
  \end{enumerate}
\end{lemma}
\begin{proof}
  We have $r=\rho_0^{-1}\rho_{\!\scri}^{-1}=\rho_0^{-1}x_{\!\scri}^{-2}>\half x_{\!\scri}^{-2}>4\bhm$, hence $r_*=r+2\bhm\log(r-2\bhm)$ is well-defined, and we then have $t=r_*-\rho_0^{-1}$. This proves the first part. For the second part, we record that in the coordinates~\eqref{EqNMx0x1}, the Schwarzschild metric reads
  \begin{equation}
  \label{EqNMSchw}
    g_\bhm = -\Bigl(1-\frac{2\bhm}{r}\Bigr)\dd x^0\,\dd x^1 + r^2\slg,\qquad
    g_\bhm^{-1} = -4\Bigl(1-\frac{2\bhm}{r}\Bigr)^{-1}\pa_0\otimes_s\pa_1 + r^{-2}\slg^{-1}.
  \end{equation}
  Thus, $\rho_0^{-2}\,\dd\rho_0=-\dd(r_*-t)=\dd x^1$ is null indeed. Furthermore,
  \begin{align*}
    r\,\dd\rho_{\!\scri} &= -r\,\dd\Bigl(\frac{x^1}{r}\Bigr) = -\dd x^1+\frac{x^1}{r}\dd r = -\dd x^1+\frac{x^1}{r}\Bigl(1-\frac{2\bhm}{r}\Bigr)\dd r_* \\
      &= \frac{x^1}{2 r}\Bigl(1-\frac{2\bhm}{r}\Bigr)\dd x^0 - \Bigl(1+\frac{x^1}{2 r}\Bigl(1-\frac{2\bhm}{r}\Bigr)\Bigr)\dd x^1,
  \end{align*}
  the inner product of which with itself is
  \begin{equation}
  \label{EqNMdxIEst}
    g_\bhm^{-1}(r\,\dd\rho_{\!\scri},r\,\dd\rho_{\!\scri}) = \frac{2 x^1}{r}\Bigl(1+\frac{x^1}{2 r}\Bigl(1-\frac{2\bhm}{r}\Bigr)\Bigr) < 0.
  \end{equation}
  Indeed, $\tfrac{x^1}{r}=-x_{\!\scri}^2<0$; and the second factor is positive, too, since $|\tfrac{x^1}{2 r}|=\half x_{\!\scri}^2<\tfrac34<1$.
\end{proof}

The upper bound $\rho_0<2$ can be increased arbitrarily; choosing a larger upper merely places a stronger restriction on $\bar x_{\!\scri}$.\footnote{The bound $\bar x_{\!\scri}\leq\sqrt{3/2}$ can be relaxed to any number less than $\sqrt 2$---this still ensures that $M$ is disjoint from a neighborhood of past null infinity in the blow-up of the Penrose diagram of the Schwarzschild spacetime at spacelike infinity $i^0$ (see Figure~\ref{FigIMain}); indeed, in $t,r$ coordinates, the level set $x_{\!\scri}^{-1}(c)$ is $\cO(\log r)$-close to $t=(1-c^2)r$..}

\begin{definition}[Ideal boundaries]
\label{DefNM}
  The boundary hypersurfaces of the spacetime manifold $M$ defined by~\eqref{EqNMMfd} are denoted $I^0:=\rho_0^{-1}(0)$ (\emph{blown-up spacelike infinity}) and $\scri^+ := x_{\!\scri}^{-1}(0)$ (\emph{null infinity}). Moreover, $\scri^+$ is fibered by the projection $\scri^+=[0,\infty)_{\rho_0}\times\Sph^2\to\Sph^2$.
\end{definition}

In the context of Definition~\ref{DefNDVf}, the boundary hypersurfaces of $M$ are $H_0:=I^0$ and $H_1:=\scri^+$, with defining functions $\rho_0$ and $\rho_1:=x_{\!\scri}=\rho_{\!\scri}^{1/2}$, respectively. Thus, the space $\Veb(M)$ is spanned over $\CI(M)$ by
\begin{equation}
\label{EqNMeb}
  \rho_0\pa_{\rho_0},\qquad
  x_{\!\scri}\pa_{x_{\!\scri}} = 2\rho_{\!\scri}\pa_{\rho_{\!\scri}},\qquad
  x_{\!\scri}\Omega = \rho_{\!\scri}^{1/2}\Omega,
\end{equation}
where $\Omega$ ranges over all vector fields on $\Sph^2$. (These are, roughly, the spacetime scaling vector field, the weighted outgoing null vector field, and weighted spherical vector field.)

\begin{rmk}[Comparison of function spaces]
\label{RmkNMIdent}
  The Sobolev space $H_\scri^1$ of \cite[Definition~4.1]{HintzVasyMink4} is the same as $\Heb^1(M)$ (upon restricting to functions with compact support in $M$) in view of~\eqref{EqNMeb}; moreover $H_{\scri,\bop}^{1,k}=H_{\ebop;\bop}^{(1,k)}(M)$. Moreover, in the notation of \cite{HintzVasyMink4}, if we adjoin $\rho_{\!\scri}^{1/2}=x_{\!\scri}$ to the smooth structure of the spacetime manifold $M$ there, smooth sections of the bundle $S^2\,{}^\beta T M+\rho_{\!\scri} S^2\,\Tb M$ in \cite[Equation~(4.17)]{HintzVasyMink4} are the same as smooth sections of $S^2\,\Teb M$ for the manifold $M$ in~\eqref{EqNMMfd}.
\end{rmk}

For later use, we compute
\begin{align}
\label{EqNMPa01}
\begin{split}
  \pa_0\equiv \pa_{x^0} &= \frac12(\pa_t+\pa_{r_*}) = -\frac12\rho_0\rho_{\!\scri}^2\Bigl(1-\frac{2\bhm}{r}\Bigr)\pa_{\rho_{\!\scri}} \in \rho_0\rho_{\!\scri}\Veb(M), \\
  \pa_1\equiv \pa_{x^1} &= \frac12(\pa_t-\pa_{r_*}) = \rho_0\Bigl(\rho_0\pa_{\rho_0}-\Bigl(1-\frac12\rho_{\!\scri}\Bigl(1-\frac{2\bhm}{r}\Bigr)\Bigr)\rho_{\!\scri}\pa_{\rho_{\!\scri}}\Bigr) \in \rho_0\Veb(M),
\end{split} \\
\label{EqNMPa01r}
  \pa_0 r&=\frac12\pa_{r_*}r=\frac12\Bigl(1-\frac{2\bhm}{r}\Bigr)=-\pa_1 r.
\end{align}

\begin{rmk}[b-regularity]
\label{RmkNBeb}
  From~\eqref{EqNMPa01}, we also obtain
  \[
    \rho_{\!\scri}\pa_{\rho_{\!\scri}} = -r\Bigl(1-\frac{2\bhm}{r}\Bigr)^{-1}(\pa_t+\pa_{r_*}),\qquad
    \rho_0\pa_{\rho_0} = (r_*-t)\pa_t + \rho_{\!\scri}\pa_{\rho_{\!\scri}}.
  \]
  Membership in $\Hb^k(M)$ is thus equivalent to the condition that up to $k$ derivatives along $r(\pa_t+\pa_{r_*})$, $(r_*-t)(\pa_t-\pa_{r_*})$, $\cV(\Sph^2)$ lie in $L^2_\bop(M)$. (These vector fields were already used by Lindblad \cite{LindbladAsymptotics} and in \cite{HintzVasyMink4}.) Thus, unlike in~\eqref{EqIVf}, the spherical vector fields do not have a decaying weight at $\scri^+$ anymore.
\end{rmk}

\begin{definition}[Rescaled vector bundle]
\label{DefNMVb}
  The vector bundle\footnote{This is equal to pullback of the scattering cotangent bundle on a suitable radial compactification of $\R^4$ to the blow-up of the future light cone at infinity, denoted $\beta^*S^2\,\Tsc^*\ol{\R^4}$ in~\cite{HintzVasyMink4}.} $\wt T^*M\to M$ is defined by
  \[
    \wt T^*M := \la\dd x^0\ra \oplus \la\dd x^1\ra \oplus T^*\Sph^2.
  \]
  Over the interior $M^\circ$, we identify $\wt T^*_{M^\circ}M\cong T^*_{M^\circ}M$ by identifying a section $(\omega_0,\omega_1,\slomega)$ of $\wt T^*M$ with the 1-form $\omega_0\,\dd x^0+\omega_1\,\dd x^1+r\slomega$ on $M^\circ$.
\end{definition}

Likewise, we identify sections of $S^2\wt T^*M$ with symmetric 2-tensors over $M^\circ$. In order to make the scaling of spherical tensors apparent, we thus write the bundle splittings as
\begin{equation}
\label{EqNMVbSplit}
\begin{split}
  \wt T^*M &= \la \dd x^0\ra \oplus \la \dd x^1\ra \oplus r T^*\Sph^2, \\
  S^2\wt T^*M &= \la (\dd x^0)^2\ra \oplus \la 2 \dd x^0\dd x^1 \ra \oplus (2 \dd x^0\otimes_s r T^*\Sph^2) \\
    &\quad\qquad \oplus \la(\dd x^1)^2\ra \oplus (2 \dd x^1\otimes_s r T^*\Sph^2) \oplus \la r^2\slg\ra \oplus r^2\ker\sltr,
\end{split}
\end{equation}
by direct analogy with~\eqref{EqLSplit}. Choosing local coordinates $x^2,x^3$ on $\Sph^2$, a smooth section $\omega\in\CI(M;\wt T^*M)$ is thus a linear combination $\omega = \omega_0\,\dd x^0 + \omega_1\,\dd x^1 + \sum_{a=2}^3 \omega_{\bar a}r\,\dd x^a$, $\omega_0,\omega_1,\omega_{\bar 2},\omega_{\bar 3}\in\CI(M)$. More generally, we use the following index notation:

\begin{definition}[Weights of spherical indices]
\label{DefNMSphInd}
  For $\mu_1,\ldots,\mu_p\in\{0,1,2,3\}$, set
  \[
    s(\mu_1,\ldots,\mu_p) := \#\bigl\{ i\in\{1,\ldots,p\} \colon \mu_i\in\{2,3\} \bigr\}.
  \]
  With $x^2,x^3$ denoting coordinates on $\Sph^2$, and for a tensor $T$ on $M^\circ$ of type $(p,q)$, we set
  \[
    T_{\bar\mu_1\dots\bar\mu_p}^{\bar\nu_1\dots\bar\nu_q} := r^{s(\nu_1,\ldots,\nu_q)-s(\mu_1,\ldots,\mu_p)}T_{\mu_1\dots\mu_p}^{\nu_1\dots\nu_q}.
  \]
  We shall henceforth denote indices in $\{0,1,2,3\}$ by Greek letters $\mu,\nu,\kappa,\ldots$, and spherical indices in $\{2,3\}$ by Roman letters $a,b,c,\ldots$.
\end{definition}

Returning to metrics on $M$, we have, directly from the definitions~\eqref{EqNMSchw} and~\eqref{EqNMVbSplit}:

\begin{lemma}[Uniform behavior of $g_\bhm$]
\label{LemmaNMSchwMet}
  We have $g_\bhm\in\CI(M;S^2\wt T^*M)$, and $g_\bhm$ is a nondegenerate section with Lorentzian signature down to $I^0\cup\scri^+$.
\end{lemma}

While $S^2\wt T^*M$ is the appropriate bundle for the unknown in the Einstein equations---the metric---to take values in, the metric also determines the linearized operators we need to study; hence, we need to connect $g_\bhm$ and its perturbations to the eb-theory in which our (energy) estimates will take place.

\begin{lemma}[Relationship between $\wt T^*M$ and $\Teb^*M$]
\label{LemmaNMBundleRel}
  Let $\slomega\in\CI(\Sph^2;T^*\Sph^2)$. Then
  \[
    \dd x^0\in\rho_0^{-1}x_{\!\scri}^{-2}\CI(M;\Teb^*M),\ 
    \dd x^1\in\rho_0^{-1}\CI(M;\Teb^*M),\ 
    r\slomega\in\rho_0^{-1}x_{\!\scri}^{-1}\CI(M;\Teb^*M).
  \]
  Writing $\CI=\CI(M;S^2\,\Teb^*M)$ for brevity, we have, for $\slh\in\CI(\Sph^2;S^2 T^*\Sph^2)$,
  \begin{equation}
  \label{EqNMBundleRel}
  \begin{aligned}
    (\dd x^0)^2 &\in \rho_0^{-2}x_{\!\scri}^{-4}\CI, &\quad
    \dd x^0\otimes_s\dd x^1 &\in \rho_0^{-2}x_{\!\scri}^{-2}\CI, &\quad
    \dd x^0\otimes_s r\slomega &\in \rho_0^{-2}x_{\!\scri}^{-3}\CI, \\
    (\dd x^1)^2 &\in \rho_0^{-2}\CI, &\quad
    \dd x^1\otimes_s r\slomega &\in \rho_0^{-2}x_{\!\scri}^{-1}\CI, &\quad
    r^2\slh &\in \rho_0^{-2}x_{\!\scri}^{-2}\CI.
  \end{aligned}
  \end{equation}
\end{lemma}
\begin{proof}
  We only need to prove the first part. It follows by duality from~\eqref{EqNMPa01} and the fact that if $V\in\cV(\Sph^2)$, then $r^{-1}V=\rho_0 x_{\!\scri}\cdot x_{\!\scri} V\in\rho_0 x_{\!\scri}\Veb(M)$.
\end{proof}

\begin{cor}[$g_\bhm$ as an eb-metric]
\label{CorNMLot}
  We have $g_\bhm\in\rho_0^{-2}x_{\!\scri}^{-2}\CI(M;S^2\,\Teb^*M)$ and $g_\bhm^{-1}\in\rho_0^2 x_{\!\scri}^2\CI(M;S^2\,\Teb M)$. Moreover,
  \begin{align*}
    \rho_0^2 x_{\!\scri}^2 g_\bhm &\equiv 2\Bigl(\frac{\dd\rho_0}{\rho_0}\Bigr)^2 + 2\frac{\dd\rho_0}{\rho_0} \otimes_s \frac{\dd\rho_{\!\scri}}{\rho_{\!\scri}}+x_{\!\scri}^{-2}\slg \bmod x_{\!\scri}\CI(M;S^2\,\Teb^*M), \\
    \rho_0^{-2}x_{\!\scri}^{-2}g_\bhm^{-1} &\equiv 2(\rho_0\pa_{\rho_0}-\rho_{\!\scri}\pa_{\rho_{\!\scri}})\otimes_s\rho_{\!\scri}\pa_{\rho_{\!\scri}} + x_{\!\scri}^2\slg^{-1} \bmod x_{\!\scri}\CI(M;S^2\,\Teb M).
  \end{align*}
\end{cor}

\begin{rmk}[Connection of weighted eb-metrics]
\label{RmkNMKoszul}
  The Koszul formula, together with the fact that $\Veb(M)$ is a Lie algebra (with differentiation along any element of $\Veb(M)$ being bounded on $\rho_0^{\ell_0}x_{\!\scri}^{\ell_{\!\scri}}\CI(M)$ for any $\ell_0,\ell_{\!\scri}\in\R$), implies that the Levi-Civita connection of $g_\bhm\in\rho_0^{-2}x_{\!\scri}^{-2}\CI(M;S^2\,\Teb^*M)$ satisfies $\nabla \in \Diffeb^1(M;\Teb M,\Teb^*M\otimes\Teb M)$. Writing eb-tensor bundles as $\Teb^{p,q}M=\Teb^*M^{\otimes p}\otimes\Teb M^{\otimes q}$, this gives
  \[
    \nabla \in \Diffeb^1(M;\Teb^{p,q}M,\Teb^{p+1,q}M).
  \]
  In particular, the tensor wave operator satisfies
  \begin{equation}
  \label{EqNMKoszul}
    \Box_{g_\bhm}\in\rho_0^2 x_{\!\scri}^2\Diffeb^2(M;\Teb^{p,q}M),\qquad p,q\in\N_0,
  \end{equation}
\end{rmk}

\begin{lemma}[$\Box_{g_\bhm}$ as an eb-operator; commutators]
\label{LemmaNMOp}
  Consider $\Box_{g_\bhm}$ acting on functions. The operator $L:=\rho_{\!\scri}\rho^{-3}\Box_{g_\bhm}\rho\in\Diffeb^2(M)$ is equal to
  \begin{equation}
  \label{EqNMOp}
    L = -2\rho_{\!\scri}\pa_{\rho_{\!\scri}}(\rho_0\pa_{\rho_0}-\rho_{\!\scri}\pa_{\rho_{\!\scri}}) + x_{\!\scri}^2\slDelta + \tilde L,\qquad \tilde L\in x_{\!\scri}\Diffeb^2(M).
  \end{equation}
  If $\Omega\in\cV(\Sph^2)$ is a spherical vector field, then\footnote{In fact, we have $[L,\rho_0\pa_{\rho_0}],[L,\Omega]\in x_{\!\scri}\Diffeb^2(M)$ when $\Omega$ is a rotation vector field.}
  \begin{equation}
  \label{EqNMOpComm}
    [L,\rho_0\pa_{\rho_0}],\ [L,\rho_{\!\scri}\pa_{\rho_{\!\scri}}],\ [L,\Omega] \in x_{\!\scri}\Diff_{\ebop;\bop}^{1,1}(M).
  \end{equation}
\end{lemma}
\begin{proof}
  The membership $L\in\Diffeb^2(M)$ is an immediate consequence of~\eqref{EqNMKoszul}, and will be confirmed here by a direct calculation. The expression for $L\bmod x_{\!\scri}\Diffeb^2(M)$ only depends on $g_\bhm$ modulo $\rho_0^{-2}x_{\!\scri}^{-1}\CI(M;S^2\,\Teb^*M)$; we may thus replace $g_\bhm^{-1}$ by the Minkowski dual metric $\ubar g^{-1}=-4\pa_0\otimes_s\pa_1+r^{-2}\slg^{-1}$, cf.\ \eqref{EqLMinkMet}, for which, in view of~\eqref{EqNMPa01}--\eqref{EqNMPa01r},
  \begin{equation}
  \label{EqNMOpMink}
  \begin{split}
    \Box_{\ubar g} &= 2 r^{-2}\pa_0 r^2\pa_1 + 2 r^{-2}\pa_1 r^2\pa_0 + r^{-2}\slDelta \\
      &\equiv -2\rho_0^2\rho_{\!\scri}(\rho_{\!\scri}\pa_{\rho_{\!\scri}}-1)(\rho_0\pa_{\rho_0}-\rho_{\!\scri}\pa_{\rho_{\!\scri}}) + \rho_0^2\rho_{\!\scri}^2\slDelta
  \end{split}
  \end{equation}
  modulo $\rho_0^2 x_{\!\scri}^3\Diffeb^2(M)$. Multiplying this on the left by $\rho_{\!\scri}\rho^{-3}=\rho_{\!\scri}^{-2}\rho_0^{-3}$ and on the right by $\rho=\rho_0\rho_{\!\scri}$ proves~\eqref{EqNMOp}.

  The expression~\eqref{EqNMOp} together with $\rho_0\pa_{\rho_0}\in\Veb(M)$ immediately gives $[L,\rho_0\pa_{\rho_0}]=[\tilde L,\rho_0\pa_{\rho_0}]\in x_{\!\scri}\Diffeb^2(M)$ since $\Veb(M)$ is a Lie algebra. Similarly,
  \[
    [L,\rho_{\!\scri}\pa_{\rho_{\!\scri}}] = -\rho_{\!\scri}\slDelta + [\tilde L,\rho_{\!\scri}\pa_{\rho_{\!\scri}}],
  \]
  with the commutator lying in $x_{\!\scri}\Diffeb^2(M)$; in the first term on the other hand, we can write $\slDelta$ as a finite sum $\slDelta=\sum_k \Omega_{k,1}\Omega_{k,2} + \Omega^\flat + f$ with spherical vector fields $\Omega_{k,1},\Omega_{k,2},\Omega^\flat\in\cV(\Sph^2)$ and $f\in\CI(\Sph^2)$; but $\cV(\Sph^2)\subset x_{\!\scri}^{-1}\Veb(M)\cap\Vb(M)$, and hence
  \begin{equation}
  \label{EqNMOpslDel}
  \begin{split}
    \rho_{\!\scri}\slDelta&=x_{\!\scri}\cdot\sum_k (x_{\!\scri}\Omega_{k,1})\Omega_{k,2} + x_{\!\scri}\cdot x_{\!\scri}\Omega^\flat + x_{\!\scri}^2 f \\
      &\in x_{\!\scri}\Diff_{\ebop;\bop}^{1,1}(M) + x_{\!\scri}\Diffeb^1(M) + x_{\!\scri}^2\Diffeb^0(M) = x_{\!\scri}\Diff_{\ebop;\bop}^{1,1}(M).
  \end{split}
  \end{equation}
  
  Finally, we consider
  \[
    [L,\Omega]=\rho_{\!\scri}[\Omega,\slDelta]+[\tilde L,\Omega].
  \]
  The term $[\Omega,\slDelta]\in\Diff^2(\Sph^2)$ contributes $\rho_{\!\scri}[\Omega,\slDelta]\in x_{\!\scri}\Diff_{\ebop;\bop}^{1,1}(M)$ by the same argument as in~\eqref{EqNMOpslDel}. In the second term, we use the fact that lifts of vector fields from the base $\Sph^2$ of the fibration $\scri^+\to\Sph^2$ enjoy improved commutation properties with eb-differential operators, cf.\ \cite[\S5.1]{HintzVasyScrieb}. Concretely, in local coordinates~\eqref{EqNDCoord} (with $y^2,y^3$ local coordinates on $\Sph^2$), we have $\Omega=\sum_{j=2}^3\Omega^j(y)\pa_{y^j}$ where the $\Omega^j$ are smooth in $y$ (and independent of $\rho_0$); writing any $V\in\Veb(M)$ as $V=a^0\rho_0\pa_{\rho_0}+a^1\rho_{\!\scri}\pa_{\rho_{\!\scri}}+\sum_{j=2}^3 a^j x_{\!\scri}\pa_{y^j}$ with $a^\mu\in\CI(M)$, this gives
  \begin{subequations}
  \begin{equation}
  \label{EqNMWaveCommOmega}
    [\Omega,V] = (\Omega a^0)\rho_0\pa_{\rho_0} + (\Omega a^1)\rho_{\!\scri}\pa_{\rho_{\!\scri}} + x_{\!\scri} \sum_j [\Omega,a^j\pa_{y^j}] \in \Veb(M).
  \end{equation}
  (This improves over the naive expectation coming from $\Omega\in x_{\!\scri}^{-1}\Veb(M)$ that one only has $[\Omega,V]\in x_{\!\scri}^{-1}\Veb(M)$.) Using the Leibniz rule, we infer that
  \begin{equation}
  \label{EqNMWaveCommOmega2}
    [-,\Omega] \colon \Diffeb^{2,(\alpha_0,\alpha_1)}(M) \to \Diffeb^{2,(\alpha_0,\alpha_1)}(M),\qquad \alpha_0,\alpha_1\in\R.
  \end{equation}
  \end{subequations}
  Applying this to $\tilde L\in x_{\!\scri}\Diffeb^2(M)$ gives $[\tilde L,\Omega]\in x_{\!\scri}\Diffeb^2(M)$. This finishes the proof.
\end{proof}

\begin{prop}[Energy estimate]
\label{PropNMWave}
  In the notation~\eqref{EqNMMfd}, let $c<\bar x_{\!\scri}$. Define
  \[
    \Omega = \{ x_{\!\scri}<c, \rho_0<1 \} \subset M,\quad
    \Sigma = x_{\!\scri}^{-1}(c) \subset M.
  \]
  Let $k\in\N_0$. Let $\alpha_0,\alpha_{\!\scri}\in\R$ with $\alpha_{\!\scri}<\min(\alpha_0,0)$. Suppose $f\in\Hb^{k,(\alpha_0,2\alpha_{\!\scri})}(\Omega)=\rho_0^{\alpha_0}x_{\!\scri}^{2\alpha_{\!\scri}}\Hb^k(\Omega)=\rho_0^{\alpha_0}\rho_{\!\scri}^{\alpha_{\!\scri}}\Hb^k(\Omega)$ vanishes near $\Sigma$. Then the unique forward solution $u$ (i.e.\ with vanishing Cauchy data at $\Sigma$) of
  \begin{equation}
  \label{EqNMWaveEq}
    L u=f,\qquad L=\rho_{\!\scri}\rho^{-3}\Box_{g_\bhm}\rho,
  \end{equation}
  satisfies $u\in H_{\ebop;\bop}^{(1;k),(\alpha_0,2\alpha_{\!\scri})}(\Omega)$, with an estimate\footnote{Thus, $u$ gains $1$ eb-derivative relative to $f$; and $u$ inherits the full amount of b-regularity from $f$.}
  \begin{equation}
  \label{EqNMWaveEst}
    \|u\|_{H_{\ebop;\bop}^{(1;k),(\alpha_0,2\alpha_{\!\scri})}(\Omega)}\leq C\|f\|_{\Hb^{k,(\alpha_0,2\alpha_{\!\scri})}(\Omega)} = C\|f\|_{H_{\ebop;\bop}^{(0;k),(\alpha_0,2\alpha_{\!\scri})}(\Omega)}.
  \end{equation}
\end{prop}
\begin{proof}
  We follow the arguments used in the proof of \cite[Propositions~4.3 and 4.8]{HintzVasyMink4} and shall thus be brief. While one can work directly with $\Box_{g_\bhm}$ (as done in~\cite[\S6]{HintzVasyScrieb}), we work with $L$ in order to simplify the weight arithmetic. Note now that $\Box_{g_\bhm}$ is symmetric with respect to the volume density
  \[
    |\dd g_\bhm|\in\rho_0^{-4}x_{\!\scri}^{-4}\CI(M;|\Lambda^4\,\Teb^*M|) = \rho_0^{-4}\rho_{\!\scri}^{-3}\CI(M;|\Lambda^4\,\Tb^*M|),
  \]
  where we use Lemma~\ref{LemmaNMBundleRel} and the relationship $|\tfrac{\dd\rho_0}{\rho_0}\tfrac{\dd\rho_{\!\scri}}{\rho_{\!\scri}}\tfrac{\dd\slg}{x_{\!\scri}^{\dim\Sph^2}}|=\rho_{\!\scri}^{-1}|\tfrac{\dd\rho_0}{\rho_0}\tfrac{\dd\rho_{\!\scri}}{\rho_{\!\scri}}\dd\slg|$ between smooth nonzero eb- and b-densities. Since $\Box_{g_\bhm}$ is formally self-adjoint on $L^2(M;|\dd g_\bhm|)$, the operator $L$ is formally self-adjoint on $L^2_\bop(M):=L^2(M,\mu_\bop)$, where $\mu_\bop=\rho_0^4\rho_{\!\scri}^3|\dd g_\bhm|$ is a smooth positive b-density on $M$.

  In order to prove~\eqref{EqNMWaveEst}, one can cut and paste energy estimates using domain of dependence properties. Away from $M^\circ$, the estimate~\eqref{EqNMWaveEst} estimates the $H^{k+1}$-norm of $u$ by the $H^k$-norm of $f$; it thus suffices to work near $I^0\cup\scri^+$. But away from $\scri^+$, $L$ is a b-differential operator, $L\in\Diffb^2(M\setminus\scri^+)$, for which $x_{\!\scri}$ is, near $I^0\setminus\scri^+$, a (past directed) time function, the gradient of which can thus be used as a vector field multiplier giving the estimate~\eqref{EqNMWaveEst} away from $\scri^+$---this was discussed in detail in \cite[Proposition~4.3]{HintzVasyMink4}.

  We thus work in a small neighborhood of $\scri^+$, in coordinates $\rho_0,\rho_{\!\scri}$. For $k=0$, \eqref{EqNMWaveEst} follows from an energy estimate on $\Omega$ with the vector field multiplier
  \begin{equation}
  \label{EqNMWaveMult}
    W=w^2 V,\qquad w:=\rho_0^{-\alpha_0}\rho_{\!\scri}^{-\alpha_{\!\scri}},\quad V=-(1+c)\rho_{\!\scri}\pa_{\rho_{\!\scri}}+\rho_0\pa_{\rho_0}\in\Veb(M),
  \end{equation}
  with $c>0$ (one may take $c=1$);\footnote{See \cite[Lemma~4.4]{HintzVasyMink4}, and also \cite[Proof of Theorem~6.4]{HintzVasyScrieb} where a positive multiple of $V$ is used (with a different value of $c$).} $W$ is future timelike. Consider the $L^2_\bop(\Omega)$-pairing
  \begin{equation}
  \label{EqNMWaveComm}
  \begin{split}
    &2\Re\la w L u, w V u\ra = \la Q u,u\ra + \text{[boundary terms]}, \\
    &\qquad Q:=L^*W+W^*L = [L,W]-(\dv_{\mu_\bop}W)L\in\rho_0^{-2\alpha_0}x_{\!\scri}^{-4\alpha_{\!\scri}}\Diffeb^2(M).
  \end{split}
  \end{equation}
  We compute the principal symbol of $Q$ at $\scri^+$. Since
  \begin{equation}
  \label{EqNMWaveDensity}
    \mu_\bop\equiv\ubar\mu_\bop\bmod x_{\!\scri}\CI(M;|\Lambda^4\,\Tb^*M|),\qquad \ubar\mu_\bop:=\Bigl|\frac{\dd\rho_0}{\rho_0}\frac{\dd\rho_{\!\scri}}{\rho_{\!\scri}}\dd\slg\Bigr|,
  \end{equation}
  we may replace $L$ by its leading order term in~\eqref{EqNMOp}, and $\mu_\bop$ by $\ubar\mu_\bop$. Thus,
  \[
    {-}\dv_{\ubar\mu_\bop}W = w^2 V+(w^2 V)^* = -[V,w^2] = w^2\bigl(-2(1+c)\alpha_{\!\scri}+2\alpha_0\bigr),
  \]
  and a quick calculation then gives $Q\equiv\ubar Q\bmod\rho_0^{-2\alpha_0}x_{\!\scri}^{-4\alpha_{\!\scri}+1}\Diffeb^2(M)$ for
  \begin{equation}
  \label{EqNMWaveubarQ}
  \begin{split}
    \ubar Q &= w^2\Bigl(-4\alpha_{\!\scri}(\rho_0 D_{\rho_0}-\rho_{\!\scri} D_{\rho_{\!\scri}})^2 + 4 c(\alpha_0-\alpha_{\!\scri})(\rho_{\!\scri} D_{\rho_{\!\scri}})^2 \\
      &\quad\hspace{10em} + \bigl(1+2(\alpha_0-\alpha_{\!\scri})+c(1-2\alpha_{\!\scri})\bigr)x_{\!\scri}^2\slDelta \Bigr).
  \end{split}
  \end{equation}
  Since $\alpha_{\!\scri}<0$ and $c>0$, $\alpha_0-\alpha_{\!\scri}>0$, and recalling that $\slDelta=-\sltr\slnabla^2$, this is a positive elliptic element of $\rho_0^{-2\alpha_0}x_{\!\scri}^{-4\alpha_{\!\scri}}\Diffeb^2(M)$. Therefore, $\la\ubar Q u,u\ra$ controls one eb-derivative of $u$ in $\rho_0^{\alpha_0}\rho_{\!\scri}^{\alpha_{\!\scri}}L^2_\bop$. Using a Poincar\'e inequality to control $\|u\|_{\rho_0^{\alpha_0}\rho_{\!\scri}^{\alpha_{\!\scri}}L^2_\bop}$ by $\|\rho_{\!\scri}\pa_{\rho_{\!\scri}}u\|_{\rho_0^{\alpha_0}\rho_{\!\scri}^{\alpha_{\!\scri}}L^2_\bop}$ (using $\alpha_{\!\scri}<0$), an application of the Cauchy--Schwarz inequality to~\eqref{EqNMWaveComm} (and using the fact that the boundary terms vanish at $\Sigma$ and have a good sign at $\rho_0^{-1}(1)$ due to the future causal nature of $W$ and $-\dd\rho_0$, and can thus be dropped) implies the estimate~\eqref{EqNMWaveEst} for $k=0$.

  We prove higher b-regularity by commuting the vector fields from Lemma~\ref{LemmaNMOp} through the equation. Suppose we have established~\eqref{EqNMWaveEst} for $k\in\N_0$. Let $f\in\Hb^{k+1,(\alpha_0,2\alpha_{\!\scri})}(\Omega)$. Let $X_0=\rho_0\pa_{\rho_0},X_1=\rho_1\pa_{\rho_1}$, and let $X_2,X_3,X_4\in\cV(\Sph^2)$ be vector fields spanning $T\Sph^2$ pointwise (e.g.\ rotation vector fields); put moreover $X_5\equiv 1$. Thus, the $X_k$ span $\Diffb^1(M)$ over $\CI(M)$. By Lemma~\ref{LemmaNMOp}, we can write
  \[
    [L,X_j] = \sum_{k=0}^5 x_{\!\scri} Y_{j k}X_k,\qquad Y_{j k}\in\Diffeb^1(M).
  \]
  Applying the inductive hypothesis to $L(X_j u)=X_j L u+[L,X_j]u$ gives
  \[
    \|X_j u\|_{H_{\ebop;\bop}^{(1;k),(\alpha_0,2\alpha_{\!\scri})}} \leq C\Bigl(\|X_j f\|_{\Hb^{k,(\alpha_0,2\alpha_{\!\scri})}} + \sum_k \| x_{\!\scri} X_k u \|_{H_{\ebop;\bop}^{(1;k),(\alpha_0,2\alpha_{\!\scri})}}\Bigr)
  \]
  Summing these estimates over $j=0,\ldots,5$, the sums over $k$ on the right can be absorbed into the sum over $j$ on the left, provided we localize to a neighborhood of $\scri^+$ where $x_{\!\scri}$ is small. This gives~\eqref{EqNMWaveEst} for $k+1$ in place of $k$, and completes the proof.
\end{proof}

\begin{rmk}[Multiplier in $(t,r_*)$-coordinates]
\label{RmkNDWaveMult}
  In the coordinates $t<r_*$ and modulo irrelevant lower order terms, the vector field multiplier~\eqref{EqNMWaveMult} is (using Remark~\ref{RmkNBeb})
  \[
    W = r^{2\alpha_{\!\scri}}(r_*-t)^{2(\alpha_0-\alpha_{\!\scri})}\bigl((r_*-t)\pa_t + c r(\pa_t+\pa_{r_*})\bigr).
  \]
\end{rmk}

\subsection{Metric perturbations}
\label{SsNP}

Following the rough discussion of couplings of metric coefficients in~\S\ref{SL}, we now define the function space for metric perturbations of the Schwarzschild metric $g_\bhm$ in~\eqref{EqNMSchw} near spacelike and null infinity.

\begin{definition}[Projections to subbundles]
\label{DefNPProj}
  The projections $\pi^\cC,\pi^\Ups,\slpi_0,\pi_{1 1}\colon S^2\wt T^*M\to S^2\wt T^*M$ respecting the splitting~\eqref{EqNMVbSplit} are defined as in~\eqref{EqLProj}. (In the notation of Definition~\usref{DefNMSphInd}, $\pi^\cC(h)=(h_{0 0},h_{0\bar a},\half\slg^{a b}h_{\bar a\bar b})$, $\pi^\Ups(h)=(h_{0 1},h_{1\bar a})$, $\slpi_0(h)=(h-\half\slg^{a b}h_{\bar a\bar b}\slg)$, and $\pi_{1 1}(h)=(h_{1 1})$.) We set $\pi^{\cC\Ups}:=\pi^\cC+\pi^\Ups$ (mapping $h\mapsto(h_{0 0},h_{0 1},h_{0\bar a},h_{1\bar a},\half\slg^{a b}h_{\bar a\bar b})$).
\end{definition}

\begin{definition}[Metric perturbations]
\label{DefNPMetrics}
  Let $\ell_0,\ell_{\!\scri}\in\R$ with $\ell_{\!\scri}<\min(-\gamma^\Ups,\ell_0,\half)$, and let $k\in\N$, $k\geq 3$. With $\bar x_{\!\scri}$ as in~\eqref{EqNMMfd}, fix $c\in(0,\bar x_{\!\scri})$ and put
  \begin{equation}
  \label{EqNPDomain}
    \Omega = \Bigl\{ x_{\!\scri}<c,\ \rho_0<1+\frac12\rho_{\!\scri}^{\ell_{\!\scri}} \Bigr\}.
  \end{equation}
  Then the space $\tilde\sG^{k,(\ell_0,\ell_{\!\scri})}\subset\Hb^{k,(\ell_0,-1)}(\Omega;S^2\wt T^*M)$ consists of all $h$ for which there exist
  \[
    h_{1 1}^{(0)} \in \Hb^{k,\ell_0}(\Omega\cap\scri^+; \pi_{1 1}S^2\wt T^*M),\qquad
    \slh^{(0)} \in \Hb^{k,\ell_0}(\Omega\cap\scri^+; \slpi_0 S^2\wt T^*M),
  \]
  so that $\pi^{\cC\Ups} h$, $\slpi_0 h - \slh^{(0)}$, $\pi_{1 1}h-h_{1 1}^{(0)} \in \Hb^{k,(\ell_0,2 \ell_{\!\scri})}(\Omega;S^2\wt T^*M)$. The norm on $\tilde\sG^{k,(\ell_0,\ell_{\!\scri})}$ is
  \begin{align*}
    \|h\|_{\tilde\sG^{k,(\ell_0,\ell_{\!\scri})}} &:= \|\slh^{(0)}\|_{\Hb^{k,\ell_0}(\Omega\cap\scri^+)} + \|h_{1 1}^{(0)}\|_{\Hb^{k,\ell_0}(\Omega\cap\scri^+)} \\
      &\quad + \|\slpi_0 h-\slh^{(0)}\|_{\Hb^{k,(\ell_0,2\ell_{\!\scri})}(\Omega)} + \|\pi_{1 1}h-h_{1 1}^{(0)}\|_{\Hb^{k,(\ell_0,2\ell_{\!\scri})}(\Omega)} + \|\pi^{\cC\Ups}h\|_{\Hb^{k,(\ell_0,2\ell_{\!\scri})}(\Omega)}.
  \end{align*}
  We finally define the affine space
  \[
    \sG^{k,(\ell_0,\ell_{\!\scri})} := \{ g_\bhm+r^{-1}h \colon h\in\tilde\sG^{k,(\ell_0,\ell_{\!\scri})} \}.
  \]
\end{definition}

Matching the model~\eqref{EqIModel3} with $\slpi_0 h,\pi_{1 1}h,\pi^\cC h,\pi^\Ups h$ corresponding to $\phi_1,\phi_2,\phi_3,\phi_4$, the trace-free spherical tensor $\slpi_0 h$ has a radiation field, and the leading order term of $\pi_{1 1}h$ at $\scri^+$ is sourced by it. See Remark~\ref{RmkNYBondi} for an interpretation of these terms.

\begin{rmk}
\label{RmkNPMink4}
  Note that all components of $h$, modulo the leading order terms of $\slpi_0 h$ and $\pi_{1 1}h$, have the same decay rates at $\scri^+$. This is a significant simplification of \cite[Definition~3.1]{HintzVasyMink4}, made possible by the absence of logarithmic terms in $\pi_{1 1}h$ due to our new choice of gauge (cf.\ by contrast the logarithmic coupling term $B_{h,1 1}$ in \cite[Equation~(3.26c)]{HintzVasyMink4}).
\end{rmk}

\begin{figure}[!ht]
\centering
\includegraphics{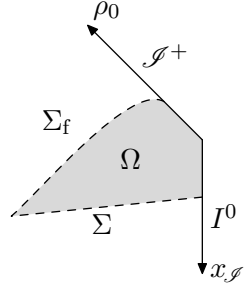}
\caption{The domain $\Omega$ defined in~\eqref{EqNPDomain} inside the manifold $M$ defined in~\eqref{EqNMMfd}. Also shown are the spacelike hypersurfaces $\Sigma,\Sigma_{\rm f}$ from Lemma~\ref{LemmaNPCausal}.  }
\label{FigNP}
\end{figure}

\begin{notation}[Remainder space]
\label{NotNPSpace}
  For $k\in\N_0$, $\alpha,\beta\in\R$, we shall use the abbreviation
  \[
    \cO_k^{\alpha,\beta} := \Hb^{k,(\alpha,2\beta)}(\Omega).
  \]
\end{notation}

The factor of $2$ in the $\scri^+$-weight is included so that $\beta$ measures the decay rate in $\rho_{\!\scri}$, as $\cO_3^{\alpha,\beta}\hra\rho_0^\alpha x_{\!\scri}^{2\beta}L^\infty(\Omega)=\rho_0^\alpha\rho_{\!\scri}^\beta L^\infty(\Omega)$. We shall repeatedly use that for $k\geq 3$,
\begin{alignat*}{3}
  u_1\in\cO_k^{\alpha_1,\beta_1},\ &u_2\in\cO_k^{\alpha_2,\beta_2}&\ &\Longrightarrow\ u_1 u_2\in\cO_k^{\alpha_1+\alpha_2,\beta_1+\beta_2}, \\
    &u_2\in\Hb^{k,\alpha_2}(\scri^+) &\ &\Longrightarrow\ u_1 u_2 \in \cO_k^{\alpha_1+\alpha_2,\beta_1}.
\end{alignat*}

\begin{lemma}[Metric coefficients]
\label{LemmaNPCoeff}
  Let $k\geq 3$ and $h\in\tilde\sG^{k,(\ell_0,\ell_{\!\scri})}$. Suppose $\|h\|_{\tilde\sG^{3,(\ell_0,\ell_{\!\scri})}}$ is sufficiently small. Then $g=g_\bhm+r^{-1}h$ is a Lorentzian metric on $\Omega^\circ$. In the notation of Definition~\usref{DefNMSphInd}, we have
  \begin{alignat*}{3}
    g_{0 0}&\in \cO_k^{1+\ell_0,1+\ell_{\!\scri}}, &\qquad g_{0 1}&\in-\frac12+\bhm r^{-1}+\cO_k^{1+\ell_0,1+\ell_{\!\scri}}, &\qquad g_{0 b}&\in\cO_k^{\ell_0,\ell_{\!\scri}}, \\
    g_{1 1}&=r^{-1}h_{1 1}, &\qquad g_{1 b}&\in\cO_k^{\ell_0,\ell_{\!\scri}}, &\qquad g_{a b}&=r^2\slg_{a b}+r h_{\bar a\bar b},
  \end{alignat*}
  and the coefficients of the dual metric $g^{-1}$ are
  \begin{alignat*}{2}
    g^{0 0} &\in -4 r^{-1}h_{1 1}+\cO_k^{1+\ell_0,1+\ell_{\!\scri}}, &\qquad
    g^{0 1} &\in -2-4\bhm r^{-1}+\cO_k^{1+\ell_0,1+\ell_{\!\scri}}, \\
    g^{0 b} &\in \cO_k^{2+\ell_0,2+\ell_{\!\scri}}, &\qquad
    g^{1 1} &\in \cO_k^{1+\ell_0,1+\ell_{\!\scri}}, \\
    g^{1 b} &\in \cO_k^{2+\ell_0,2+\ell_{\!\scri}}, &\qquad
    g^{a b} &\in r^{-2}\slg^{a b}-r^{-3}h^{\bar a\bar b}+\cO_k^{3+\ell_0,3+\ell_{\!\scri}}.
  \end{alignat*}
  As a symmetric eb-2-tensor, $r^{-1}h$ is a decaying perturbation of $g_\bhm$ (cf.\ Corollary~\usref{CorNMLot}):
  \begin{equation}
  \label{EqNPCoeffeb}
    g - g_\bhm \in \rho_0^{-2}x_{\!\scri}^{-2}\Hb^{k,(1+\ell_0,2\ell_{\!\scri})}(\Omega;S^2\,\Teb^*M).
  \end{equation}
\end{lemma}
\begin{proof}
  Sobolev embedding~\eqref{EqNDSobEmb} implies the pointwise bound $|h_{\bar\mu\bar\nu}| \leq C\|h\|_{\tilde\sG^{3,(\ell_0,\ell_{\!\scri})}} \rho_0^{\ell_0}$. The Lorentzian nature of $g$ then follows for small $h$ from the nondegenerate Lorentzian nature of $g_\bhm$ as a section of $S^2\wt T^*M$ (see Lemma~\ref{LemmaNMSchwMet}). The expressions for the inverse metric follow from~\eqref{EqNMSchw} by working in the bundle $S^2\wt T^*M$ and writing $g^{-1}=g_\bhm^{-1}-r^{-1}g_\bhm^{-1}h g_\bhm^{-1}+r^{-2}E(h)$, where $E(h)$ vanishes quadratically at $h=0$, so $E(h)\in\cO_k^{2\ell_0,0-}$ since $h\in\cO_k^{\ell_0,0-}$. This gives
  \[
    g^{0 0} \in -r^{-1}g_\bhm^{0 1}h_{1 1}g_\bhm^{0 1} + \cO_k^{2+2\ell_0,2-} \subset -4 r^{-1}h_{1 1} + r^{-2}\CI(M)\cdot\cO_k^{\ell_0,0-} + \cO_k^{2+2\ell_0,2-},
  \]
  similarly for the other coefficients.

  The statement~\eqref{EqNPCoeffeb} follows from~\eqref{EqNMBundleRel}; for instance, this gives
  \begin{align*}
    r^{-1}h_{0 0}(\dd x^0)^2 &\in \cO_k^{1+\ell_0,1+\ell_{\!\scri}} \cdot \rho_0^{-2}\rho_{\!\scri}^{-2}\CI(\Omega;S^2\,\Teb^*M) \\
      &\subset \rho_0^{-2}x_{\!\scri}^{-2}\Hb^{k,(1+\ell_0,2\ell_{\!\scri})}(\Omega;S^2\,\Teb^*M).\qedhere
  \end{align*}
\end{proof}

\begin{lemma}[Causal nature of $\pa\Omega$]
\label{LemmaNPCausal}
  For $h\in\tilde\sG^{3,(\ell_0,\ell_{\!\scri})}$ with sufficiently small norm,
  \[
    \Sigma = \Bigl\{ x_{\!\scri}=c,\ \rho_0<1+\frac12\rho_{\!\scri}^{\ell_{\!\scri}} \Bigr\}\quad \text{and}\quad
    \Sigma_{\rm f} = \Bigl\{ x_{\!\scri}<c,\ \rho_0=1+\frac12\rho_{\!\scri}^{\ell_{\!\scri}} \Bigr\}
  \]
  are spacelike hypersurfaces for $g=g_\bhm+r^{-1}h$.
\end{lemma}
\begin{proof}
  We recall from the proof of Lemma~\ref{LemmaNMSchwCausal} that
  \begin{equation}
  \label{EqNPCausala0}
    r\,\dd\rho_{\!\scri}=a_0\,\dd x^0-a_1\,\dd x^1,\qquad
    a_0:=\frac{x^1}{2 r}\Bigl(1-\frac{2\bhm}{r}\Bigr)<-\theta\rho_{\!\scri},\quad a_1:=1+a_0,
  \end{equation}
  for some $\theta>0$; note here that $\tfrac{x^1}{r}=-\rho_{\!\scri}$. The expression~\eqref{EqNMdxIEst} gives an upper bound $g_\bhm^{-1}(r\,\dd\rho_{\!\scri},r\,\dd\rho_{\!\scri})\leq-\theta\rho_{\!\scri}$ with $\theta>0$; since $a_0$ and $a_1$ are bounded, we have
  \[
    \bigl|(g^{-1}-g_\bhm^{-1})(r\,\dd\rho_{\!\scri},r\,\dd\rho_{\!\scri})\bigr| \leq C r^{-1}\|h\|_{\tilde\sG^{3,(\ell_0,\ell_{\!\scri})}} \leq \frac12\theta\rho_{\!\scri}
  \]
  on $\Omega$ for small $h$. Therefore, $\dd x_{\!\scri}$ is (past) timelike for $g$.

  For $\Sigma_{\rm f}$, we compute for the differential of its defining function, using $\rho_0=-(x^1)^{-1}$,
  \begin{align*}
    2\rho_{\!\scri}^{1-\ell_{\!\scri}}r\,\dd\Bigl(\rho_0-\frac12\rho_{\!\scri}^{\ell_{\!\scri}}\Bigr) &= 2\rho_{\!\scri}^{1-\ell_{\!\scri}}\Bigl(\rho_0\rho_{\!\scri}^{-1}\,\dd x^1 - \frac12\ell_{\!\scri}\rho_{\!\scri}^{\ell_{\!\scri}-1}\cdot r\,\dd\rho_{\!\scri}\Bigr) \\
      &= -\ell_{\!\scri} a_0\,\dd x^0 + \bigl(2\rho_0\rho_{\!\scri}^{-\ell_{\!\scri}}+\ell_{\!\scri} a_1\bigr)\,\dd x^1.
  \end{align*}
  The squared length of the 1-form with respect to $g^{-1}$ is
  \[
    (4+\cO(r^{-1}))\ell_{\!\scri} a_0 (2 \rho_0\rho_{\!\scri}^{-\ell_{\!\scri}}+\ell_{\!\scri} a_1) + \cO(r^{-1}) \cdot \bigl( \cO(\rho_{\!\scri}^2) \cO(\rho_0^{\ell_0}) + \cO(\rho_0^2\rho_{\!\scri}^{-2\ell_{\!\scri}}+1)\cO(\rho_0^{\ell_0}\rho_{\!\scri}^{\ell_{\!\scri}})\bigr).
  \]
  For small $h$, the first term is positive, and in view of~\eqref{EqNPCausala0} dominates the second term (collecting the contributions from $h^{0 0}$, $h^{1 1}$) which is of size $\cO(\|h\|_{\tilde\sG^{3,(\ell_0,\ell_{\!\scri})}}\rho_0^{1+\ell_0}\rho_{\!\scri}^{1-\ell_{\!\scri}})$. 
\end{proof}

\begin{lemma}[Connection coefficients]
\label{LemmaNEGamma}
  Let $g\in\sG^{k,(\ell_0,\ell_{\!\scri})}$, with $g-g_\bhm\in\tilde\sG^{3,(\ell_0,\ell_{\!\scri})}$ small. Then the Christoffel symbols of the first kind $\Gamma_{\kappa\mu\nu}=\half(\pa_\mu g_{\nu\kappa}+\pa_\nu g_{\mu\kappa}-\pa_\kappa g_{\mu\nu})$ are\footnote{Since the asymptotics and decay rates of the metric coefficients here are stronger than those in \cite{HintzVasyMink4}, the expressions here and in Corollary~\ref{CorNERiem} can also be read off from those in \cite[\S{A.2}]{HintzVasyMink4}; many terms, due to the stronger metric asymptotics and weaker error spaces here, can be regarded as error terms. Note that our signature convention for $g$ is different from the reference, which causes a number of sign switches.}
  \begin{alignat*}{2}
    \Gamma_{0 0 0}&\in\cO_{k-1}^{2{+}\ell_0,1{+}\ell_{\!\scri}}, &\quad
    \Gamma_{0 0 1}&\in\cO_{k-1}^{2{+}\ell_0,1{+}\ell_{\!\scri}}, \\
    \Gamma_{1 0 0}&\in{-}\half\bhm r^{-2}{+}\cO_{k-1}^{2{+}\ell_0,1{+}\ell_{\!\scri}}, &\quad
    \Gamma_{1 0 1}&\in\cO_{k-1}^{2{+}\ell_0,1{+}\ell_{\!\scri}}, \\
    \Gamma_{c 0 0}&\in\cO_{k-1}^{1{+}\ell_0,\ell_{\!\scri}}, &\quad
    \Gamma_{c 0 1}&\in\cO_{k-1}^{1{+}\ell_0,\ell_{\!\scri}}, \\
    \Gamma_{0 0 b}&\in\cO_{k-1}^{1{+}\ell_0,\ell_{\!\scri}}, &\quad
    \Gamma_{0 1 1}&\in\half\bhm r^{-2}{+}\cO_{k-1}^{2{+}\ell_0,1{+}\ell_{\!\scri}}, \\
    \Gamma_{1 0 b}&\in\cO_{k-1}^{1{+}\ell_0,\ell_{\!\scri}}, &\quad
    \Gamma_{1 1 1}&\in\half r^{-1}\pa_1 h_{1 1}{+}\cO_{k-1}^{2{+}\ell_0,1{+}\ell_{\!\scri}}, \\
    \Gamma_{c 0 b}&\in\half(r{-}2\bhm)\slg_{b c}{+}\cO_{k-1}^{\ell_0,{-}1{+}\ell_{\!\scri}}, &\quad
    \Gamma_{c 1 1}&\in\cO_{k-1}^{1{+}\ell_0,\ell_{\!\scri}}, \\
    \Gamma_{0 1 b}&\in\cO_{k-1}^{1{+}\ell_0,\ell_{\!\scri}}, &\quad
    \Gamma_{0 a b}&\in{-}\half(r{-}2\bhm)\slg_{a b}{+}\cO_{k-1}^{\ell_0,{-}1{+}\ell_{\!\scri}}, \\
    \Gamma_{1 1 b}&\in\cO_{k-1}^{1{+}\ell_0,\ell_{\!\scri}}, &\quad
    \Gamma_{1 a b}&\in\half(r{-}2\bhm)\slg_{a b}{-}\half r\pa_1 h_{\bar a\bar b}{+}\cO_{k-1}^{\ell_0,{-}1{+}\ell_{\!\scri}}, \\
    \Gamma_{c 1 b}&\in{-}\half(r{-}2\bhm)\slg_{b c}{+}\half r\pa_1 h_{\bar b\bar c}{+}\cO_{k-1}^{\ell_0,{-}1{+}\ell_{\!\scri}}, &\qquad \Gamma_{c a b}&\in r^2\slGamma_{c a b}{+}\cO_{k-1}^{{-}1{+}\ell_0,{-}2{+}\ell_{\!\scri}}.
  \end{alignat*}
  The Christoffel symbols of the second kind, $\Gamma^\kappa_{\mu\nu}=g^{\kappa\lambda}\Gamma_{\lambda\mu\nu}$, are
  \begin{alignat*}{2}
    \Gamma^0_{0 0} &\in \bhm r^{-2}{+}\cO_{k-1}^{2{+}\ell_0,1{+}\ell_{\!\scri}}, &\quad
    \Gamma^0_{0 1} &\in \cO_{k-1}^{2{+}\ell_0,1{+}\ell_{\!\scri}}, \\
    \Gamma^1_{0 0} &\in \cO_{k-1}^{2{+}\ell_0,1{+}\ell_{\!\scri}}, &\quad
    \Gamma^1_{0 1} &\in \cO_{k-1}^{2{+}\ell_0,1{+}\ell_{\!\scri}}, \\
    \Gamma^c_{0 0} &\in \cO_{k-1}^{3{+}\ell_0,2{+}\ell_{\!\scri}}, &\quad
    \Gamma^c_{0 1} &\in \cO_{k-1}^{3{+}\ell_0,2{+}\ell_{\!\scri}}, \\
    \Gamma^0_{0 b} &\in \cO_{k-1}^{1{+}\ell_0,\ell_{\!\scri}}, &\quad
    \Gamma^0_{1 1} &\in {-}r^{-1}\pa_1 h_{1 1}{+}\cO_{k-1}^{2{+}\ell_0,1{+}\ell_{\!\scri}}, \\
    \Gamma^1_{0 b} &\in \cO_{k-1}^{1{+}\ell_0,\ell_{\!\scri}}, &\quad
    \Gamma^1_{1 1} &\in {-}\bhm r^{-2}{+}\cO_{k-1}^{2{+}\ell_0,1{+}\ell_{\!\scri}}, \\
    \Gamma^c_{0 b} &\in \half r^{-1}(1{-}\tfrac{2\bhm}{r})\delta_b^c{+}\cO_{k-1}^{2{+}\ell_0,1{+}\ell_{\!\scri}}, &\quad
    \Gamma^c_{1 1}&\in \cO_{k-1}^{3{+}\ell_0,2{+}\ell_{\!\scri}}, \\
    \Gamma^0_{1 b} &\in \cO_{k-1}^{1{+}\ell_0,\ell_{\!\scri}}, &\quad
    \Gamma^0_{a b}&\in {-}r\slg_{a b}{+}r\pa_1 h_{\bar a\bar b}{+}\cO_{k-1}^{\ell_0,{-}1{+}\ell_{\!\scri}}, \\
    \Gamma^1_{1 b}&\in \cO_{k-1}^{1{+}\ell_0,\ell_{\!\scri}}, &\quad
    \Gamma^1_{a b} &\in r\slg_{a b}{+}\cO_{k-1}^{\ell_0,{-}1{+}\ell_{\!\scri}}, \\
    \Gamma^c_{1 b} &\in {-}\half r^{-1}(1{-}\tfrac{2\bhm}{r})\delta_b^c{+}\half r^{-1}\pa_1 h_{\bar b}{}^{\bar c}{+}\cO_{k-1}^{2{+}\ell_0,1{+}\ell_{\!\scri}}, &\qquad
    \Gamma^c_{a b} &\in \slGamma^c_{a b}{+}\cO_{k-1}^{1{+}\ell_0,\ell_{\!\scri}}.
  \end{alignat*}
\end{lemma}
\begin{proof}
  Direct computation using Lemma~\ref{LemmaNPCoeff} and equation~\eqref{EqNMPa01r}.
\end{proof}

\begin{cor}[Curvature coefficients]
\label{CorNERiem}
  Let $g=g_\bhm+r^{-1}h\in\sG^{k,(\ell_0,\ell_{\!\scri})}$, $k\geq 4$, with $\|h\|_{\tilde\sG^{3,(\ell_0,\ell_{\!\scri})}}$ small. Define the Riemann curvature tensor by $R^\kappa{}_{\lambda\mu\nu}=\pa_\mu\Gamma^\kappa_{\lambda\nu}-\pa_\nu\Gamma^\kappa_{\lambda\mu}+\Gamma^\kappa_{\mu\rho}\Gamma^\rho_{\lambda\nu}-\Gamma^\kappa_{\nu\rho}\Gamma^\rho_{\lambda\mu}$. Use the notation from Definition~\usref{DefNMSphInd}. Then, modulo $r^{-3}\CI+\cO_{k-2}^{3+\ell_0,1+\ell_{\!\scri}}$,
  \begin{alignat*}{2}
    R^0{}_{\bar b 1\bar d} &\equiv r^{-1}\pa_1^2 h_{\bar b\bar d}, &\qquad
      R^{\bar a}{}_{1 1\bar d} &\equiv \half r^{-1}\pa_1^2 h_{\bar d}{}^{\bar a},
  \end{alignat*}
  while $R^{\bar\kappa}{}_{\bar\lambda\bar\mu\bar\nu}\equiv 0$ for all other $\kappa,\lambda,\mu,\nu$ with $\mu<\nu$; and $R^{\bar\kappa}{}_{\bar\lambda\bar\mu\bar\nu}=-R^{\bar\kappa}{}_{\bar\lambda\bar\nu\bar\mu}$. The Ricci tensor $\Ric(g)_{\bar\lambda\bar\nu}=R^{\bar\kappa}{}_{\bar\lambda\bar\kappa\bar\nu}$ satisfies $\Ric(g)\in\cO_{k-2}^{3+\ell_0,1+\ell_{\!\scri}}$.
\end{cor}
\begin{proof}
  Direct computation. The stated membership of $R^{\bar\kappa}{}_{\bar\lambda\bar\mu\bar\nu}$ gives $\Ric(g)\in r^{-3}\CI+\cO_{k-2}^{3+\ell_0,1+\ell_{\!\scri}}$, with the $r^{-3}\CI$ term coming from $g_\bhm$ which satisfies $\Ric(g_\bhm)=0$.
\end{proof}

\subsection{Gauge-fixed Einstein operator} 
\label{SsNE}

Encouraged by the calculations in~\S\ref{SL}, we now define the nonlinear gauge-fixed Einstein operator whose linearization will be shown to have the main properties of $L_{\ubar g,E^\cC,E^\Ups}$ discussed after~\eqref{EqLOp}.

\begin{definition}[Nonlinear modified gauge-fixed Einstein operator]
\label{DefNEOp}
  Set $\cd^\cC=\cd^\Ups:=r^{-1}\,\dd t=\half r^{-1}(\dd x^0+\dd x^1)$ as in~\eqref{EqLMod}, and choose $\gamma^\cC\in(0,1)$, $\gamma^\Ups\in(-1,0)$ with $-\gamma^\Ups<\gamma^\cC$. Write $E^\bullet=(\cd^\bullet,\gamma^\bullet)$, $\bullet=\cC,\Ups$, and define $\delta_{g,E^\cC}^*,\delta_{g,E^\Ups}$ by~\eqref{EqLModCD}--\eqref{EqLModGC}. Given a Lorentzian metric $g$, and denoting by $g_\bhm$ the Schwarzschild metric from Definition~\ref{DefNMSchw}, put
  \[
    \Ups_{E^\Ups}(g;g_\bhm) := \Ups(g;g_\bhm) - (\delta_{g_\bhm,E^\Ups}-\delta_{g_\bhm})\sfG_{g_\bhm}(g-g_\bhm),
  \]
  where $\Ups(g;g_\bhm)=g(g_\bhm)^{-1}\delta_g\sfG_g g_\bhm$ as in~\eqref{EqLPUps}. We then define\footnote{The definition of $P'_{g,E^\cC,E^\Ups}$ is consistent with the motivational Definition~\ref{DefLMod} for $g=g_\bhm$, as follows from a brief calculation using Lemma~\ref{LemmaLPLin}.}
  \begin{align*}
    &P_{E^\cC,E^\Ups}(g) := \Ric(g) - \delta_{g_\bhm,E^\cC}^*\Ups_{E^\Ups}(g;g_\bhm), \\
    &\qquad P'_{g,E^\cC,E^\Ups} := D_g P_{E^\cC,E^\Ups},\qquad
      L_{g,E^\cC,E^\Ups} := 2\rho_{\!\scri}\rho^{-3}P'_{g,E^\cC,E^\Ups}\rho.
  \end{align*}
\end{definition}

\begin{lemma}[Gauge 1-form]
\label{LemmaNEUps}
  For $g$ as in Lemma~\usref{LemmaNEGamma}, we have $\Ups_{E^\Ups}(g;g_\bhm)\in\cO_{k-1}^{2+\ell_0,1+\ell_{\!\scri}}$.
\end{lemma}
\begin{proof}
  We have $\Ups(g;g_\bhm)^\mu=g^{\kappa\lambda}(\Gamma(g)_{\kappa\lambda}^\mu-\Gamma(g_\bhm)_{\kappa\lambda}^\mu)$; lowering the index using $g$ gives $\Ups(g;g_\bhm)_0\equiv-\half\Ups(g;g_\bhm)^1$ and $\Ups(g;g_\bhm)_1\equiv-\half\Ups(g;g_\bhm)^0$ modulo $\cO_{k-1}^{2+\ell_0,1+\ell_{\!\scri}}$. For $E^\Ups=(0,0,0)$, the result can now be read off from Lemma~\ref{LemmaNEGamma}. Likewise,
  \[
    \Ups_{E^\Ups}(g;g_\bhm)-\Ups(g;g_\bhm) = -(\delta_{g_\bhm,E^\Ups}-\delta_{g_\bhm})\sfG_{g_\bhm}(g-g_\bhm) \in \cO_{k-1}^{2+\ell_0,1+\ell_{\!\scri}}
  \]
   since $\delta_{g_\bhm,E^\Ups}-\delta_{g_\bhm}\in r^{-1}\CI(M;\Hom(S^2\wt T^*M,\wt T^*M))$ and $g-g_\bhm\in\cO_k^{1+\ell_0,1-}$.
\end{proof}

\begin{prop}[Structure of the linearized gauge-fixed Einstein operator]
\label{PropNELin}
  Write symmetric scattering 2-tensors in the splitting~\eqref{EqNMVbSplit}. Let $g=g_\bhm+r^{-1}h\in\sG^{k,(\ell_0,\ell_{\!\scri})}$, $k\geq 4$, with $\|h\|_{\tilde\sG^{3,(\ell_0,\ell_{\!\scri})}}$ small. Then the operator $L_{g,E^\cC,E^\Ups}$ from Definition~\usref{DefNEOp} takes the form
  \begin{subequations}
  \begin{align}
    &L_{g,E^\cC,E^\Ups} = L_{g,E^\cC,E^\Ups}^0 + \tilde L_{g,E^\cC,E^\Ups}, \nonumber\\
  \label{EqNELinMain}
    &\qquad L_{g,E^\cC,E^\Ups}^0 = -2(\rho_{\!\scri}\pa_{\rho_{\!\scri}}-A_{g,E^\cC,E^\Ups})(\rho_0\pa_{\rho_0}-\rho_{\!\scri}\pa_{\rho_{\!\scri}}) + x_{\!\scri}^2\slDelta + 2 B_{g,E^\cC,E^\Ups}, \\
  \label{EqNELinError}
    &\qquad \tilde L_{g,E^\cC,E^\Ups} \in \bigl(x_{\!\scri}\CI(\Omega) + \Hb^{k-2,(\ell_0,2\ell_{\!\scri})}(\Omega)\bigr)\Diffeb^2(M;S^2\wt T^*M),
  \end{align}
  \end{subequations}
  where the endomorphisms $A_{g,E^\cC,E^\Ups}$ and $B_{g,E^\cC,E^\Ups}$ of $S^2\wt T^*M$ are defined by
  \begin{align*}
    A_{g,E^\cC,E^\Ups} &=
      \openbigpmatrix{1pt}
        2\gamma^\cC & 0 & 0 & 0 & 0 & 0 & 0 \\
        -\gamma^\Ups & -\gamma^\Ups & 0 & 0 & 0 & 0 & 0 \\
        0 & 0 & \gamma^\cC & 0 & 0 & 0 & 0 \\
        0 & -2\gamma^\Ups & 0 & -2\gamma^\Ups & 0 & \gamma^\cC & -\half\pa_1 h^{\bar a\bar b} \\
        0 & 0 & \gamma^\cC{-}\gamma^\Ups & 0 & -\gamma^\Ups & 0 & 0 \\
        2\gamma^\cC & 0 & 0 & 0 & 0 & \gamma^\cC & 0 \\
        2\pa_1 h_{\bar a\bar b} & 0 & 0 & 0 & 0 & 0 & 0
      \closebigpmatrix, \\
    B_{g,E^\cC,E^\Ups} &=
      \begin{pmatrix}
        0 & 0 & 0 & 0 & 0 & 0 & 0 \\
        0 & 0 & 0 & 0 & 0 & 0 & 0 \\
        0 & 0 & 0 & 0 & 0 & 0 & 0 \\
        2\rho_0^{-1}\pa_1^2 h_{1 1} & 0 & 0 & 0 & 0 & 0 & 0 \\
        0 & 0 & 0 & 0 & 0 & 0 & 0 \\
        0 & 0 & 0 & 0 & 0 & 0 & 0 \\
        2\rho_0^{-1}\pa_1^2 h_{\bar a\bar b} & 0 & 0 & 0 & 0 & 0 & 0
      \end{pmatrix}.
  \end{align*}
\end{prop}

If $h=0$, then $A_{g,E^\cC,E^\Ups}$ equals $A_{E^\cC,E^\Ups}$ from~\eqref{EqLAEnd}. General $h$ contribute bounded terms at $\scri^+$ and do not affect the block triangular structure of $A_{g,E^\cC,E^\Ups}$; see~\S\ref{SsNY}.

\begin{proof}[Proof of Proposition~\usref{PropNELin}]
  We will analyze the terms in the expression
  \begin{equation}
  \label{EqNELinTerms}
  \begin{split}
    2 P'_{g,E^\cC,E^\Ups} &= \Box_g + 2(\delta_{g_\bhm,E^\cC}^*-\delta_g^*)\delta_g\sfG_g + 2\delta_{g_\bhm,E^\cC}^*(\delta_{g_\bhm,E^\Ups}-\delta_{g_\bhm})\sfG_{g_\bhm} \\
      &\qquad + 2\delta_{g_\bhm,E^\cC}^*\sC_g - 2\delta_{g_\bhm,E^\cC}^*\sY_g + 2\sR_g,
  \end{split}
  \end{equation}
  with $\sC_g$ and $\sY_g$ defined in~\eqref{EqLPLinUps}, one by one.

  \pfstep{Tensor wave operator.} Following Definition~\ref{DefNMSphInd}, we set
  \[
    \Gamma^{\bar\kappa}_{\bar\mu\bar\nu}=r^{s(\kappa)-s(\mu,\nu)}\Gamma^\kappa_{\mu\nu}, \qquad
    \Gamma_{\bar\kappa\bar\mu\bar\nu}=r^{-s(\kappa,\mu,\nu)}\Gamma_{\kappa\mu\nu}.
  \]
  By Lemma~\ref{LemmaNEGamma}, we have
  \begin{subequations}
  \begin{alignat}{3}
  \label{EqNELinGamma1}
      \Gamma^{\bar\sigma}_{0\bar\mu}&\in r^{-2}\CI+\cO_{k-1}^{2+\ell_0,1+\ell_{\!\scri}}, & & s(\sigma,\nu)<2, \\
  \label{EqNELinGamma2}
      \Gamma^{\bar\sigma}_{0\bar\mu}&\in \half r^{-1}\delta_\mu^\sigma +r^{-2}\CI+\cO_{k-1}^{2+\ell_0,1+\ell_{\!\scri}}, &\qquad & s(\sigma,\nu)=2, \\
  \label{EqNELinGammaTot}
    \Gamma^{\bar\sigma}_{\bar\kappa\bar\mu} &\in r^{-1}\CI + \cO_{k-1}^{2+\ell_0,1-} & & \forall\,\sigma,\kappa,\mu.
  \end{alignat}
  \end{subequations}
  
  Given a symmetric 2-tensor $u$ on $\Omega\subset M$, we begin by calculating the form of
  \[
    u_{\bar\mu\bar\nu;\bar\kappa} = r^{-s(\mu,\nu,\kappa)}\pa_\kappa\bigl(r^{s(\mu,\nu)}u_{\bar\mu\bar\nu}\bigr) - \Gamma^{\bar\sigma}_{\bar\kappa\bar\mu}u_{\bar\sigma\bar\nu} - \Gamma^{\bar\sigma}_{\bar\nu\bar\kappa}u_{\bar\mu\bar\sigma}.
  \]
  For $\kappa=0$, note that $r^{-s(\mu,\nu)}[\pa_0,r^{s(\mu,\nu)}]\equiv\half s(\mu,\nu) r^{-1}\bmod r^{-2}\CI$, which cancels the contribution of the leading order term of~\eqref{EqNELinGamma2}. Thus, by~\eqref{EqNMPa01},
  \begin{subequations}
  \begin{alignat}{2}
  \label{EqNELinCov10}
    u_{\bar\mu\bar\nu;0} &\in \pa_0 u_{\bar\mu\bar\nu} + \bigl(r^{-2}\CI+\cO_{k-1}^{2+\ell_0,1+\ell_{\!\scri}}\bigr)u &\ \subset\ & \bigl(\rho_0 x_{\!\scri}^2\CI+\cO_{k-1}^{2+\ell_0,1+\ell_{\!\scri}}\bigr)\Diffeb^1(M)u, \\
  \label{EqNELinCov11}
    u_{\bar\mu\bar\nu;1} &\in \pa_1 u_{\bar\mu\bar\nu} + \bigl(r^{-1}\CI+\cO_{k-1}^{2+\ell_0,1-}\bigr)u &\ \subset\ & \bigl(\rho_0\CI+\cO_{k-1}^{2+\ell_0,1-}\bigr)\Diffeb^1(M)u, \\
  \label{EqNELinCov1c}
    u_{\bar\mu\bar\nu;\bar c} &\in r^{-1}\pa_c u_{\bar\mu\bar\nu} + \bigl(r^{-1}\CI+\cO_{k-1}^{2+\ell_0,1-})u &\ \subset\ & \bigl(\rho_0 x_{\!\scri}\CI+\cO_{k-1}^{2+\ell_0,1-}\bigr)\Diffeb^1(M)u.
  \end{alignat}
  \end{subequations}
  We use this to compute the form of
  \begin{equation}
  \label{EqELinBox}
  \begin{split}
    (\Box_g u)_{\bar\mu\bar\nu} &= -r^{-s(\mu,\nu,\kappa,\lambda)}g^{\bar\kappa\bar\lambda}\pa_\lambda\bigl(r^{s(\mu,\nu,\kappa)}u_{\bar\mu\bar\nu;\bar\kappa}\bigr) \\
    &\qquad + g^{\bar\kappa\bar\lambda}\bigl(\Gamma^{\bar\sigma}_{\bar\mu\bar\lambda}u_{\bar\sigma\bar\nu;\bar\kappa} + \Gamma^{\bar\sigma}_{\bar\nu\bar\lambda}u_{\bar\mu\bar\sigma;\bar\kappa} + \Gamma^{\bar\sigma}_{\bar\kappa\bar\lambda}u_{\bar\mu\bar\nu;\bar\sigma}\bigr).
  \end{split}
  \end{equation}
  In the second line of~\eqref{EqELinBox}, those terms in which $u$ is covariantly differentiated along $\pa_0,\pa_a$ lie in $(\rho_0^2 x_{\!\scri}^3\CI+\cO_{k-2}^{3+\ell_0,3/2-})\Diffeb^1(M)u$ by~\eqref{EqNELinGammaTot}, \eqref{EqNELinCov10}, and \eqref{EqNELinCov1c} (using that multiplication by $x_{\!\scri}$ maps $\cO_{k-2}^{\alpha,1-}\to\cO_{k-2}^{\alpha,3/2-}$). Next, Lemmas~\ref{LemmaNPCoeff} and \ref{LemmaNEGamma} give $g^{\bar\kappa\bar\lambda}\Gamma^1_{\bar\kappa\bar\lambda} \in 2 r^{-1}+\cO_{k-1}^{2+\ell_0,1+\ell_{\!\scri}}$; using~\eqref{EqNELinCov11}, the terms in the second line of~\eqref{EqELinBox} involving derivatives of $u$ along $\pa_1$ are thus modulo $(r^{-2}\CI+\cO_{k-1}^{3+\ell_0,1+\ell_{\!\scri}})\Diffeb^1(M)u$ equal to
  \begin{align*}
    &g^{1 0}\Gamma^{\bar\sigma}_{\bar\mu 0}u_{\bar\sigma\bar\nu;1} + g^{1 0}\Gamma^{\bar\sigma}_{\bar\nu 0}u_{\bar\mu\bar\sigma;1} + g^{\bar\kappa\bar\lambda}\Gamma^1_{\bar\kappa\bar\lambda}u_{\bar\mu\bar\nu;1} \equiv (-s(\mu,\nu)+2)r^{-1}\pa_1 u_{\bar\mu\bar\nu}.
  \end{align*}
  For the first term on the right in~\eqref{EqELinBox}, all terms with $(\kappa,\lambda)\neq(0,1),(1,0),(a,b)$ produce terms in $\cO_{k-2}^{3+\ell_0,1+\ell_{\!\scri}}\Diffeb^2(M)u$. The remaining terms sum to
  \begin{align*}
    &2 r^{-s(\mu,\nu)}\pa_0\bigl(r^{s(\mu,\nu)}\pa_1 u_{\bar\mu\bar\nu}\bigr) + 2\pa_1\pa_0 u_{\bar\mu\bar\nu} + \bigl(r^{-2}\CI+\cO_{k-1}^{3+\ell_0,1+\ell_{\!\scri}}\bigr)\Diffeb^2(M)u \\
    &\quad\qquad - r^{-2}\slg^{a b}\pa_a\pa_b u_{\bar\mu\bar\nu} + \bigl(\rho_0^2 x_{\!\scri}^3\CI+\cO_{k-2}^{3+\ell_0,3/2-}\bigr)\Diffeb^2(M)u,
  \end{align*}
  with the first line capturing the non-spherical, the second line the spherical terms. Plugging in~\eqref{EqNMPa01} and using $2\ell_{\!\scri}<1$ (so $\rho_{\!\scri}\rho^{-3}\cO_{k-2}^{3+\ell_0,3/2-}\rho\subset\cO_{k-2}^{1+\ell_0,\ell_{\!\scri}}$), we thus obtain
  \begin{equation}
  \label{EqNELinTerm1}
  \begin{split}
    (\rho_{\!\scri}\rho^{-3}\Box_g\rho u)_{\bar\mu\bar\nu} &\in -2\rho_{\!\scri}\pa_{\rho_{\!\scri}}(\rho_0\pa_{\rho_0}-\rho_{\!\scri}\pa_{\rho_{\!\scri}})u_{\bar\mu\bar\nu} - \slg^{a b}(x_{\!\scri}\pa_a)(x_{\!\scri}\pa_b)u_{\bar\mu\bar\nu} \\
      &\quad\qquad + \bigl(x_{\!\scri}\CI+\cO_{k-2}^{1+\ell_0,\ell_{\!\scri}}\bigr)\Diffeb^2(M)u.
  \end{split}
  \end{equation}
  The coordinate derivatives $\pa_a$ on $\Sph^2$ can be replaced by covariant derivatives $\slnabla_a$, the difference in local coordinates being $x_{\!\scri}(\slnabla_a-\pa_a)x_{\!\scri}\pa_a\in x_{\!\scri}\Diffeb^1$.

  \pfstep{Modified symmetric gradient.} Next, consider the second summand in~\eqref{EqNELinTerms}. We have
  \begin{align}
    \bigl((\delta_{g_\bhm,E^\cC}^*-\delta_g^*)\omega\bigr)_{\bar\mu\bar\nu} &= \bigl((\delta_{g_\bhm,E^\cC}^*-\delta_{g_\bhm}^*)\omega\bigr)_{\bar\mu\bar\nu} + \bigl((\delta_{g_\bhm}^*-\delta_g^*)\omega\bigr)_{\bar\mu\bar\nu} \nonumber\\
  \label{EqNELinDelstarDiff}
      &=\bigl((\delta_{g_\bhm,E^\cC}^*-\delta_{g_\bhm}^*)\omega\bigr)_{\bar\mu\bar\nu} + C_{\bar\mu\bar\nu}^{\bar\kappa}\omega_{\bar\kappa},
  \end{align}
  where $C_{\bar\mu\bar\nu}^{\bar\kappa}=\Gamma(g)_{\bar\mu\bar\nu}^{\bar\kappa}-\Gamma(g_\bhm)_{\bar\mu\bar\nu}^{\bar\kappa}$. In the splittings~\eqref{EqNMVbSplit}, we have $\cd^\cC=\cd^\Ups=(\half,\half,0)$, so
  \begin{equation}
  \label{EqELinTdelDelDiff2}
    \delta_{g_\bhm,E^\cC}^*-\delta_{g_\bhm}^* \in
    \gamma^\cC r^{-1}
    \begin{pmatrix}
      1 & 0 & 0 \\
      0 & 0 & 0 \\
      0 & 0 & \half \\
      0 & 1 & 0 \\
      0 & 0 & \half \\
      1 & 1 & 0 \\
      0 & 0 & 0
    \end{pmatrix}
    +r^{-2}\CI(M;\Hom(\wt T^*M,S^2\wt T^*M)).
  \end{equation}
  For the second term in~\eqref{EqNELinDelstarDiff}, we infer from Lemma~\ref{LemmaNEGamma} that, modulo $\cO_{k-1}^{2+\ell_0,1+\ell_{\!\scri}}$, we have
  \begin{equation}
  \label{EqNELinCtensor}
    C^0_{1 1}\equiv -r^{-1}\pa_1 h_{1 1}, \qquad
    C^{\bar c}_{1\bar b}=C_{\bar b 1}^{\bar c}\equiv \half r^{-1}\pa_1 h_{\bar b}{}^{\bar c}, \qquad
    C^0_{\bar a\bar b}\equiv r^{-1}\pa_1 h_{\bar a\bar b},
  \end{equation}
  while $C^{\bar\kappa}_{\bar\mu\bar\nu}\equiv 0$ for all other $\kappa,\mu,\nu$. Using $\sltr h\in\cO_k^{\ell_0,\ell_{\!\scri}}$, the operator $\omega\mapsto (C^{\bar\kappa}_{\bar\mu\bar\nu}\omega_{\bar\kappa})$ is thus
  \begin{equation}
  \label{EqNELinC1to2}
    r^{-1}
    \begin{pmatrix}
      0 & 0 & 0 \\
      0 & 0 & 0 \\
      0 & 0 & 0 \\
      -\pa_1 h_{1 1} & 0 & 0 \\
      0 & 0 & \half\pa_1 h_{\bar a}{}^{\bar b} \\
      0 & 0 & 0 \\
      \pa_1 h_{\bar a\bar b} & 0 & 0
    \end{pmatrix}
    + \cO_{k-1}^{2+\ell_0,1+\ell_{\!\scri}}.
  \end{equation}

  We compute $(\delta_g u)_{\bar\mu}=-g^{\bar\lambda\bar\kappa}u_{\bar\mu\bar\lambda;\bar\kappa}$ using~\eqref{EqNELinCov10}--\eqref{EqNELinCov1c} and Lemma~\ref{LemmaNPCoeff}. The terms with $\kappa\neq 1$ contribute $\bigl(\rho_0 x_{\!\scri}\CI+\cO_{k-1}^{2+\ell_0,1-}\bigr)\Diffeb^1(M)u$, as do the terms with $\kappa=1$, $\lambda\neq 0$, so
  \begin{equation}
  \label{EqNELinDiv}
    \delta_g \in
      \begin{pmatrix}
        2\pa_1 & 0 & 0 & 0 & 0 & 0 & 0 \\
        0 & 2\pa_1 & 0 & 0 & 0 & 0 & 0 \\
        0 & 0 & 2\pa_1 & 0 & 0 & 0 & 0
      \end{pmatrix}
      + \bigl(\rho_0 x_{\!\scri}\CI+\cO_{k-1}^{2+\ell_0,1-}\bigr)\Diffeb^1.
  \end{equation}

  Lastly, Lemma~\ref{LemmaNPCoeff} implies
  \begin{equation}
  \label{EqNELinTrRev}
    \sfG_g \in
    \begin{pmatrix}
      1 & 0 & 0 & 0 & 0 & 0 & 0 \\
      0 & 0 & 0 & 0 & 0 & \half & 0 \\
      0 & 0 & 1 & 0 & 0 & 0 & 0 \\
      0 & 0 & 0 & 1 & 0 & 0 & 0 \\
      0 & 0 & 0 & 0 & 1 & 0 & 0 \\
      0 & 2 & 0 & 0 & 0 & 0 & 0 \\
      0 & 0 & 0 & 0 & 0 & 0 & 1
    \end{pmatrix}
    + r^{-1}\CI + \cO_k^{1+\ell_0,1-}.
  \end{equation}
  Combining~\eqref{EqNELinDelstarDiff}, \eqref{EqELinTdelDelDiff2}, and \eqref{EqNELinC1to2}--\eqref{EqNELinTrRev} gives
  \begin{equation}
  \label{EqNELinTerm2}
  \begin{split}
    &\rho_{\!\scri}\rho^{-3}\bigl(2(\delta_{g_\bhm,E^\cC}^*-\delta_g^*)\delta_g\sfG_g\bigr)\rho \\
    &\quad\in
      2\begin{pmatrix}
        2\gamma^\cC & 0 & 0 & 0 & 0 & 0 & 0 \\
        0 & 0 & 0 & 0 & 0 & 0 & 0 \\
        0 & 0 & \gamma^\cC & 0 & 0 & 0 & 0 \\
        -2\pa_1 h_{1 1} & 0 & 0 & 0 & 0 & \gamma^\cC & 0 \\
        0 & 0 & \gamma^\cC+\pa_1 h_{\bar a}{}^{\bar b} & 0 & 0 & 0 & 0 \\
        2\gamma^\cC & 0 & 0 & 0 & 0 & \gamma^\cC & 0 \\
        2\pa_1 h_{\bar a\bar b} & 0 & 0 & 0 & 0 & 0 & 0
      \end{pmatrix}\rho_0^{-1}\pa_1 \\
      &\quad\qquad + \bigl(x_{\!\scri}\CI+\cO_{k-1}^{1+\ell_0,\ell_{\!\scri}}\bigr)\Diffeb^1.
  \end{split}
  \end{equation}

  \pfstep{Modified divergence.} Using Lemma~\ref{LemmaNEGamma} with $h=0$, the third summand in~\eqref{EqNELinTerms} is
  \begin{align}
  \label{EqNELinDelstar}
    \delta_{g_\bhm,E^\cC}^* &\in \begin{pmatrix} 0 & 0 & 0 \\ \half & 0 & 0 \\ 0 & 0 & 0 \\ 0 & 1 & 0 \\ 0 & 0 & \half \\ 0 & 0 & 0 \\ 0 & 0 & 0 \end{pmatrix}\pa_1 + \rho_0 x_{\!\scri}\Diffeb^1, \\
  \label{EqNELinDelstar2}
    \delta_{g_\bhm,E^\Ups}-\delta_{g_\bhm} &\in \gamma^\Ups r^{-1}\begin{pmatrix} -2 & 0 & 0 & 0 & 0 & -1 & 0 \\ 0 & 0 & 0 & -2 & 0 & -1 & 0 \\ 0 & 0 & -2 & 0 & -2 & 0 & 0 \end{pmatrix} + r^{-2}\CI.
  \end{align}
  Therefore,
  \begin{align}
    &\rho_{\!\scri}\rho^{-3}\bigl(2\delta_{g_\bhm,E^\cC}^*(\delta_{g_\bhm,E^\Ups}-\delta_{g_\bhm})\sfG_{g_\bhm}\bigr)\rho \nonumber\\
  \label{EqNELinTerm3}
    &\quad \in 2
     \begin{pmatrix}
       0 & 0 & 0 & 0 & 0 & 0 & 0 \\
       -\gamma^\Ups & -\gamma^\Ups & 0 & 0 & 0 & 0 & 0 \\
       0 & 0 & 0 & 0 & 0 & 0 & 0 \\
       0 & -2\gamma^\Ups & 0 & -2\gamma^\Ups & 0 & 0 & 0 \\
       0 & 0 & -\gamma^\Ups & 0 & -\gamma^\Ups & 0 & 0 \\
       0 & 0 & 0 & 0 & 0 & 0 & 0 \\
       0 & 0 & 0 & 0 & 0 & 0 & 0
     \end{pmatrix}
     (\rho_0\pa_{\rho_0}-\rho_{\!\scri}\pa_{\rho_{\!\scri}}) \\
    &\quad\qquad + \bigl(x_{\!\scri}\CI+\cO_{k-1}^{1+\ell_0,\ell_{\!\scri}}\bigr)\Diffeb^1. \nonumber
  \end{align}

  \pfstep{Term involving $\sC_g$.} We turn to the fourth summand in~\eqref{EqNELinTerms}. When calculating $(\sC_g u)_{\bar\kappa}=g_{\bar\kappa\bar\lambda}g^{\bar\mu\bar\sigma}g^{\bar\nu\bar\tau}C^{\bar\lambda}_{\bar\mu\bar\nu} u_{\bar\sigma\bar\tau}$, one can replace $g\in g_\bhm+\cO_k^{1+\ell_0,1-}$ by $g_\bhm$ at the expense of an error term in $\cO_{k-1}^{3+2\ell_0,2-}$ since $C_{\bar\mu\bar\nu}^{\bar\lambda}\in\cO_{k-1}^{2+\ell_0,1-}$ (cf.\ \eqref{EqNELinCtensor}); furthermore, the components of the tensor $C$ other than those in~\eqref{EqNELinCtensor} contribute terms in $\cO_{k-1}^{2+\ell_0,1+\ell_{\!\scri}}$. Therefore,
  \[
    \sC_g =
    \begin{pmatrix}
      0 & 0 & 0 & 0 & 0 & 0 & 0 \\
      2 r^{-1}\pa_1 h_{1 1} & 0 & 0 & 0 & 0 & 0 & -\half r^{-1}\pa_1 h^{\bar a\bar b} \\
      0 & 0 & -2 r^{-1}\pa_1 h_{\bar a}{}^{\bar b} & 0 & 0 & 0 & 0
    \end{pmatrix}
    + \cO_{k-1}^{2+\ell_0,1+\ell_{\!\scri}}.
  \]
  Together with~\eqref{EqNELinDelstar}, and using again that $\ell_{\!\scri}<\half$, we thus have
  \begin{equation}
  \label{EqELinTerm4}
  \begin{split}
    2\rho_{\!\scri}\rho^{-3}\delta_{g_\bhm,E^\cC}^*\sC_g\rho &\in
      2\rho_0^{-1}\pa_1 \circ
      \begin{pmatrix}
        0 & 0 & 0 & 0 & 0 & 0 & 0 \\
        0 & 0 & 0 & 0 & 0 & 0 & 0 \\
        0 & 0 & 0 & 0 & 0 & 0 & 0 \\
        2\pa_1 h_{1 1} & 0 & 0 & 0 & 0 & 0 & -\half\pa_1 h^{\bar a\bar b} \\
        0 & 0 & -\pa_1 h_{\bar a}{}^{\bar b} & 0 & 0 & 0 & 0 \\
        0 & 0 & 0 & 0 & 0 & 0 & 0 \\
        0 & 0 & 0 & 0 & 0 & 0 & 0
      \end{pmatrix} \\
      & \qquad + \cO_{k-1}^{1+\ell_0,\ell_{\!\scri}}\Diffeb^1.
  \end{split}
  \end{equation}

  \pfstep{Term involving $\sY_g$.} For the fifth summand in~\eqref{EqNELinTerms}, note that $\delta_{g_\bhm,E^\cC}^*\in\rho_0\Diffeb^1$ by~\eqref{EqNELinDelstar}. Together with $\Ups(g;g_\bhm)^{\bar\nu}\in\cO_{k-1}^{2+\ell_0,1+\ell_{\!\scri}}$ from Lemma~\ref{LemmaNEUps}, we get
  \begin{equation}
  \label{EqNELinTerm5}
    -2\rho_{\!\scri}\rho^{-3}\delta_{g_\bhm,E^\cC}^*\sY_g\rho \in \cO_{k-1}^{1+\ell_0,\ell_{\!\scri}}\Diffeb^1.
  \end{equation}

  \pfstep{Curvature term.} The final term of~\eqref{EqNELinTerms} can be computed using Corollary~\ref{CorNERiem}. A fortiori, all components of the Riemann and Ricci tensor lie in $r^{-3}\CI+\cO_{k-2}^{3+\ell_0,1-}$, and hence replacing $g$ by $g_\bhm$ in the definition of $\sR_g$ produces $\cO_{k-2}^{4+\ell_0,2-}$ error terms. One computes
  \[
    2\rho_{\!\scri}\rho^{-3}\sR_g\rho \in 2\rho_0^{-1}
      \begin{pmatrix}
        0 & 0 & 0 & 0 & 0 & 0 & 0 \\
        0 & 0 & 0 & 0 & 0 & 0 & 0 \\
        0 & 0 & 0 & 0 & 0 & 0 & 0 \\
        0 & 0 & 0 & 0 & 0 & 0 & \half\pa_1^2 h^{\bar a\bar b} \\
        0 & 0 & \pa_1^2 h_{\bar a}{}^{\bar b} & 0 & 0 & 0 & 0 \\
        0 & 0 & 0 & 0 & 0 & 0 & 0 \\
        2\pa_1^2 h_{\bar a\bar b} & 0 & 0 & 0 & 0 & 0 & 0
      \end{pmatrix}
       +\rho_0 x_{\!\scri}^4\CI + \cO_{k-2}^{1+\ell_0,\ell_{\!\scri}}.
  \]
  Combining this with~\eqref{EqNELinTerm1}, \eqref{EqNELinTerm2}, and \eqref{EqNELinTerm3}--\eqref{EqNELinTerm5}, and recalling that $\rho_0^{-1}\pa_1\equiv\rho_0\pa_{\rho_0}-\rho_{\!\scri}\pa_{\rho_{\!\scri}}\bmod x_{\!\scri}\Diffeb^1$, proves the Proposition.
\end{proof}

\begin{definition}[Forcing terms]
\label{DefNEForc}
  For $k\in\N_0$ and $\ell_0,\ell_{\!\scri}\in\R$, $\ell_{\!\scri}>0$, we define
  \[
    \sF^{k,(\ell_0,\ell_{\!\scri})} := \bigl\{ f=\tilde f+f_{1 1}^{(0)}(\dd x^1)^2 \colon \tilde f\in\Hb^{k,(\ell_0,2\ell_{\!\scri})}(\Omega;S^2\wt T^*M),\ f_{1 1}^{(0)}\in\Hb^{k,\ell_0}(\scri^+\cap\Omega) \bigr\},
  \]
  with norm $\|f\|_{\sF^{k,(\ell_0,\ell_{\!\scri})}} := \|\tilde f\|_{\Hb^{k,(\ell_0,2\ell_{\!\scri})}(\Omega)} + \|f_{1 1}^{(0)}\|_{\Hb^{k,\ell_0}(\scri^+\cap\Omega)}$.
\end{definition}

\begin{cor}[Nonlinear error term]
\label{CorNENonlin}
  Let $g=g_\bhm+r^{-1}h\in\sG^{k,(\ell_0,\ell_{\!\scri})}$, $k\geq 4$, with $h$ small in $\tilde\sG^{3,(\ell_0,\ell_{\!\scri})}$. Then $2 \rho_{\!\scri}^{-1}\rho^3 P_{E^\cC,E^\Ups}(g)\in\sF^{k-2,(\ell_0,\ell_{\!\scri})}$; more precisely,\footnote{We can replace $h_{1 1}$ and $\slpi_0 h$ by their $\scri^+$-leading order terms $h_{1 1}^{(0)}$ and $\slh^{(0)}$, cf.\ Definition~\ref{DefNPMetrics}.}
  \begin{equation}
  \label{EqNENonlin}
    2\rho_{\!\scri}^{-1}\rho^3 P_{E^\cC,E^\Ups}(g) \in \rho_0^{-1}\Bigl(-4\gamma^\Ups\pa_1 h_{1 1}-\frac12|\pa_1\slpi_0 h|_{\slg^{-1}}^2\Bigr)(\dd x^1)^2 + \cO_{k-2}^{\ell_0,\ell_{\!\scri}}.
  \end{equation}
\end{cor}

By contrast to \cite[Lemma~3.5]{HintzVasyMink4}, it is the leading order term of $h_{1 1}$ that enters in~\eqref{EqNENonlin}, rather than a logarithmically divergent term of $h_{1 1}$. The term $|\pa_1\slpi_0 h|^2$ is captured by the term $(\pa_t\phi_1)^2$ in the equation for $\phi_2$ in~\eqref{EqIModel3} (with $\phi_1,\phi_2$ being models for $\slpi_0 h$, $h_{1 1}$).

\begin{proof}[Proof of Corollary~\usref{CorNENonlin}]
  Instead of a direct computation, we integrate up the linearization of $P_{E^\cC,E^\Ups}$: the fundamental theorem of calculus gives
  \[
    P_{E^\cC,E^\Ups}(g_\bhm+r^{-1}h) = P_{E^\cC,E^\Ups}(g_\bhm) + \int_0^1 P'_{g_\bhm+r^{-1}s h,E^\cC,E^\Ups}(r^{-1}h)\,\dd s;
  \]
  since $P_{E^\cC,E^\Ups}(g_\bhm)=0$, we can therefore use Proposition~\ref{PropNELin} to compute
  \[
    2\rho_{\!\scri}^{-1}\rho^3 P_{E^\cC,E^\Ups}(g) = \int_0^1 L_{g_\bhm+r^{-1}s h,E^\cC,E^\Ups}(h)\,\dd s
  \]
  Using Definition~\ref{DefNPMetrics} and~\eqref{EqNELinError}, the error term of the Proposition contributes
  \begin{equation}
  \label{EqNENonlinErr}
    \tilde L_{g_\bhm+r^{-1}s h,E^\cC,E^\Ups}(h) \in \cO_{k-2}^{\ell_0,\ell_{\!\scri}}.
  \end{equation}
  
  Regarding the main term~\eqref{EqNELinMain}, the contribution $x_{\!\scri}^2\slDelta h\in \rho_{\!\scri}\Diffb^2(M)h\subset\cO_{k-2}^{\ell_0,1-}$ lies, a fortiori, in the space~\eqref{EqNENonlinErr}; and $B_{g_\bhm+r^{-1}s h,E^\cC,E^\Ups}h\in\cO_k^{\ell_0,\ell_{\!\scri}}$ since $h_{0 0}\in\cO_k^{\ell_0,\ell_{\!\scri}}$. In the first term of~\eqref{EqNELinMain},
  \[
    \rho_{\!\scri}\pa_{\rho_{\!\scri}}(\rho_0\pa_{\rho_0}-\rho_{\!\scri}\pa_{\rho_{\!\scri}})h=(\rho_0\pa_{\rho_0}-\rho_{\!\scri}\pa_{\rho_{\!\scri}})\rho_{\!\scri}\pa_{\rho_{\!\scri}}h\in\cO_{k-2}^{\ell_0,\ell_{\!\scri}}
  \]
  is an error term as well since $\rho_{\!\scri}\pa_{\rho_{\!\scri}}$ annihilates the leading order terms of $h$ at $\scri^+$; thus,
  \begin{equation}
  \label{EqNENonlinCalc}
    2\rho_{\!\scri}^{-1}\rho^3 P_{E^\cC,E^\Ups}(g) \equiv 2\int_0^1 A_{g_\bhm+r^{-1}s h,E^\cC,E^\Ups}(\rho_0^{-1}\pa_1 h)\,\dd s\ \bmod \cO_{k-2}^{\ell_0,\ell_{\!\scri}}.
  \end{equation}
  All coefficients of $h$ in the splitting~\eqref{EqNMVbSplit} except for $\pi_{1 1}h$ and $\slpi_0 h$ lie in $\cO_k^{\ell_0,\ell_{\!\scri}}$ and thus contribute error terms. The $\pi_{1 1}h$, resp.\ $\slpi_0 h$ component only contributes through the $(4,4)$, resp.\ $(4,7)$ entry of $A_{g_\bhm+r^{-1} s h,E^\cC,E^\Ups}$. Therefore, only the $4$-th, i.e.\ $(\dd x^1)^2$, component of~\eqref{EqNENonlinCalc} does not lie in $\cO_{k-2}^{\ell_0,\ell_{\!\scri}}$, and modulo $\cO_{k-2}^{\ell_0,\ell_{\!\scri}}$ it equals
  \[
    2\rho_0^{-1}\int_0^1\Bigl(-2\gamma^\Ups\pa_1 h_{1 1} - \frac12\pa_1(s h^{\bar a\bar b})\pa_1 h_{\bar a\bar b}\Bigr)\,\dd s = \rho_0^{-1}\Bigl(-4\gamma^\Ups\pa_1 h_{1 1} - \frac12\pa_1 h_{\bar a\bar b}\pa_1 h^{\bar a\bar b}\Bigr).\qedhere
  \]
\end{proof}

\subsection{Tame energy estimate}
\label{SsNQ}

With the modification parameters $E^\cC,E^\Ups$ fixed as in Definition~\ref{DefNEOp}, we shall now drop them from the notation, and thus simply write
\[
  P(g) := P_{E^\cC,E^\Ups}(g),\qquad
  L_g := L_{g,E^\cC,E^\Ups},\qquad
  A_g := A_{g,E^\cC,E^\Ups},\qquad\text{etc.}
\]

The first key step is an energy estimate for the linearized operator from Definition~\ref{DefNEOp} on spaces with fixed weights but arbitrarily high b-regularity; precise decay is obtained in a second step in~\S\ref{SsNY}.

\begin{prop}[Tame energy estimate]
\label{PropNQ}
  Fix $\ell_0,\ell_{\!\scri}$ as in Definition~\usref{DefNPMetrics}, and let $g=g_\bhm+r^{-1}h\in\sG^{k,(\ell_0,\ell_{\!\scri})}$, with $h$ small in $\tilde\sG^{8,(\ell_0,\ell_{\!\scri})}$. Let $\alpha_0,\alpha_{\!\scri}\in\R$ with $\alpha_{\!\scri}<\min(\alpha_0,0)$, and let $k,m\in\N_0$ with $k\geq 8$ and $m\leq k-3$. Suppose $f\in\Hb^{m,(\alpha_0,2\alpha_{\!\scri})}(\Omega;S^2\wt T^*M)$ vanishes near $\Sigma$ (in the notation of Lemma~\usref{LemmaNPCausal}). Then the unique forward solution $u$ of
  \begin{equation}
  \label{EqNQEq}
    L_g u = f
  \end{equation}
  satisfies $u\in H_{\ebop;\bop}^{(1;m),(\alpha_0,2\alpha_{\!\scri})}(\Omega;S^2\wt T^*M)$. For $m\geq 3$, we moreover have the tame estimate
  \begin{equation}
  \label{EqNQTame}
    \|u\|_{H_{\ebop;\bop}^{(1;m),(\alpha_0,2\alpha_{\!\scri})}} \leq C\Bigl(\|f\|_{\Hb^{m,(\alpha_0,2\alpha_{\!\scri})}} + \|h\|_{\tilde\sG^{m+3,(\ell_0,\ell_{\!\scri})}}\|f\|_{\Hb^{3,(\alpha_0,2\alpha_{\!\scri})}}\Bigr),
  \end{equation}
  where $C$ depends on $m,k,\alpha_0,\alpha_{\!\scri},\ell_0,\ell_{\!\scri}$, but not on $f,h$.
\end{prop}

We shall give a proof based on elementary (and rather imprecise) considerations.

\begin{lemma}[Tame product estimate]
\label{LemmaNQProd}
  Write points $x\in\R^n=\R\times\R^{n-1}$ as $x=(x_1,x')$. Let $q\leq m\in\N_0$. For $p\in\N_0$, denote by $d_p=\lceil\frac{p+1}{2}\rceil$ the smallest integer $>\frac{p}{2}$. Then there exists a constant $C=C(m,q)$ so that for all $h\in\CIc(\R^{n-1})$ and $u\in\CIc(\R^n)$,
  \[
    \| (D^q h)(D^{m-q}u) \|_{L^2(\R^n)} \leq C\Bigl(\|h\|_{H^{d_{n-1}}(\R^{n-1})}\|u\|_{H^m(\R^n)} + \|h\|_{H^{m+d_{n-1}}(\R^{n-1})}\|u\|_{L^2(\R^n)}\Bigr).
  \]
\end{lemma}
\begin{proof}
  We repeatedly use the following estimate for integers $0\leq a<b<c$:
  \begin{equation}
  \label{EqNQProdabc}
    \| D^b u \|_{L^2} \leq C_{a b c}\|D^a u\|_{L^2}^{\frac{c-b}{c-a}}\|D^c u\|_{L^2}^{\frac{b-a}{c-a}};
  \end{equation}
  this follows by an inductive argument from the base case
  \[
    \|D u\|_{L^2}^2 = \int D(u\ol{D u})\,\dd x + \int u\ol{D^2 u}\,\dd x = \int u\ol{D^2 u}\,\dd x \leq \|u\|_{L^2}\|D^2 u\|_{L^2}.
  \]

  We then estimate, using Sobolev embedding $H^{d_{n-1}}(\R^{n-1})\hra L^\infty(\R^{n-1})$,
  \begin{align*}
    \| (D^q h)(D^{m-q}u) \|_{L^2(\R^n)}^2 &\leq \|D^q h\|_{L^\infty(\R^{n-1})} \|D^{m-q}u\|_{L^2(\R^n)} \\
      &\lesssim \|D^q h\|_{H^{d_{n-1}}(\R^{n-1})} \|D^{m-q}u\|_{L^2(\R^n)}.
  \end{align*}
  Since $\|D^q h\|_{H^{d_{n-1}}}\lesssim\|D^{q+d_{n-1}}h\|_{L^2}+\|D^q h\|_{L^2}$, we can further estimate, using~\eqref{EqNQProdabc},
  \begin{align*}
    \|D^q(D^{d_{n-1}}h)\|_{L^2}\|D^{m-q}u\|_{L^2} &\lesssim \|D^{d_{n-1}}h\|_{L^2}^{\frac{m-q}{m}}\|D^{m+d_{n-1}}h\|_{L^2}^{\frac{q}{m}} \|u\|_{L^2}^{\frac{q}{m}}\|D^m u\|_{L^2}^{\frac{m-q}{m}} \\
      &\leq \|h\|_{H^{m+d_{n-1}}} \|u\|_{L^2} + \|h\|_{H^{d_{n-1}}}\|u\|_{H^m}.
  \end{align*}
  We can estimate $\|D^q h\|_{L^2}\|D^{m-q}u\|_{L^2}$ by the same right hand side.
\end{proof}

\begin{lemma}[Commutator identity]
\label{LemmaNQComm}
  Let $\cA$ be an algebra. Let $L,X_1,\ldots,X_N\in\cA$. Write $\ad_{X_j}=[-,X_j]$. Then
  \[
    [L,X_1\cdots X_N] = \sum_{q=1}^N (-1)^{q-1}\sum_{i_1<\cdots<i_q} \bigl(\ad_{X_{i_1}}\dots\ad_{X_{i_q}} L\bigr) \prod_{\genfrac{}{}{0pt}{}{j=1}{j\neq i_1,\ldots,i_q}}^N X_j.
  \]
\end{lemma}
\begin{proof}
  The case $N=1$ is clear. The inductive step follows from $[L,X_1\cdots X_{N+1}]=[L,X_1\cdots X_N]X_{N+1}+[L,X_{N+1}]X_1\cdots X_N-[\ad_{X_{N+1}}L,X_1\cdots X_N]$.
\end{proof}

\begin{proof}[Proof of Proposition~\usref{PropNQ}]
  \pfstep{Basic energy estimate.} For fixed $\alpha_0\in\R$, we first prove the Proposition for $m=3$ and fixed but large negative $\alpha_{\!\scri}<\alpha_0$. We use the vector field multiplier $W=w^2 V$ from~\eqref{EqNMWaveMult} with, say, $c=1$, the volume density $\ubar\mu_\bop$ from~\eqref{EqNMWaveDensity}, and a pairing calculation analogous to~\eqref{EqNMWaveComm}. Using the $L^2$ inner product on sections of $S^2\wt T^*_\Omega M\to\Omega$ relative to $\ubar\mu_\bop$ and any fixed smooth, \emph{positive definite} fiber inner product on $S^2\wt T^*M$, we shall evaluate
  \begin{align*}
    &2\Re\la w L_g u,w V u\ra = \la Q u,u\ra + \text{[boundary terms]}, \\
    &\qquad Q := [L_g,W] - (\dv_{\ubar\mu_\bop}W)L_g + (L_g^*-L_g)W,
  \end{align*}
  The first two summands of $Q$ were computed to leading order at $\scri^+$ to be equal to $\ubar Q$ in~\eqref{EqNMWaveubarQ}; the point now is that for $\alpha_{\!\scri}$ sufficiently large and negative, $\ubar Q$ dominates the principal symbol of the skew-adjoint part $(L_g^*-L_g)W\in w^2 L^\infty(\Omega)\Diffeb^2$ (using Proposition~\ref{PropNELin} and Sobolev embedding) whose bound in this space only depending on $\|h\|_{\tilde\sG^{5,(\ell_0,\ell_{\!\scri})}}$.\footnote{The boundedness of $h_{1 1}$ at $\scri^+$ comes in handy here and allows for a proof of the energy estimate without the need for using the block triangular structure of $A_g$ yet, unlike in \cite[\S4.1]{HintzVasyMink4}.} Following the proof of Proposition~\ref{PropNMWave}, this gives
  \begin{equation}
  \label{EqNQTame0}
    \|u\|_{H_\ebop^{1,(\alpha_0,2\alpha_{\!\scri})}} \leq C_0\|L_g u\|_{\Hb^{0,(\alpha_0,2\alpha_{\!\scri})}},
  \end{equation}
  with $C_0$ independent of $h$ as long as $\|h\|_{\tilde\sG^{5,(\ell_0,\ell_{\!\scri})}}$ is small. One can commute any number b-derivatives through the equation $L_g u=f$ as in the proof of Proposition~\ref{PropNMWave}; we give details in a tame setting momentarily. We content ourselves with $3$ b-derivatives for now; thus, for a constant $C_3$ only depending on $\|h\|_{\tilde\sG^{8,(\ell_0,\ell_{\!\scri})}}$, we have
  \begin{equation}
  \label{EqNQTame3}
    \|u\|_{H_{\ebop;\bop}^{(1;3),(\alpha_0,2\alpha_{\!\scri})}} \leq C_3\|L_g u\|_{\Hb^{3,(\alpha_0,2\alpha_{\!\scri})}}.
  \end{equation}

  \pfstep{Tame estimate.} We shall localize to small neighborhoods of $\scri^+$ whenever convenient below; proofs of tame estimates away from $\scri^+$ follow from simplifications of the arguments below. Recall from the proof of Proposition~\ref{PropNMWave} the set of commutators
  \[
    \cX = \{ X_0,\ldots,X_5 \} \subset \Diffb^1(M;S^2\wt T^*M)
  \]
  given by $X_0=\rho_0\pa_{\rho_0}$, $X_1=\rho_1\pa_{\rho_1}$, spherical vector fields $X_2,X_3,X_4\in\cV(\Sph^2)$ (acting by covariant differentiation on spherical 1-forms and symmetric 2-tensors in the splitting~\eqref{EqNMVbSplit}), and $X_5\equiv 1$. The estimate~\eqref{EqNQTame0} and Lemma~\ref{LemmaNQComm} applied to $V_1\cdots V_l u$ for $l\leq m$ and $V_1,\ldots,V_l\in\cX$ give
  \begin{align*}
    &\|u\|_{H_{\ebop;\bop}^{(1;m),(\alpha_0,2\alpha_{\!\scri})}} \\
    &\quad \leq C_0\Bigl(\|L_g u\|_{\Hb^{m,(\alpha_0,2\alpha_{\!\scri})}} \\
    &\quad\quad + \sum_{l\leq m}\sum_{V_1,\ldots,V_l\in\cX}\sum_{q=1}^l\sum_{i_1<\cdots<i_q} \| (\ad_{V_{i_1}}\cdots\ad_{V_{i_q}} L_g) V_1\ldots\wh{V_{i_1}}\ldots\wh{V_{i_q}}\ldots V_l u\|_{\Hb^{0,(\alpha_0,2\alpha_{\!\scri})}}\Bigr),
  \end{align*}
  which we schematically write as
  \begin{equation}
  \label{EqNQTameSchema}
    \|u\|_{H_{\ebop;\bop}^{(1;m),(\alpha_0,2\alpha_{\!\scri})}} \leq C_0\Bigl(\|L_g u\|_{\Hb^{m,(\alpha_0,2\alpha_{\!\scri})}} + \sum_{q\leq l\leq m}\|(\ad_\cX^q L_g)\cX^{l-q}u\|_{\Hb^{0,(\alpha_0,2\alpha_{\!\scri})}}\Bigr).
  \end{equation}
  
  Consider first the contributions from $\tilde L_g$ to~\eqref{EqNQTameSchema}. We can write $\tilde L_g$ in~\eqref{EqNELinError} as
  \begin{equation}
  \label{EqNQTameTildeLStruct}
    \tilde L_g=\sum (x_{\!\scri} a_j+\tilde a_j)P_j,\qquad a_j\in\CI,\quad \tilde a_j\in\Hb^{k-2,(\ell_0,2\ell_{\!\scri})},
  \end{equation}
  where the operators $P_j\in\Diffeb^2(M;S^2\wt T^*M)$ span $\Diffeb^2(M;S^2\wt T^*M)$ over $\CI(M)$, and so that, for some constant $C=C(k,\ell_0,\ell_{\!\scri})$,
  \begin{equation}
  \label{EqNQTameTildeL}
    \|\tilde L_g\|_{k-2,(\ell_0,2\ell_{\!\scri})} := \max_j\|\tilde a_j\|_{\Hb^{k-2,(\ell_0,2\ell_{\!\scri})}}\leq C\|h\|_{\tilde\sG^{k,(\ell_0,\ell_{\!\scri})}};
  \end{equation}
  this uses: \begin{enumerate*} \item the coefficients of $\tilde L_g$ are rational functions of up to $2$ b-derivatives of $h$; \item for $k\geq 5$, $\Hb^{k-2}(\Omega)$ is an algebra, with a Moser estimate for the norm of products (which is a consequence of the corresponding result on $\R^n$, see e.g.\ \cite[\S13, Proposition~3.7]{TaylorPDE3}, upon passing to coordinates $\log\rho_0$ and $\log\rho_{\!\scri}$).\end{enumerate*}

  We will use the fact that commutators with elements $X_j\in\cX$ preserve the space $\Diffeb^k(M;S^2\wt T^*M)$ for all $k$; this is clear for $j=0,1,5$ (in which case $X_j$ is itself an eb-operator), and for spherical vector fields ($j=2,3,4$) relies on their $\rho_0$-independence, as discussed in~\eqref{EqNMWaveCommOmega} and \eqref{EqNMWaveCommOmega2}. Thus, for $1\leq q\leq l\leq m$, we have (using~\eqref{EqNQTameTildeLStruct})
  \begin{align}
    &\|(\ad_\cX^q\tilde L_g)\cX^{l-q}u\|_{\Hb^{0,(\alpha_0,2\alpha_{\!\scri})}} \nonumber\\
    &\quad\leq \|x_{\!\scri}\sL_\ebop^2\sL_\bop^{l-q}u\|_{\Hb^{0,(\alpha_0,2\alpha_{\!\scri})}} + \| (\sL_\bop^q\tilde a)\sL_\ebop^2 \sL_\bop^{l-q}u\|_{\Hb^{0,(\alpha_0,2\alpha_{\!\scri})}} \nonumber\\
  \label{EqNQTameTildeLest}
    &\quad\leq C_m\|u\|_{H_{\ebop;\bop}^{(1;m),(\alpha_0,2\alpha_{\!\scri}-1)}} + \| (\sL_\bop^q\tilde a)(\sL_\ebop^1\sL_\bop^{m-q+1}u)\|_{\Hb^{0,(\alpha_0,2\alpha_{\!\scri})}},
  \end{align}
  where we write $\sL_\bop^a$ and $\sL_\ebop^b$ for elements of $\Diffb^a$ and $\Diffeb^b$, respectively, whose precise forms do not matter; and in passing to the second line, we used $\Diffeb^1\subset\Diffb^1$. The second term of~\eqref{EqNQTameTildeLest} is $\leq\|\sL_\bop^q\tilde a\|_{\rho_0^{\ell_0}x_{\!\scri}^{2\ell_{\!\scri}}L^\infty}\|u\|_{H_{\ebop;\bop}^{(1;m-q+1),(\alpha_0-\ell_0,2(\alpha_{\!\scri}-\ell_{\!\scri}))}}$; due to the weaker weight at $\scri^+$, we can conclude that upon working in a sufficiently small neighborhood of $\scri^+$, this is bounded by a small constant times $\|u\|_{H_{\ebop;\bop}^{(1;m),(\alpha_0,2\alpha_{\!\scri})}}$ and can thus be absorbed into the left hand side of~\eqref{EqNQTameSchema}; similarly for the first term. To get a \emph{tame} estimate, we use~\cite[\S13, Proposition~3.6]{TaylorPDE3} and Sobolev embedding~\eqref{EqNDSobEmb} to bound the second term in~\eqref{EqNQTameTildeLest} by
  \begin{align*}
    &C_m\Bigl(\|\sL_\bop^1\tilde a\|_{\rho_0^{\ell_0}x_{\!\scri}^{2\ell_{\!\scri}}L^\infty} \|\sL_\ebop^1 u\|_{\Hb^{m,(\alpha_0-\ell_0,2(\alpha_{\!\scri}-\ell_{\!\scri}))}} + \|\sL_\bop^1\tilde a\|_{\Hb^{m,(\ell_0,2\ell_{\!\scri})}}\|\sL_\ebop^1 u\|_{\rho_0^{\alpha_0}x_{\!\scri}^{2\alpha_{\!\scri}}L^\infty}\Bigr) \\
    &\quad\leq C_m'\Bigl(\|\tilde a\|_{\Hb^{4,(\ell_0,2\ell_{\!\scri})}}\|u\|_{H_{\ebop;\bop}^{(1;m),(\alpha_0,2(\alpha_{\!\scri}-\ell_{\!\scri}))}} + \|\tilde a\|_{\Hb^{m+1,(\ell_0,2\ell_{\!\scri})}}\|u\|_{H_{\ebop;\bop}^{(1;3),(\alpha_0,2\alpha_{\!\scri})}}\Bigr) \\
    &\quad\leq C_m''\Bigl(\|h\|_{\tilde\sG^{6,(\ell_0,\ell_{\!\scri})}}\|u\|_{H_{\ebop;\bop}^{(1;m),(\alpha_0,2(\alpha_{\!\scri}-\ell_{\!\scri}))}} + \|h\|_{\tilde\sG^{m+3,(\ell_0,\ell_{\!\scri})}}\|L_g u\|_{\Hb^{3,(\alpha_0,2\alpha_{\!\scri})}}\Bigr),
  \end{align*}
  where in passing to the final line we used the estimates~\eqref{EqNQTameTildeL} and~\eqref{EqNQTame3}. The first term only involves a fixed low regularity norm of $h$, and upon localizing to a sufficiently small (only depending on $m$) neighborhood of $\scri^+$ can be absorbed into the left hand side of~\eqref{EqNQTameSchema}. The second term already fits into the estimate~\eqref{EqNQTame}.

  Next, we decompose the main term $L_g^0$ of $L_g$ in~\eqref{EqNELinMain} into $L_{g_\bhm}^0+(L_g^0-L_{g_\bhm}^0)$, with the first term capturing the smooth terms and the second term capturing the terms involving $h$. Using $[L_{g_\bhm}^0,X_i]\in x_{\!\scri}\Diff_{\ebop;\bop}^{1,1}$ as in Lemma~\ref{LemmaNMOp}, the contribution of $L_{g_\bhm}^0$ to the second summand on the right in~\eqref{EqNQTameSchema} can be estimated by $C_m\|x_{\!\scri} u\|_{H_{\ebop;\bop}^{(1;m),(\alpha_0,2\alpha_{\!\scri})}}=C_m\|u\|_{H_{\ebop;\bop}^{(1;m),(\alpha_0,2\alpha_{\!\scri}-1)}}$, which can again be absorbed into the left hand side of~\eqref{EqNQTameSchema} for small $x_{\!\scri}$.

  Turning to the term $A_h\sL_\ebop^1$ of $L_g^0-L_{g_\bhm}^0$, where $A_h:=A_g-A_{g_\bhm}$, we decompose $A_h=A_h^0+\tilde A_h$ into the $\rho_{\!\scri}$-independent leading order term $A_h^0\in\Hb^{k-1,1+\ell_0}(\scri^+\cap\Omega)$ plus a remainder term $\tilde A_h\in\Hb^{k-1,(1+\ell_0,2\ell_{\!\scri})}(\Omega)$. The contribution from $\tilde A_h$ to the second term on the right in~\eqref{EqNQTameSchema} can be treated like the contribution from $\tilde L_g$. For the contribution from $A_h^0$, which is linear in $\pa_1\slh^{(0)}$ in the notation of Definition~\ref{DefNPMetrics}, we apply Lemma~\ref{LemmaNQProd} in logarithmic coordinates $\log\rho_0,\log\rho_{\!\scri}$ (with $\log\rho_{\!\scri}$ playing the role of $x_1$ in the lemma) with $n=4$, and $h,m$ there replaced by $\sL_\bop^1\pa_1\slh^{(0)},l-1$ where $l\leq m$, so
  \begin{align*}
    &\|(\ad_\cX^q A_h^0)(\cX^{l-q}\sL_\ebop^1 u)\|_{\Hb^{0,(\alpha_0,2\alpha_{\!\scri})}} = \|(\ad_\cX^{q-1}\ad_\cX A_h^0)(\cX^{l-1-(q-1)}\sL_\ebop^1 u)\|_{\Hb^{0,(\alpha_0,2\alpha_{\!\scri})}} \\
    &\quad \leq C_m\Bigl(\|\sL_\bop^1\pa_1\slh^{(0)}\|_{\Hb^{2,1+\ell_0}(\scri^+\cap\Omega)}\|\sL_\ebop^1 u\|_{\Hb^{m-1,(\alpha_0-1-\ell_0,2\alpha_{\!\scri})}(\Omega)} \\
    &\quad\hspace{3em} + \|\sL_\bop^1\pa_1\slh^{(0)}\|_{\Hb^{m+1,1+\ell_0}(\scri^+\cap\Omega)} \|\sL_\ebop^1 u\|_{\Hb^{0,(\alpha_0-1-\ell_0,2\alpha_{\!\scri})}(\Omega)}\Bigr) \\
    &\quad \leq C_m'\Bigl( \|h\|_{\tilde\sG^{4,(\ell_0,\ell_{\!\scri})}}\|u\|_{H_{\ebop;\bop}^{(1;m-1),(\alpha_0,2\alpha_{\!\scri})}} + \|h\|_{\tilde\sG^{m+3,(\ell_0,\ell_{\!\scri})}}\|u\|_{H_\ebop^{1,(\alpha_0,2\alpha_{\!\scri})}}\Bigr),
  \end{align*}
  We can then use~\eqref{EqNQTame0} to bound the second term on the right. The contribution from $B_g$ to the right hand side of~\eqref{EqNQTameSchema} is analyzed similarly. This finishes the proof of~\eqref{EqNQTame}.

  \pfstep{Estimate with sharp weights.} $A_g$ is lower triangular in the bundle splitting $S^2\wt T^*M=\ran(\pi^{\cC\Ups})\oplus\ran\slpi_0\oplus\ran\pi_{1 1}$, with scalar diagonal entries that are independent of $h$; see~\eqref{EqNYA} below for the explicit expression.\footnote{For $g=g_\bhm$, a simpler version of this is~\eqref{EqLApiC}--\eqref{EqLApiCcompl}.} We may thus choose a positive definite fiber inner product on $S^2\wt T^*M$ with respect to which the skew-adjoint part of $A_g$ is as small as we like along $\scri^+$ in $\rho_0^{1+\ell_0}L^\infty(\scri^+\cap\Omega)$ (using only that $\|h\|_{\tilde\sG^{3,(\ell_0,\ell_{\!\scri})}}\lesssim 1$). The calculation~\eqref{EqNMWaveubarQ} thus shows that the condition $\alpha_{\!\scri}<\min(\alpha_0,0)$ suffices to obtain the estimate~\eqref{EqNQTame}.
\end{proof}

\subsection{Recovery of decay; proof of nonlinear stability}
\label{SsNY}

In the splitting $S^2\wt T^*M=\ran(\pi^\cC+\pi^\Ups)\oplus\ran\slpi_0\oplus\ran\pi_{1 1}$, the endomorphisms $A_g$ and $B_g$ from Proposition~\ref{PropNELin} are
\begin{subequations}
\begin{align}
\label{EqNYA}
  &A_g=\begin{pmatrix}
         A^{\cC\Ups} & 0 & 0 \\
         \slA^{\cC\Ups} & 0 & 0 \\
         A_{1 1}^{\cC\Ups} & \slA_{1 1} & A_{1 1}
       \end{pmatrix}, \\
  &\ A^{\cC\Ups}=\begin{pmatrix}
              2\gamma^\cC & 0 & 0 & 0 & 0 \\
              0 & \gamma^\cC & 0 & 0 & 0 \\
              2\gamma^\cC & 0 & \gamma^\cC & 0 & 0 \\
              -\gamma^\Ups & 0 & 0 & -\gamma^\Ups & 0 \\
              0 & \gamma^\cC{-}\gamma^\Ups & 0 & 0 & -\gamma^\Ups
            \end{pmatrix},\qquad
     A_{1 1}=-2\gamma^\Ups, \nonumber \\
  &\ \slA^{\cC\Ups} = (2\pa_1 h_{\bar a\bar b},0,0,0,0), \qquad A_{1 1}^{\cC\Ups} = (0,0,\gamma^\cC,-2\gamma^\Ups,0),\qquad
  \slA_{1 1} = -\half\pa_1 h^{\bar a\bar b}, \nonumber \\
\label{EqNYB}
  &B_g=\begin{pmatrix} 0 & 0 & 0 \\ \slB^{\cC\Ups} & 0 & 0 \\ B_{1 1}^{\cC\Ups} & 0 & 0 \end{pmatrix},\qquad
  \slB^{\cC\Ups} = (2\rho_0^{-1}\pa_1^2 h_{\bar a\bar b},0,0,0,0),\quad
  B_{1 1}^{\cC\Ups} = (2\rho_0^{-1}\pa_1^2 h_{1 1},0,0,0,0).
\end{align}
\end{subequations}

\begin{thm}[Tame estimate with sharp decay]
\label{ThmNY}
  Fix $\ell_0,\ell_{\!\scri}$ as in Definition~\usref{DefNPMetrics}. Let $k,m\in\N_0$ with $k\geq m+11$. Let $g=g_\bhm+r^{-1}h\in\sG^{k,(\ell_0,\ell_{\!\scri})}$, with $h$ small in $\tilde\sG^{8,(\ell_0,\ell_{\!\scri})}$. Consider $f\in\sF^{m+8,(\ell_0,\ell_{\!\scri})}$ (see Definition~\usref{DefNEForc}) which vanishes near $\Sigma$. Then the unique forward solution $u$ of $L_g u=f$ satisfies $u\in\tilde\sG^{m,(\ell_0,\ell_{\!\scri})}$ and a tame estimate
  \begin{equation}
  \label{EqNYEst}
    \|u\|_{\tilde\sG^{m,(\ell_0,\ell_{\!\scri})}} \leq C\Bigl(\|f\|_{\sF^{m+8,(\ell_0,\ell_{\!\scri})}} + \|h\|_{\tilde\sG^{m+11,(\ell_0,\ell_{\!\scri})}}\|f\|_{\sF^{3,(\ell_0,\ell_{\!\scri})}}\Bigr).
  \end{equation}
\end{thm}
\begin{proof}
  For $\alpha_0=\ell_0$ and $\alpha_{\!\scri}\in(-\ell_{\!\scri},0)$, we can apply Proposition~\ref{PropNQ} to obtain $u\in H_{\ebop;\bop}^{(1;m+8),(\ell_0,2\alpha_{\!\scri})}(\Omega)$ satisfying the estimate~\eqref{EqNQTame} with $m+8$ in place of $m$. Write
  \[
    L_g=-2(\rho_{\!\scri}\pa_{\rho_{\!\scri}}-A_g)(\rho_0\pa_{\rho_0}-\rho_{\!\scri}\pa_{\rho_{\!\scri}})+2 B_g+L_g^\flat,\qquad L_g^\flat\in(x_{\!\scri}\CI(\Omega)+\Hb^{k-2,(\ell_0,2\ell_{\!\scri})}(\Omega))\Diffb^2,
  \]
  where spherical derivatives are error terms since we work in the b-setting now. We now use
  \begin{equation}
  \label{EqNYNormalEq}
    2(\rho_{\!\scri}\pa_{\rho_{\!\scri}}-A_g)(\rho_0\pa_{\rho_0}-\rho_{\!\scri}\pa_{\rho_{\!\scri}})u = f + 2 B_g u + L_g^\flat u
  \end{equation}
  repeatedly, together with the spectral information on $A_g$ given in~\eqref{EqNYA}, to prove sharp decay for the various components of $h$ at $\scri^+$.

  \pfstep{First improvement.} Applying $\pi^{\cC\Ups}$ to~\eqref{EqNYNormalEq}, we get
  \begin{align*}
    (\rho_{\!\scri}\pa_{\rho_{\!\scri}}-A^{\cC\Ups})(\rho_0\pa_{\rho_0}-\rho_{\!\scri}\pa_{\rho_{\!\scri}})(\pi^{\cC\Ups} u) &\in \Hb^{m+8,(\ell_0,2\ell_{\!\scri})} + \Hb^{m+6,(\ell_0,2(\alpha_{\!\scri}+\ell_{\!\scri}))} \\
      &\subset \Hb^{m+6,(\ell_0,2(\alpha_{\!\scri}+\ell_{\!\scri}))}.
  \end{align*}
  Definition~\ref{DefNPMetrics} ensures that all eigenvalues of $A^{\cC\Ups}$ are $>\ell_{\!\scri}$. Thus, we get improved decay $(\rho_0\pa_{\rho_0}-\rho_{\!\scri}\pa_{\rho_{\!\scri}})(\pi^{\cC\Ups}u)\in\Hb^{m+6,(\ell_0,2(\alpha_{\!\scri}+\ell_{\!\scri}))}$ at the cost of $2$ b-derivatives. Integrating this from $\Sigma$ (see \cite[Lemma~7.7(1)]{HintzVasyMink4})  and using that $\alpha_{\!\scri}+\ell_{\!\scri}<\ell_{\!\scri}<\ell_0$ gives
  \begin{equation}
  \label{EqNYuCY}
    \pi^{\cC\Ups}u\in\Hb^{m+6,(\ell_0,2(\alpha_{\!\scri}+\ell_{\!\scri}))}.
  \end{equation}

  Applying $\slpi_0$ to~\eqref{EqNYNormalEq} and using~\eqref{EqNYuCY} to estimate the contributions from $\slA^{\cC\Ups}(\rho_0\pa_{\rho_0}-\rho_{\!\scri}\pa_{\rho_{\!\scri}})$ and $\slB^{\cC\Ups}$, we obtain
  \[
    (\rho_0\pa_{\rho_0}-\rho_{\!\scri}\pa_{\rho_{\!\scri}})\rho_{\!\scri}\pa_{\rho_{\!\scri}}(\slpi_0 u) \in \Hb^{m+6,(\ell_0,2(\alpha_{\!\scri}+\ell_{\!\scri}))} + \Hb^{m+5,(1+2\ell_0,2(\alpha_{\!\scri}+\ell_{\!\scri}))} \subset \Hb^{m+5,(\ell_0,2(\alpha_{\!\scri}+\ell_{\!\scri}))}.
  \]
  Integrating $\rho_0\pa_{\rho_0}-\rho_{\!\scri}\pa_{\rho_{\!\scri}}$ gives $\rho_{\!\scri}\pa_{\rho_{\!\scri}}(\slpi_0 u)\in\Hb^{m+5,(\ell_0,2(\alpha_{\!\scri}+\ell_{\!\scri}))}$ and therefore
  \begin{equation}
  \label{EqNYslu}
    \slpi_0 u=\slu^{(0)}+\tilde\slu,\qquad \slu^{(0)}\in\Hb^{m+5,\ell_0}(\scri^+\cap\Omega),\quad \tilde\slu\in\Hb^{m+5,(\ell_0,2(\alpha_{\!\scri}+\ell_{\!\scri}))}.
  \end{equation}

  Lastly, we apply $\pi_{1 1}$ to~\eqref{EqNYNormalEq} and use~\eqref{EqNYuCY}--\eqref{EqNYslu}, and note that $\slpi_0 u$ is coupled to $\pi_{1 1}u$ via $\slA_{1 1}\in\Hb^{k-1,1+\ell_0}(\scri^+\cap\Omega)$ to obtain
  \[
    (\rho_{\!\scri}\pa_{\rho_{\!\scri}}-A_{1 1})(\rho_0\pa_{\rho_0}-\rho_{\!\scri}\pa_{\rho_{\!\scri}})(\pi_{1 1}u) \in \Hb^{m+4,1+2\ell_0}(\scri^+\cap\Omega) + \Hb^{m+4,(\ell_0,2(\alpha_{\!\scri}+\ell_{\!\scri}))}(\Omega).
  \]
  Since $A_{1 1}=-2\gamma^\Ups>\ell_{\!\scri}>\alpha_{\!\scri}+\ell_{\!\scri}$, integration of this implies
  \begin{equation}
  \label{EqNYu11}
    \pi_{1 1}u = u_{1 1}^{(0)} + \tilde u_{1 1},\qquad
    u_{1 1}^{(0)} \in \Hb^{m+4,1+2\ell_0}(\scri^+\cap\Omega),\quad
    \tilde u_{1 1} \in \Hb^{m+4,(\ell_0,2(\alpha_{\!\scri}+\ell_{\!\scri}))}(\Omega).
  \end{equation}

  \pfstep{Second improvement.} We again apply $\pi^{\cC\Ups}$ to~\eqref{EqNYNormalEq}; exploiting the sharper (as far as decay is concerned) information~\eqref{EqNYuCY}--\eqref{EqNYu11}, we now get
  \[
    (\rho_{\!\scri}\pa_{\rho_{\!\scri}}-A^{\cC\Ups})(\rho_0\pa_{\rho_0}-\rho_{\!\scri}\pa_{\rho_{\!\scri}})(\pi^{\cC\Ups}u) \in \Hb^{m+8,(\ell_0,2\ell_{\!\scri})} + \Hb^{m+2,(\ell_0,2\ell_{\!\scri})},
  \]
  with the second term coming from the second order operator $L_g^\flat$ acting on $\slpi_0 u$, $\pi_{1 1}u$. Integrating this gives $\pi^{\cC\Ups}u\in\Hb^{m+2,(\ell_0,2\ell_{\!\scri})}$. For $\slpi_0 u$, this improved information gives
  \[
    \rho_{\!\scri}\pa_{\rho_{\!\scri}}(\rho_0\pa_{\rho_0}-\rho_{\!\scri}\pa_{\rho_{\!\scri}})(\slpi_0 u) \in \Hb^{m+1,(\ell_0,2\ell_{\!\scri})},
  \]
  which implies that $\tilde\slu\in\Hb^{m+1,(\ell_0,2\ell_{\!\scri})}$ in~\eqref{EqNYslu}. This in turn gives
  \[
    (\rho_{\!\scri}\pa_{\rho_{\!\scri}}-A_{1 1})(\rho_0\pa_{\rho_0}-\rho_{\!\scri}\pa_{\rho_{\!\scri}})(\pi_{1 1}u) \in \Hb^{m+4,1+2\ell_0}(\scri^+\cap\Omega) + \Hb^{m,(\ell_0,2\ell_{\!\scri})}(\Omega),
  \]
  and hence $\tilde u_{1 1}\in\Hb^{m,(\ell_0,2\ell_{\!\scri})}(\Omega)$ in~\eqref{EqNYu11}. This demonstrates that $u\in\sG^{m,(\ell_0,\ell_{\!\scri})}$. The tame estimate~\eqref{EqNYEst} follows from that in Proposition~\ref{PropNQ} together with tame estimates for products, as already exploited in the proof of Proposition~\ref{PropNQ}.
\end{proof}

\begin{cor}[Nonlinear stability near the far end]
\label{CorNYStab}
  Let $\Omega$ and $\Sigma$ be as in Definition~\usref{DefNPMetrics} and Lemma~\usref{LemmaNPCausal}, and consider the quasilinear wave operator $P(g)$ from Definition~\usref{DefNEOp}. Let $\ell_0>0$, and let $\ell_{\!\scri}\in(0,\min(\ell_0,\half))$. Suppose $h_0,h_1\in\Hb^{\infty,\ell_0}(\Sigma;S^2\wt T^*M)$; putting $\|(h_0,h_1)\|_m:=\|h_0\|_{\Hb^{m+1,\ell_0}}+\|h_1\|_{\Hb^{m,\ell_0}}$, assume that $(h_0,h_1)$ is small in the sense that $\|(h_0,h_1)\|_{22}<C$ where $C=C(\|(h_0,h_1)\|_{2433})$, with $C=C(q)$ positive and continuous in $q\in[0,\infty)$.\footnote{Thus, there exists $\eps>0$ so that any $(h_0,h_1)$ with $\|(h_0,h_1)\|_{2433}<\eps$ is small in this sense.} Then the initial value problem
  \begin{equation}
  \label{EqNYStabEq}
    P(g_\bhm+r^{-1}h) = 0,\qquad
    (h,\cL_{x_{\!\scri}}h)|_{\Sigma} = (h_0,h_1)
  \end{equation}
  has a unique solution $h\in\tilde\sG^{\infty,(\ell_0,\ell_{\!\scri})}$.\footnote{Recall the causal structure of $\Omega$ recorded in Lemma~\ref{LemmaNMSchwCausal}.}
  
  In particular, if the induced metric and second fundamental form of $g:=g_\bhm+r^{-1}h\in\sG^{\infty,(\ell_0,\ell_{\!\scri})}$ at $\Sigma$ satisfy the constraint equations, and $\Ups_{E^\Ups}(g;g_\bhm)=0$ at $\Sigma$, then $g$ solves the Einstein vacuum equations $\Ric(g)=0$ in the gauge $\Ups_{E^\Ups}(g;g_\bhm)=0$.
\end{cor}

\begin{rmk}[Initial data]
\label{RmkNYData}
  Given geometric initial data (i.e.\ a Riemannian metric and second fundamental form) on $\Sigma$ satisfying the constraint equations, it is easy to construct $h_0,h_1$ so that $g=g_\bhm+r^{-1}h$, with $h$ having initial data $h_0,h_1$ at $\Sigma$, attains these data at $\Sigma$ and satisfies $\Ups_{E^\Ups}(g;g_\bhm)=0$ at $\Sigma$, see e.g.\ \cite[Lemma~6.2]{HintzVasyMink4} (for a slightly different choice of gauge).
\end{rmk}

\begin{proof}[Proof of Corollary~\usref{CorNYStab}]
  While so far we have only discussed forcing problems, our energy estimate based arguments apply to initial value problems as well. Alternatively, one can piece together a short time solution $h_{\rm in}$ on $x_{\!\scri}^{-1}([\half c,c])$, say, with the forward solution of
  \[
    P_{\rm fw}(h) := \rho_{\!\scri}\rho^{-3}P\bigl(g_\bhm+r^{-1}(\chi h_{\rm in}+h)\bigr)=0,
  \]
  where $\chi\in\CIc((\half c,c])$ is $1$ on $[\frac34 c,c]$. Since $P_{\rm fw}(h)\in\Hb^{\infty,(\ell_0,\infty)}(\Omega;S^2\wt T^*M)$ is small in $\Hb^{22,(\ell_0,1)}(\Omega)$ and has support in $x_{\!\scri}^{-1}([\half c,\tfrac34 c])$, Nash--Moser iteration can be applied to the nonlinear map
  \[
    \tilde\sG^{\infty,(\ell_0,\ell_{\!\scri})}\ni h\mapsto P_{\rm fw}(h)\in\sF^{\infty,(\ell_0,\ell_{\!\scri})}
  \]
  in view of Corollary~\ref{CorNENonlin} and Theorem~\ref{ThmNY}, upon restricting to inputs $h$ vanishing on $x_{\!\scri}^{-1}([\half c,c])$. Indeed, applying the main theorem of \cite{SaintRaymondNashMoser} with loss of derivatives parameter $d=11$ (cf.~\eqref{EqNYEst}) produces the solution of~\eqref{EqNYStabEq}; here, $2433=16 d^2+43 d+24$.

  The second part is standard: given a solution $g=g_\bhm+r^{-1}h$ of~\eqref{EqNYStabEq} satisfying the constraint equations and the gauge condition initially, one first concludes that also $\cL_{\pa_{x_{\!\scri}}}\Ups_{E^\Ups}(g;g_\bhm)=0$ at $\Sigma$. The second Bianchi identity implies the homogeneous wave-type equation $2\delta_g\sfG_g\delta_{g,E^\cC}^*(\Ups_{E^\Ups}(g;g_\bhm))=0$ which gives $\Ups_{E^\Ups}(g;g_\bhm)\equiv 0$ and therefore, by definition of $P(g)$, also $\Ric(g)=0$.
\end{proof}

\begin{rmk}[Gravitational radiation and Bondi mass]
\label{RmkNYBondi}
  Given a Ricci-flat metric $g=g_\bhm+r^{-1}h\in\sG^{\infty,(\ell_0,\ell_{\!\scri})}$ in the gauge $\Ups_{E^\Ups}(g;g_\bhm)=0$, one can (with some effort) adapt the arguments in~\cite[\S8]{HintzVasyMink4} to identify the Bondi mass at retarded time $u:=-\rho_0^{-1}=t-r_*$ as
  \[
    M_{\rm B}(u) = m + \frac{\gamma^\Ups}{4\pi} \int_{S(u)} h_{1 1}^{(0)}\,\dd\slg,\qquad
    S(u):=\scri^+\cap\rho_0^{-1}(-1/u),
  \]
  using the notation of Definition~\ref{DefNPMetrics}. By~\eqref{EqNENonlin}, $M_{\rm B}(u)$ satisfies the mass loss formula
  \[
    \frac{\dd}{\dd u}M_{\rm B}(u) = -\frac{1}{32\pi}\int_{S(u)} |\pa_u\slh^{(0)}|_{\slg^{-1}}^2\,\dd\slg.
  \]
\end{rmk}

\begin{rmk}[Polyhomogeneity of the metric]
\label{RmkNYPhg}
  The methods of \cite[\S7]{HintzVasyMink4} apply, mutatis mutandis, to demonstrate the polyhomogeneity of the spacetime metric $g=g_\bhm+r^{-1}h$ on $M$ provided the initial data $h_0,h_1$ are polyhomogeneous. Since the metric perturbation $h$ here has stronger decay at $\scri^+$ compared to the reference, the index sets will be smaller than in \cite[Theorem~7.1]{HintzVasyMink4}.
\end{rmk}

\appendix
\section{Constraint damping and gauge change for the Maxwell equations}
\label{SMax}

As a simpler analogue to the Einstein vacuum equations, we consider the Maxwell equations for a 1-form (gauge potential) $A$ on the domain of dependence $t<r-1$ of the complement of the ball of radius $1$ inside $t^{-1}(0)$ inside Minkowski space $(\R^4,g)$, $g=-\dd x^0\,\dd x^1+r^2\slg$,
\begin{equation}
\label{EqMaxEq}
  \delta_g\dd A = 0.
\end{equation}
A standard way to break the gauge invariance $A\to A+\dd\chi$ is the imposition of the Lorenz gauge $\delta_g A=0$. The most simple-minded gauge-fixed Maxwell equations are then $\delta_g\dd A+\dd\delta_g A=0$; this is the tensor wave equation on 1-forms on $(\R^4,g)$. \emph{Constraint damping} and a \emph{gauge change} amount to modifications in the second, gauge breaking, term: letting
\[
  \dd_{E^\cC} := \dd + 2\gamma^\cC\cd^\cC,\qquad
  \delta_{g,E^\Ups} := \delta_g + 2\gamma^\Ups\iota_{\cd^\Ups}
\]
for $E^\bullet=(\cd^\bullet,\gamma^\bullet)$, $\bullet=\cC,\Ups$, we consider
\begin{equation}
\label{EqMaxFix}
  P'_{E^\cC,E^\Ups}A := \delta_g\dd A + \dd_{E^\cC}\delta_{g,E^\Ups}A = 0.
\end{equation}
On Minkowski space, we concretely take $\cd^\cC=\cd^\Ups=r^{-1}\,\dd t$, $\gamma^\cC>0$, $\gamma^\Ups<0$. Writing 1-forms in the splitting~\eqref{EqNMVbSplit}, we compute on functions, resp.\ 1-forms,
\begin{alignat*}{3}
  \dd&=\begin{pmatrix} \pa_0 \\ \pa_1 \\ r^{-1}\sld \end{pmatrix},&\qquad
  \dd_{E^\cC}-\dd&=r^{-1}\begin{pmatrix}\gamma^\cC \\ \gamma^\cC \\ 0\end{pmatrix}, \\
  \delta_g&=(2 r^{-2}\pa_1 r^2,2 r^{-2}\pa_0 r^2,r^{-1}\sldelta),&\qquad
  \delta_{g,E^\Ups}-\delta_g&=r^{-1}(-2\gamma^\Ups,-2\gamma^\Ups,0).
\end{alignat*}
In the notation of~\S\ref{SsND} and Lemma~\ref{LemmaNMSchwCausal}, recall that $\pa_0,r^{-1}\sld,r^{-1}\sldelta\in x_{\!\scri}\Veb(M)$ decay at $\scri^+$, whereas $\pa_1=\rho_0(\rho_0\pa_{\rho_0}-\rho_{\!\scri}\pa_{\rho_{\!\scri}})$ does not; the analogue of Proposition~\ref{PropNELin} then reads
\begin{align*}
  &L := \rho_{\!\scri}\rho^{-3}P'_{E^\cC,E^\Ups}\rho = L^0+\tilde L, \\
  &\quad L_0=-2(\rho_{\!\scri}\pa_{\rho_{\!\scri}}-S_{E^\cC,E^\Ups})(\rho_0\pa_{\rho_0}-\rho_{\!\scri}\pa_{\rho_{\!\scri}}) + x_{\!\scri}^2\slDelta,\qquad \tilde L\in x_{\!\scri}\Diffeb^2(M;\wt T^*M), \\
  &\quad S_{E^\cC,E^\Ups} = \begin{pmatrix} \gamma^\cC & 0 & 0 \\ \gamma^\cC{-}\gamma^\Ups & -\gamma^\Ups & 0 \\ 0 & 0 & 0 \end{pmatrix}.
\end{align*}
Thus, if $\omega=(\omega_0,\omega_1,\slomega)$ solves $L\omega=0$ with sufficiently decaying initial data, then $\omega_0$ (the Maxwell analogue of $\pi^\cC h$ in Definition~\ref{DefNPMetrics} or $\phi_3$ in the model~\eqref{EqIModel3}) and $\omega_1$ (the Maxwell analogue of $\pi^\Ups h$ or $\phi_4$) decay at $\scri^+$, while $\slomega$ has a leading order term at $\scri^+$. For initial data satisfying the Maxwell constraint equations and the gauge condition $\delta_{g,E^\Ups}\omega=0$ initially, $A:=r^{-1}\omega$ solves~\eqref{EqMaxEq} and globally satisfies the gauge condition (using an argument in which the Bianchi identity $\delta_g\sfG_g\Ric(g)\equiv 0$ is replaced by $\delta_g\delta_g\equiv 0$)
\begin{equation}
\label{EqMaxGauge}
  \delta_{g,E^\Ups}A = 0.
\end{equation}

Now to leading order at $\scri^+$, this gauge condition reads $\pa_1\omega_0=0$ (independently of $\gamma^\Ups$) and thus, by itself, only recovers the improved decay of $\omega_0$ at $\scri^+$. The improvement coming from the gauge change encoded by $E^\Ups$ only arises once one considers~\eqref{EqMaxGauge} \emph{together with the Maxwell equations}~\eqref{EqMaxEq}; on an algebraic level, the gauge condition~\eqref{EqMaxGauge} allows one to exchange occurrences of $\pa_1\omega_0$ in~\eqref{EqMaxEq} (in particular in second order terms $\pa_1^2\omega_0$, which do appear for~\eqref{EqMaxEq} but not for gauge-fixed equations such as~\eqref{EqMaxFix}) by lower order terms in the sense of decay, and one particular such combination is~\eqref{EqMaxFix} which implies the desired improved decay of $\omega_1$ due to the structure of $S_{E^\cC,E^\Ups}$.

We give a more conceptual (but more abstract) reason for the fact that the gauge change improves decay for a component ($\omega_1$) other than $\omega_0$ (which is affected by constraint damping and the accompanying improved decay of the gauge condition), based on duality considerations. Namely, since $\dd_{E^\cC}^*=\delta_{g,E^\cC}$, constraint damping with strength $\gamma^\cC>0$, resp.\ a gauge change with strength $\gamma^\Ups<0$, is dual to a gauge change with strength $\gamma^\cC>0$, resp.\ constraint damping with strength $\gamma^\Ups<0$ (note the `wrong' signs), in the sense that $E^\Ups$ and $E^\cC$ get interchanged when passing from $P'_{E^\cC,E^\Ups}$ in~\eqref{EqMaxFix} to its adjoint $(P'_{E^\cC,E^\Ups})^*=P'_{E^\Ups,E^\cC}$. Taking for simplicity $E^\cC=0$, `negative' constraint damping (encoded by $E^\Ups$ for the adjoint operator) allows one to solve the adjoint (thus, \emph{backwards}) forcing problem $(P'_{0,E^\Ups})^*\tilde u=\tilde f$ for $\tilde u,\tilde f$ having additional terms with \emph{more} growth at $\scri^+$ than without `inverse' constraint damping---concretely, $\tilde f$ may have growing contributions which are sections of the bundle $\la\dd x^0+\dd x^1\ra$ spanned by the nonzero eigenvector of $S_{E^\Ups,0}$, or equivalently $\tilde f$ lies in a space of more growing 1-forms so that $\pi\tilde f$ has standard bounds, with $\pi$ any bundle projection with $\ker\pi=\la\dd x^0+\dd x^1\ra$. Dually, this means that one can solve the forward problem for $P_{0,E^\Ups}\omega=f$ on function spaces encoding extra decay in certain components---concretely, for suitably decaying $f$, the solution $\omega$ is the sum of a 1-form with improved decay and a 1-form valued in $\ran\pi^*=(\ker\pi)^\perp=\la\dd x^0-\dd x^1\ra\oplus r T^*\Sph^2$ with standard $r^{-1}$ decay at $\scri^+$. Since the annihilator of $\ran\pi^*$ is $\pa_0+\pa_1$, this means that $\omega_0+\omega_1$ has improved decay. Combined with constraint damping (which gives the improved decay of $\omega_0$ as discussed after~\eqref{EqMaxGauge}), this finally provides improved decay of $\omega_1$. Analogous remarks apply to the (linearized) gauge-fixed Einstein equations; see Remark~\ref{RmkLDuality}.

\bibliographystyle{alphaurl}


\end{document}